\documentclass[%
 aip,
 amsmath,amssymb,
 reprint,%
]{revtex4-1}

\usepackage{bm}
\usepackage{graphicx}
\usepackage{dcolumn}
\usepackage[utf8]{inputenc}
\usepackage[T1]{fontenc}
\usepackage{color}
\usepackage{amsopn}
\usepackage{amsmath} 

\usepackage{amssymb}
\usepackage{amsthm}
\usepackage{amsfonts}
\usepackage{algorithmic}

\usepackage{cleveref}
\usepackage{algorithmic}

\newtheorem{remark}{Remark}
\newtheorem{definition}{Definition}
\newtheorem{lemma}{Lemma}
\newtheorem{theorem}{Theorem}
\newtheorem{proposition}{Proposition}

\DeclareMathOperator*{\argmin}{arg\,min}
\DeclareMathOperator{\spn}{span}
\DeclareMathOperator{\rank}{rank}
\DeclareMathOperator{\diag}{diag}
\newcommand{\Mod}[1]{\ (\mathrm{mod}\ #1)}

\DeclareFontFamily{U}{mathx}{\hyphenchar\font45}
\DeclareFontShape{U}{mathx}{m}{n}{
      <5> <6> <7> <8> <9> <10>
      <10.95> <12> <14.4> <17.28> <20.74> <24.88>
      mathx10
      }{}
\DeclareSymbolFont{mathx}{U}{mathx}{m}{n}
\DeclareFontSubstitution{U}{mathx}{m}{n}
\DeclareMathAccent{\widebar}{0}{mathx}{"73}

\begin{document}

\preprint{AIP/123-QED}

\title[On the Structure of Time-delay Embedding in  Linear Models of Non-linear Dynamical Systems]{On the Structure of Time-delay Embedding in  Linear Models of Non-linear Dynamical Systems}

\author{Shaowu Pan}
\email{shawnpan@umich.edu.}
\affiliation{Department of Aerospace Engineering, University of Michigan, Ann Arbor, MI 48105, USA}
\author{Karthik Duraisamy}%
\affiliation{ 
Department of Aerospace Engineering, University of Michigan, Ann Arbor, MI 48105, USA
}%

\date{\today}

\begin{abstract}
This work addresses fundamental issues related to the structure and conditioning of linear time-delayed models of non-linear dynamics on an attractor. 
While this approach has been well-studied in the asymptotic sense (e.g. for infinite number of delays), the non-asymptotic setting is not well-understood.
First, we show  that the minimal time-delays required for perfect signal recovery are solely determined by the sparsity in the Fourier spectrum for scalar systems. For the vector case, we provide a  rank test and a geometric interpretation for the necessary and sufficient conditions for the existence of an accurate linear time delayed model. Further, we prove that the output controllability index of a linear system induced by the Fourier spectrum serves as a tight upper bound on the minimal number of time delays required.  An explicit expression for the exact linear model in the spectral domain is also provided. From a numerical perspective, the effect of the sampling rate and the number of time delays on numerical conditioning  is examined. An upper bound on the condition number is derived, with the  implication that  conditioning can be improved with additional time delays and/or decreasing sampling rates. Moreover, it is explicitly shown that the underlying dynamics can be accurately recovered using only a partial period of the attractor.  Our analysis is first validated in simple periodic and quasi-periodic systems, and sensitivity to noise is also investigated.
Finally, issues and practical strategies of choosing time delays in  large-scale chaotic systems are discussed and demonstrated on 3D  turbulent Rayleigh-B\'{e}nard convection. 
\end{abstract}

\maketitle

\begin{quotation}
It is well-known that periodic and quasi-periodic attractors of  a non-linear dynamical system can be reconstructed in a discrete sense using time-delay embedding. Following this argument, it has been shown that even chaotic non-linear systems can be represented as a linear system with intermittent forcing. Although it is known that linear models such as those generated by the Hankel Dynamic Mode Decomposition can - in  principle - reconstruct an ergodic dynamical system in an asymptotic sense,  quantitative details such as the required sampling rate and the number of delays remain unknown. For scalar and vector periodic systems, we derive the minimal necessary time delays and show that time delays not only lead to a more expressive feature space but also result in better numerical conditioning. Further, we explain the reason behind the accurate recovery of   attractor dynamics using only a partial period of data. Finally, we discuss the impact of the number of delays in modeling large-scale chaotic systems, e.g., turbulent Rayleigh-B\'{e}nard convection.
\end{quotation}

\section{Introduction}
Time-delay embedding, also known as delay-coordinate embedding, refers to the inclusion of history information in dynamical system models.  This idea has been employed in a wide variety of contexts including time series modeling~\cite{chen1989representations,hegger1999practical},  Koopman operators~\cite{arbabi2017ergodic,arbabi2017study,kamb2018time,brunton2017chaos} and closure modeling~\cite{pan2018_ddc}. The use of  delays to construct a ``rich" feature space for geometrical reconstruction of non-linear dynamical systems is justified by the Takens embedding theorem~\cite{takens1981detecting} which states that by using a \emph{delay-coordinate map}, one can construct a diffeomorphic shadow manifold from univariate observations of the original system in the generic sense, and its extensions in a measure-theoretic sense~\cite{Sauer1991}, filtered memory~\cite{Sauer1991}, deterministic/stochastic forcing~\cite{Stark2003,Stark2003b},  and multivariate embeddings~\cite{Deyle2011}.

Time delay embedding  naturally arises in the representation of the evolution of partially observed states in dynamical systems. As an illustrative example, consider a $N$-dimensional linear autonomous discrete dynamical system with $Q$ partially observed (or resolved) states, $Q < N$:
\begin{equation}{\label{eq:general_fom_linear}}
\begin{bmatrix}
\mathbf{\hat x}^{n+1}        \\
\mathbf{\tilde x}^{n+1}
\end{bmatrix}
=
\begin{bmatrix}
\mathbf{A}_{11} & \mathbf{A}_{12}\\
\mathbf{A}_{21} & \mathbf{A}_{22}
\end{bmatrix}
\begin{bmatrix}
\textcolor{black}{\mathbf{\hat x}^{n}}         \\
\textcolor{black}{\mathbf{\tilde x}^{n}}
\end{bmatrix},
\end{equation}
where $\mathbf{\hat{x}}^n \in \mathbb{R}^{Q}$, $\mathbf{\tilde{x}}^n \in \mathbb{R}^{N-Q}$, $n \in \mathbb{N}$, $ \mathbf{A}_{11} \in \mathbb{R}^{Q \times Q}$, $\mathbf{A}_{12} \in \mathbb{R}^{Q \times (N-Q)}$, $\mathbf{A}_{21} \in \mathbb{R}^{(N-Q) \times Q}$, $\mathbf{A}_{22} \in \mathbb{R}^{(N-Q) \times (N-Q)}$. The dynamical evolution of the observed states $\mathbf{\hat{x}}$ is given by:
\begin{equation}{\label{eq:linear_gle}}
\mathbf{\hat{x}}^{n+1}
=
\mathbf{A}_{11}\mathbf{\hat x}^{n} +
\sum_{k=0}^{n-1} \mathbf{A}_{12} \mathbf{A}_{22}^{k} \mathbf{A}_{21} \mathbf{\hat x}^{n-1-k} + \mathbf{A}_{12} \mathbf{A}_{22}^n \mathbf{\tilde{x}}^{0}.
\end{equation}
Typically, the last term is of a  transient nature, and thus the above equation can be considered to be \emph{closed} in the observed variables $\bm{\hat{x}}$. The  second term on the right hand side of \Cref{eq:linear_gle} describes how the time-history of the observed  modes affects the dynamics.  Thus, \Cref{eq:linear_gle} implies that it is  possible to extract the dynamics of the observables  $\mathbf{\hat{x}}$ using time delayed observables, i.e., $\mathbf{\hat{x}}^{n+1} = \mathbf{C}_0 \mathbf{\hat{x}}^{n} +  \sum_{k=1}^{L}\mathbf{C}_k \mathbf{\hat{x}}^{n-k}$, where $\mathbf{C}_k \in \mathbb{R}^{Q\times Q}$, and $L$ is the number of time delays.  It should, however, be noted that \emph{explicit} delays might not be  necessary if one has access to high order time derivatives~\cite{takens1981detecting} or abundant distinct observations~\cite{Deyle2011}. 


Leveraging delay coordinates  to construct predictive models of dynamical systems has been a topic of great interest. As an example, such models have been studied extensively in the time series analysis community via the well-known family of autoregressive and moving average (ARMA) models~\cite{box2015time}. In the machine learning community, related ideas are used in feedforward neural networks (FNN) that augment input dimensions with time delays~\cite{frank2001time}, time-delay neural networks (TDNN)~\cite{lang1990time,peddinti2015time,bromley1994signature} that statically perform convolutions in time, and the family of recurrent neural networks (RNN)~\cite{goodfellow2016deep} that dynamically perform non-linear convolutions in time~\cite{ma2018model}. In a dynamical systems context, time delays are leveraged in higher order or Hankel Dynamic Mode Decomposition ~\cite{le2017higher,arbabi2017ergodic,brunton2017chaos}. Although in essence, each community relies on approximations with time-delays, the focus is typically on different aspects: the time series community focuses on stochastic problems, and prefer explicit and interpretable models~\cite{box2015time}; the machine learning community is typically more performance-driven and focuses on minimizing the error and scalability~\cite{peddinti2015time}; the dynamical systems community is focused on the regulated, continuous dynamical system and interpretability of temporal behavior in terms of eigenvalues and eigenvectors~\cite{kaiser2018sparse}. Moreover, the scientific computing community emphasizes very high dimensional settings, as exemplified by fluid dynamics.

A relevant and outstanding question in each of the aforementioned contexts is the following: \emph{Given time series data from a non-linear dynamical system, how much memory is required to accurately recover the underlying dynamics, given a model structure?} The memory can be characterized by the two hyperparameters, namely \emph{the number of time delays} and the \emph{corresponding data sampling intervals, if uniformly sampled}. Takens embedding theorem~\cite{takens1981detecting} proved the generic existence of a time delayed system with $L=\lceil{2n_{box}}\rceil$ delays, where ($n_{box}$ is box counting dimension of the attractor, given the model has enough \emph{non-linearity} to approximate the diffeomorphism. However, the question of how to determine the number of time delays and sampling rate  is not well-addressed. Given $n_{box}$ as  the box counting dimension of the attractor, the number of required time delays  $L_{takens} = \lceil{2n_{box}}\rceil$  is rather conservative~\cite{gilmore2003topology}. For example, it is both well known in practice and shown analytically~\cite{pan2018_ddc}, that a typical chaotic Lorenz attractor with box counting dimension $\approx 2.06$~\cite{mcguinness1983fractal} can be well embedded with $L=2$, i.e., an equivalent 3D time delay system, while $L=4$ is required from Takens embedding theorem.


However, other than acknowledging a diffeomorphism, the Takens embedding theorem does not posit any constraints on the mapping from time delay coordinates to the original system state. Clearly, the required number of time delays depends on the richness (non-linearity) of the embedding. 
In general, for nonlinear models, the determination of the time delays becomes a problem of phase-space reconstruction~\cite{frank2001time,abarbanel1993analysis}. Popular methods include the false nearest neighbor method~\cite{kennel1992determining}, singular value analysis~\cite{broomhead1989time}, averaged mutual information~\cite{sugihara1990g}, saturation of system invariants~\cite{abarbanel1993analysis}, box counting methods~\cite{sauer1993many}, correlation integrals~\cite{kim1999nonlinear}, standard model selection techniques~\cite{cao1997practical}, and even reinforcement learning~\cite{liu2007rldde}. On the other hand, for linear models, criteria based on statistical significance such as the model utility F-test~\cite{lomax2013statistical} or information theoretic techniques such as AIC/BIC~\cite{box2015time} are used. The use of the partial autocorrelation in linear autoregressive (AR) models to  determine the number of delays can be categorized as a model selection approach. It should be mentioned that by treating the models as a black-box, a general approach such as cross validation can be leveraged. 

When the sampling rate is fixed, the question of the number of time delays required should not be confused with the length of statistical dependency between the present  and past states on the trajectory. For example, an AR(2) model can have a long time statistical dependency, but the number of time delays in the model may be very small. Indeed, it has been explicitly shown~\cite{pan2018_ddc} that for a non-linear dynamical system with dual linear structure,  embedding the memory in a dynamic fashion requires a much smaller number of delays compared to a prescribed static model structure~\cite{gouasmi2017priori}. 

From the viewpoint of discovering the dynamics of a partially observed system, the goal is to determine the non-linear convolution operator~\cite{hald,gouasmi2017priori} or the so-called closure dynamics~\cite{pan2018_ddc}. 
It has to be recognized that the number of time delays will also be dependent on the specific structure of the model. The interchangeability between the number of distinct observables and the number of time delays is also reflected in Takens' original work on the embedding theorem~\cite{takens1981detecting}. Such interchangeability with the latent space dimension is also explored in closure dynamics~\cite{pan2018_ddc,gouasmi2017priori,parish2018residual} and recurrent neural networks~\cite{goodfellow2016deep}. Since the required number of delays is strongly dependent on the model structure, it is prudent to first narrow down to a specific type of model, and then determine the delays needed. 

\textcolor{black}{The connection between time delay embedding and the Koopman operator is elucidated by Brunton et al.~\cite{brunton2017chaos}. Further theoretical investigations were conducted by Arbabi and Mezi\'c~\cite{arbabi2017ergodic}.} 
For an ergodic dynamical system, assuming that the observable belongs to a finite-dimensional Koopman invariant subspace $\mathcal{H}$, they showed that Hankel-DMD, a linear model \textcolor{black}{(first proposed and connected to ERA~\cite{juang1985eigensystem}/SSA~\cite{vautard1992singular} by Tu et al.~\cite{tu2013dynamic})}, can provide an exact representation of the Koopman eigenvalues and eigenfunctions in $\mathcal{H}$. This pioneering work, together with several  numerical investigations on the application of Hankel-DMD to non-linear dynamical systems~\cite{champion2018discovery,le2017higher,brunton2017chaos} and theoretical studies on time-delayed observables using singular value decomposition (SVD)~\cite{kamb2018time} highlight the ability of linear time delayed models to represent non-linear dynamics. From a heuristic viewpoint, SVD has been demonstrated~\cite{broomhead1989time,broomhead1986extracting,gibson1992analytic} to serve as a practical guide to determine the required number of time delays and sampling rate, for linear models. 

It should be noted that much of the literature~\cite{tu2013dynamic,schmid2010dynamic,brunton2013compressive} related to DMD and Hankel-DMD consider SVD projection either in the time delayed dimension (e.g. singular spectrum analysis) or the state dimension. SVD can   provide   optimal linear coordinates to maximize signal-to-noise ratio~\cite{gibson1992analytic}, and thus promote robustness and efficiency.  On the other hand, projection via Fourier transformation enables the possibility of additional theoretical analysis. For instance, Fourier-based analysis of the Navier--Stokes equations include non-linear triadic wave interactions~\cite{pope_2000} and decomposition into solenoidal and dilatational components~\cite{pan2017role}. Pertinent to the present work, ergodic systems characterized by periodic or quasi-periodic attractors have been shown to be well approximated by Fourier analysis~\cite{schilder2006fourier,rowley2009spectral,mezic2005spectral}. Fourier analysis has also been employed to approximate  the transfer function to obtain an intermediate discrete-time reduced order model with stability guarantees for very large scale linear systems~\cite{willcox2005fourier,gugercin2008krylov}.
For general phase space reconstruction, asymptotic decay rates from Fourier analysis have been leveraged to infer appropriate sampling intervals and number of delays~\cite{lipton1996reconstructing}. 
We thus leverage a Fourier basis representation to uncover the structure of time delay embeddings in  \emph{linear models} of non-linear dynamical systems. We also address related issues of numerical conditioning. It should be emphasized that this work is purely concerned with deterministic linear models and noise free data. It can also be shown that SVD becomes equivalent to Fourier analysis in the limit of large windows~\cite{gibson1992analytic}. 

The manuscript is organized as follows: The problem formulation and model structure is presented in \Cref{sec:linear_model_desc}. Following this, the Fourier transformation of the problem and main theoretical results regarding  the minimal time delay embedding for both scalar  and vector time series together with  explicit, exact solutions of the delay transition matrix after Fourier transformation are presented in
\Cref{sec:theory,sec:vec_case}. 
Modal decompositions related to the Koopman operator is described in \Cref{sec:connect_koopman}.
Numerical implementation and theoretical results related to conditioning issues is presented and verified numerically in \Cref{sec:experiments}, while applications on several nonlinear dynamical systems are displayed in \Cref{sec:application_nonlinear_system}.
The main contributions of the work are summarized in
\Cref{sec:conclusions}.

\section{Linear model with time-delay embedding}
\label{sec:linear_model_desc}

Consider a continuous autonomous dynamical system,
\begin{equation}
    \label{eq:dynamics}
    \frac{d}{dt} \mathbf{x} = \mathbf{F}(\mathbf{x}(t)), 
\end{equation}
on a state space $\mathcal{M} \subset \mathbb{R}^J$, $J \in \mathbb{N}^+$, where $\mathbf{x}$ is the coordinate vector of the state, $\mathbf{x} \in \mathcal{M}$, $\mathbf{F}(\cdot): \mathcal{M} \mapsto \mathbb{R}^{J}$ is in $C^{\infty}$. Denote $\phi_t(\mathbf{x}_0)$, i.e., the flow generated by \Cref{eq:dynamics} as the state at time $t$ of the dynamical system that is initialized as  $\mathbf{x}(0) = \mathbf{x}_0 \in \mathcal{M}$. By uniformly sampling with time interval $\Delta t$, the trajectory data of the dynamical systems can be obtained as $\{\mathbf{x}_j\}_{j=0}^{\infty}$, where $\mathbf{x}_j \triangleq \mathbf{x}(j\Delta t), j \in \mathbb{N}$. 

The aforementioned linear model with time-delay embedding order $L$ \emph{assumes} that the predicted future state $\mathbf{\hat{x}}_{j+1}$ is a sum of $L+1$ linear mappings from the present state $\mathbf{x}_{j}$ and previous $L$ states $\{\mathbf{x}_{j-l}\}_{l=1}^{L}$, $j \in \mathbb{N}$, 
\begin{equation}
\label{eq:linear_ass}
\mathbf{\hat{x}}_{j+1} = \mathbf{W}_0 \mathbf{x}_j + \mathbf{W}_1 \mathbf{x}_{j-1} + \ldots + \mathbf{W}_{L} \mathbf{x}_{j-L},
\end{equation}
where   $\mathbf{W}_l \in \mathbb{R}^{J \times J}$ is the associated weight matrix for the $l$-th time-delay snapshot, $l=0,\ldots,L$. 
As a side note, many data-driven models such as \textcolor{black}{ERA}, AR, VAR~\cite{box2015time}, SSA~\cite{vautard1992singular}, HAVOK~\cite{brunton2017chaos}, Hankel-DMD~\cite{arbabi2017ergodic} or HODMD~\cite{le2017higher}, can be derived from the above setup by \textcolor{black}{leveraging impulse response data,} introducing stochasticity, analyzing the eigenspectrum on the principal components, or adding intermittent forcing as inputs.

Given $M$ snapshots, the goal is to determine the weight matrices that result in the best possible  approximation $\mathbf{\hat{x}}_{j+1}$ to the true future state $\mathbf{x}_{j+1}$ in  a priori $L_2$ sense, i.e.,
\begin{widetext}
\begin{equation}
\label{eq:argmin_W}
\mathbf{W}_0,\ldots,\mathbf{W}_L = \argmin_{\substack{{\{\mathcal{W}_i\}_{i=0}^{L} \in \mathbb{R}^{J\times J}}}} 
\Bigg\Vert
\begin{bmatrix}
\mathcal{W}_L & \ldots & \mathcal{W}_0
\end{bmatrix} 
\begin{bmatrix}
\mathbf{x}_0 & \ldots & \mathbf{x}_{M-2-L}  \\
\vdots       & \vdots & \vdots  \\
\mathbf{x}_L & \ldots & \mathbf{x}_{M-2}
\end{bmatrix}
 - 
\begin{bmatrix}
\mathbf{x}_{L+1} & \ldots & \mathbf{x}_{M-1} 
\end{bmatrix} 
 \Bigg\Vert_F,
\end{equation}
\end{widetext}
if the minimizer is unique. Otherwise, 
\begin{align}
\label{eq:argmin_W_minimal}
& \mathbf{W}_0, \ldots, \mathbf{W}_L = \argmin_{\substack{{\mathcal{W}_0, \ldots, \mathcal{W}_L \in \mathbb{R}^{J\times J}}}} 
\lVert
\begin{bmatrix}
\mathcal{W}_L & \ldots & \mathcal{W}_0
\end{bmatrix} 
 \rVert_F, \\
 & \nonumber \textrm{ subject to } \\
 & \nonumber
\begin{bmatrix}
\mathcal{W}_L & \ldots & \mathcal{W}_0
\end{bmatrix}
\begin{bmatrix}
\mathbf{x}_0 & \ldots & \mathbf{x}_{M-2-L}  \\
\vdots       & \vdots & \vdots  \\
\mathbf{x}_L & \ldots & \mathbf{x}_{M-2}
\end{bmatrix}
 =
\begin{bmatrix}
\mathbf{x}_{L+1}  & \ldots & \mathbf{x}_{M-1} 
\end{bmatrix}.
\end{align}
The analytical solution of the above optimization in \Cref{eq:argmin_W,eq:argmin_W_minimal} is simply the pseudoinverse with SVD~\cite{schmid2010dynamic}, with trunctation for robustness. However,  straightforward SVD computation of the $L$ time-delay matrix for large-scale dynamical systems, e.g., fluid flows $J\sim O(10^6)$ with $L\sim O(10^2)$, is challenging. It is therefore prudent to perform spatial truncation using the SVD computed from $\{\mathbf{x}_j\}_{j=0}^{M-1}$ that reduces the dimension from $J$ to $r$ ($r \ll J$ and $r \le \min(J, M)$) and then perform the above optimizations with $L$ time-delays on the $r$-dimensional system~\cite{le2017higher}.




\subsection{Illustrative example and simplified consideration for analysis}
\label{sec:motivation}

Consider a scalar non-linear periodic trajectory, 
\begin{equation}
x(t) = \cos(t) \sin(\cos(t)) + \cos(t/5),
\end{equation}
where $t \in [0, 40]$. \Cref{fig:motivate_quasi} shows the result of a posteriori prediction using a linear model with $L=1$ and $L=12$ trained only on $t \in [0, 6]$ with 60 uniform samples. Considering that training data in the above example only covers  $[0.6, 1.8]$, the prediction of the trajectory over $[-0.9, 1.8]$,  maybe somewhat surprising.  \textcolor{black}{Although the increased expressiveness with time delay embedding have been reported~\cite{kutz2016dynamic,le2017higher}, reported investigations of the ability of temporal extrapolation are mostly empirical~\cite{le2017higherextra,beltran2018temporal}.}  Note that popular non-linear models, e.g., neural network-based models~\cite{pan2018long,pan2020physics}, despite their property of universal approximation \footnote{This problem can be viewed as an example of \emph{no free lunch} theorem~\cite{wolpert1997no}}, are  trustworthy only within the range of training data. In the present context, this means they are only suitable when training data approximately covers the whole data distribution. 

\begin{figure}[htbp]
  \centering
    \includegraphics[width=1\linewidth]{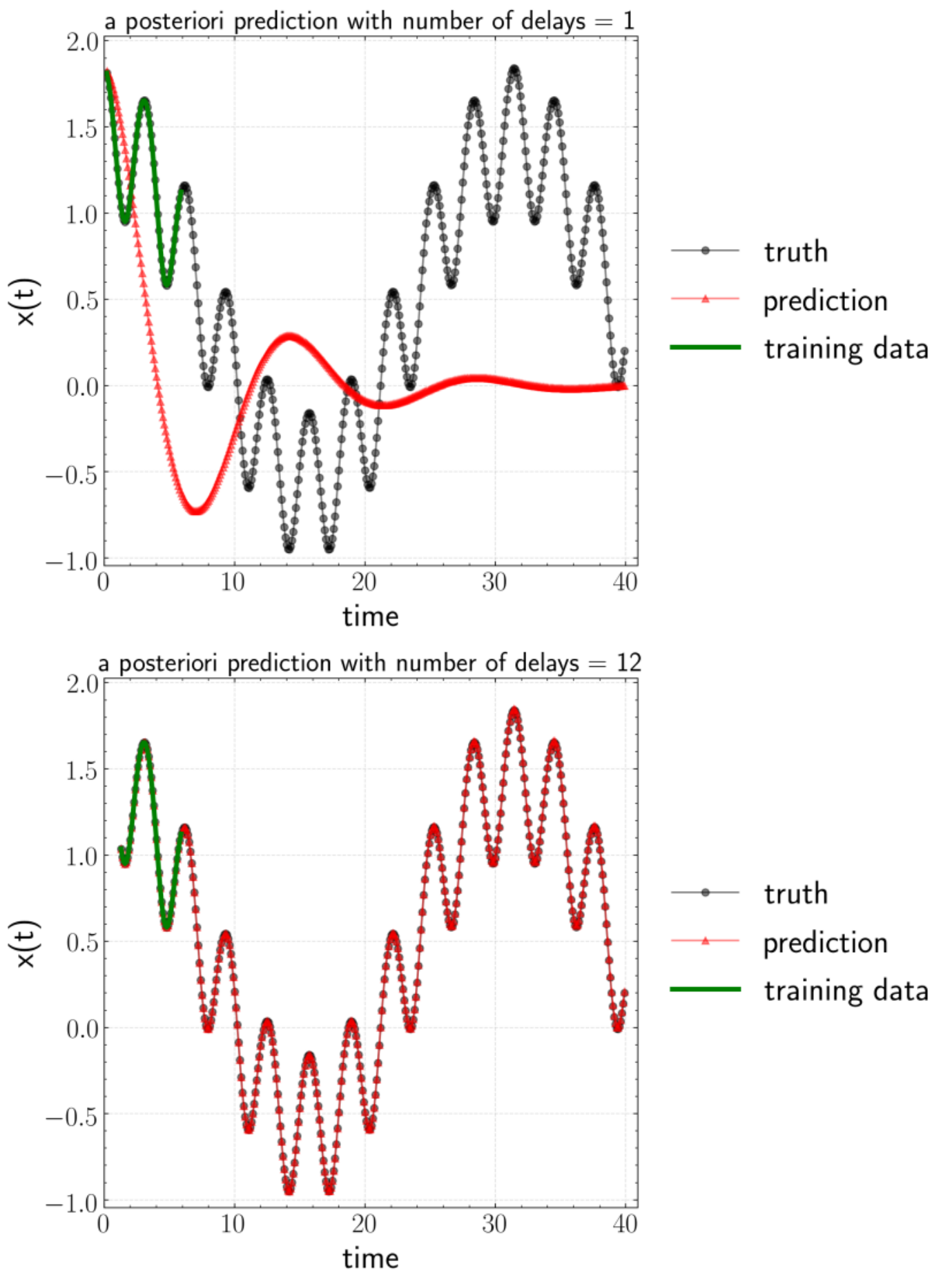}
  \caption{A posteriori prediction on non-linear periodic system with limited training horizon. Top: $L=1$. Bottom: $L=12$.}
  \label{fig:motivate_quasi}
\end{figure}

To provide  insight into role of time-delays, we consider the following simplification for the ease of analysis: we restrict ourselves to the dynamics on a periodic attractor, for which  one can determine an arbitrarily close Fourier interpolation in time at a  uniform sampling rate~\cite{attinger1966use}. 
In addition, without loss of generality, we assume that the data has zero mean, i.e., $\int_{\mathbb{R}^{+}} \mathbf{x}(\tau) d\tau = \mathbf{0}$. We start with the scalar case,  and extend the corresponding results to the vector case  $\mathbf{x} \in \mathbb{R}^J$ in \Cref{sec:vec_case}. Note that the data is collected by uniformly sampling a $T$-periodic time series $x(t) \in \mathbb{R}$. The number of samples per period is $M$, with uniform sampling interval $\Delta t = T/M$. Without loss of generality, we assume that sampling is initiated at $t = 0$, $x_k = x(t_k)$, $t_k = k \Delta t$, $k \in \mathcal{I}_M$, $\mathcal{I}_M = \{  0,1,\ldots, M-1 \}$, and $T$ is the smallest positive real number that represents the periodicity.

\subsection{Projection of the trajectory on a Fourier basis}
\label{sec:fourier_proj}
With the simplifications in \Cref{sec:motivation}, we consider a surrogate signal of $x(t)$: $S_M(t)$

\begin{equation}
\label{eq:sm}
S_M (t) = \sum_{i\in \mathcal{I}_M} a_i e^{-j\frac{2\pi i t}{T}} \ \ \textrm{with}  \ \
a_i = \frac{1}{M} \sum_{k\in \mathcal{I}_M} x_k e^{j \frac{2\pi k i}{M}} \in \mathbb{C},
\end{equation}
where $j =\sqrt{-1}$ and  
\begin{equation}
\label{eq:x_k_s_m}
\forall k \in \mathcal{I}_M, \quad x_k = x(k\Delta t) = S_M(k\Delta t),
\end{equation}

which is obtained by projecting  $\mathbf{x}(t)$ on the following linear space $\mathcal{H}_F$ 
\begin{equation}
\label{eq:fb}
    \mathcal{H}_{F} = \spn \{1, e^{ -j \frac{2\pi t}{T}}, \ldots,  e^{ -j \frac{2\pi (M-1) t}{T}}\},
\end{equation}
which is spanned by the Fourier basis in \Cref{eq:fb} with test functions as delta functions as $\delta(t - t_k), k\in\mathcal{I}_M$.
This process is equivalent to the discrete Fourier transform (DFT).

The above procedure naturally represents the uniformly sampled trajectory in the time domain $\{ x_k\}_{k=0}^{M-1}$ using coefficients in the frequency domain $\{ a_i \}_{i=0}^{M-1}$.
Since we consider real signals, $\{a_i\}_{i=0}^{M-1}$ possess reflective symmetry:
$\forall i \in \mathcal{I}_M$, $\textrm{Re}(a_i) = \textrm{Re}(a_{M-i})$, $\textrm{Im}(a_i) + \textrm{Im}(a_{M-i})=0$,
where $\textrm{Re}$ and $\textrm{Im}$ represent the real and imaginary part of a complex number. 
In addition, since $T$ is the smallest period by definition, we must have $a_1 = \overline{a_{M-1}} \neq 0$.  
Further,
since $\mathbf{F}$ is smooth, the flow $\phi_t(\mathbf{x}_0) = \mathbf{x}(t)$ is also smooth in $t$~\cite{nijmeijer1990nonlinear}. Thus, the error in the Fourier interpolation is uniformly bounded by twice  the sum of the absolute value of truncated Fourier coefficients~\cite{boyd2001chebyshev}. This leads to the uniform convergence
\begin{equation}
\label{eq:dft}
\lim_{M\xrightarrow{}\infty} \vert x(t) - S_M (t) \vert = 0.
\end{equation}
Hence, one can easily approximate the original periodic trajectory uniformly to the desired level of accuracy by increasing $M$ above a certain threshold.


 




\section{The structure of time delay embedding for scalar  time series}
\label{sec:theory}

Now, we apply the linear model with time-delay embedding (\Cref{eq:linear_ass})  at the locations  $\{x_k\}_{k=0}^{M-1}$. 
Given $\{ x_k\}_{k=0}^{M-1}$, consider constructing  $L$-time delays of $x(t)$, $L\in \mathbb{N}$. Note that $L=0$ corresponds to no   delays considered. To avoid negative indices,  we utilize the modulo operation defined in \Cref{eq:mod}, 
\begin{equation}
\label{eq:mod}
\textrm{$\forall q \in \mathbb{N}$, }
    \mathcal{P}(q) \triangleq q \Mod{M} = 
\begin{cases}
    q,& \text{if } q \in \mathcal{I}_M, \\
    q - M \left \lfloor q/M \right\rfloor,              & \text{otherwise}
\end{cases}
\end{equation}
to construct the $L$ time-delay vector $\mathbf{Y}_k$,  
\begin{equation}
\label{eq:Y_k}
\mathbf{Y}_k = 
\begin{bmatrix}
x_{\mathcal{P}(k)} \\
x_{\mathcal{P}(k-1)} \\
\vdots \\
x_{\mathcal{P}(k-L)}
\end{bmatrix} \in \mathbb{R}^{L+1},
\end{equation}
where $k \in \mathcal{I}_M$, $\left\lfloor \cdot \right\rfloor$ is the floor function. 
Considering Fourier interpolation,  we have 
\begin{equation}
  \forall q \in \mathcal{I}_M, \ \ \ \   x_{\mathcal{P}(q)} = \sum_{i\in \mathcal{I}_M} a_i \omega^{qi}, \quad \omega \triangleq e^{-j\frac{2\pi}{M}} \in \mathbb{C},
\end{equation}
which is also true for $q \not\in \mathcal{I}_M$ 
\begin{align}
\nonumber
x_{\mathcal{P}(q)} &= S_M((q - M\left\lfloor q/M\right\rfloor)\Delta t) = \sum_{i\in \mathcal{I}_M} a_i e^{-j\frac{2\pi i (q - M \left\lfloor q/M \right\rfloor)}{M}} \\
&= \sum_{i\in \mathcal{I}_M} a_i \omega^{qi}.    
\end{align}

Using \Cref{eq:sm}, we can rewrite the $L$ time-delay vector $\mathbf{Y}_k$ in \Cref{eq:Y_k} in the Fourier basis as
\begin{equation}
\label{eq:Y_k_stack}
\mathbf{Y}_k = 
\mathbf{\Omega}_{k,L}
\mathbf{a},
\end{equation}
where $\forall k \in \mathcal{I}_M$, $\mathbf{\Omega}_{k,L} \triangleq \begin{bmatrix}
1 & \omega^k & \omega^{2k} & \ldots & \omega^{(M-1)k} \\
\vdots & \vdots & \vdots & \ddots & \vdots \\
1 & \omega^{k-L} & \omega^{2(k-L)} & \ldots & \omega^{(M-1)(k-L)}
\end{bmatrix}$, $\mathbf{a} \triangleq \begin{bmatrix}
a_0 \\
\vdots\\
a_{M-1}
\end{bmatrix}\in \mathbb{C}^{M \times 1}$.

The problem of the minimal time delay required for the linear model with $L$ time  delays  in \Cref{eq:linear_ass} to \emph{perfectly} predict the data $\{x_k\}_{k=0}^{M-1}$ is equivalent to the existence of the {\em delay transition matrix} $\mathbf{K}$ such that, 
\begin{equation}
\label{eq:X_K_Y}
x_{\mathcal{P}(k+1)} = \mathbf{K}^\top \mathbf{Y}_{k},  \ \ \forall{k} \in \mathcal{I}_M,
\end{equation}
where \[\mathbf{K} = 
\begin{bmatrix}
K_0 & K_1 & \ldots & K_{L}
\end{bmatrix}^\top \in \mathbb{R}^{(L+1) \times 1},
\] and 
\begin{equation}
\label{eq:x_k_p_1}
x_{\mathcal{P}(k+1)} = 
\mathbf{\Upsilon}_k^\top
\mathbf{a},
\end{equation}
where 
\begin{equation}
\label{eq:gamma_k}
    \mathbf{\Upsilon}_k \triangleq \begin{bmatrix}
1 & \omega^{k+1} & \omega^{2(k+1)} & \ldots & \omega^{(M-1)(k+1)} \\
\end{bmatrix}^\top.
\end{equation}

For convenience, we vertically stack \Cref{eq:X_K_Y}  $\forall k \in \mathcal{I}_M$, 
\begin{equation}
    \label{eq:X_K_Y_stacked}
    \mathbf{Y}_M
    \mathbf{K} =
    \mathbf{x}_M,
\end{equation}
where
$
\mathbf{Y}_M \triangleq
    \begin{bmatrix}
        \mathbf{Y}_0^\top \\
        \mathbf{Y}_1^\top \\
        \vdots \\
        \mathbf{Y}_{M-2}^\top \\
        \mathbf{Y}_{M-1}^\top
    \end{bmatrix},
$
$
    \mathbf{x}_M \triangleq 
    \begin{bmatrix}
        x_1 \\ 
        x_2 \\ 
        \vdots \\ 
        x_{M-1} \\ 
        x_{0}
    \end{bmatrix}.
$

In the following subsections, we  discuss the minimal number of required time delays, the exact solution of $\mathbf{K}$ and the number of samples required on the time domain.

\subsection{Minimal number of time delays}

Our goal is to determine the minimal number of time delays $L$, such that there exists a  matrix $\mathbf{K}$ that satisfies the linear system \Cref{eq:X_K_Y}. Given one period of data, we can transform the system from the time domain  to the spectral domain. 
Consider \Cref{eq:Y_k_stack,eq:x_k_p_1}, then \Cref{eq:X_K_Y_stacked} is equivalent to the following,  $\forall{k} \in \mathcal{I}_M$:
\begin{equation}
\label{eq:decompose_fourier_space}
\mathbf{a}^\top 
\left(\begin{bmatrix}
1 \\ 
\omega^{k+1} \\ 
\omega^{2(k+1)} \\
\vdots \\ 
\omega^{(k+1)(M-1)}
\end{bmatrix}
-
\begin{bmatrix}
1 & \ldots  & 1\\ 
\omega^k & \ldots & \omega^{k-L}\\ 
\omega^{2k} & \ldots & \omega^{2(k-L)}\\ 
\vdots & \ddots & \vdots \\ 
\omega^{(M-1)k} & \ldots &  \omega^{(M-1)(k-L)}
\end{bmatrix} 
\mathbf{K} \right) = 0.
\end{equation}
This can be written as
\begin{widetext}
\begin{equation}
\label{eq:general_physics_to_freq}
 \mathbf{a}^\top 
\left(
\begin{bmatrix}
1 & & & & \\ 
& \omega & &  & \\
& & \omega^{2} & &\\
& & & \ddots &  \\
& & &  & \omega^{(M-1)}
\end{bmatrix}^k
\left(
\begin{bmatrix}
1 \\ 
\omega \\ 
\omega^{2} \\
\vdots \\ 
\omega^{M-1}
\end{bmatrix}  \\
- 
\begin{bmatrix}
1 & \ldots  & 1\\ 
1 & \ldots & \omega^{-L}\\ 
1 & \ldots & \omega^{2(-L)}\\ 
\vdots & \ddots & \vdots \\ 
1 & \ldots &  \omega^{(M-1)(-L)}
\end{bmatrix} \mathbf{K} \right)
 \right) = 0.
\end{equation}
\end{widetext}
We define the residual matrix $\mathbf{R}$ as, 
\begin{equation}
\mathbf{R}  \triangleq
\begin{bmatrix}
1 \\ 
\omega \\ 
\omega^{2} \\
\vdots \\ 
\omega^{M-1}
\end{bmatrix}
- 
\begin{bmatrix}
1 & 1 & \ldots  & 1\\ 
1 & \omega^{-1} & \ldots & \omega^{-L}\\ 
1 & \omega^{-2}& \ldots & \omega^{2(-L)}\\ 
\vdots & \vdots & \ddots & \vdots \\ 
1 &  \omega^{-(M-1)} & \ldots &  \omega^{(M-1)(-L)}
\end{bmatrix} \mathbf{K}.
\end{equation}

Given one period of data, we vertically stack the above equation for each $k \in \mathcal{I}_M$. Recognizing the non-singular nature of a Vandermonde square matrix with distinct nodes, we have
\begin{equation}
\label{eq:general_linear_system}
\begin{bmatrix}
a_0 & a_1 & a_2 & \ldots & a_{M-1} \\
a_0 & \omega a_1 & \omega^2 a_2 & \ldots & \omega^{M-1} a_{M-1} \\
a_0 & \omega^2 a_1 & \omega^4 a_2 & \ldots & \omega^{2(M-1)} a_{M-1} \\
\vdots & \vdots & \vdots & \ddots & \vdots\\
a_0 & \omega^{M-1} a_1 & \omega^{2(M-1)} a_2 & \ldots & \omega^{(M-1)(M-1)} a_{M-1} 
\end{bmatrix}
\mathbf{R} = \bm{0}.    
\end{equation}
This gives
\begin{align}
\label{eq:above_residual_R}
\begin{bmatrix}
1 & 1  & \ldots & 1 \\
1 & \omega   & \ldots & \omega^{M-1}  \\
\vdots &  \vdots & \ddots & \vdots\\
1 & \omega^{M-1} & \ldots & \omega^{(M-1)(M-1)} 
\end{bmatrix}
\begin{bmatrix}
a_0 & & &  & \\
& a_1 & & & \\
& & & \ddots & \\
& & & & a_{M-1}
\end{bmatrix}
 \mathbf{R}  =  \bm{0},
\end{align}
and thus
\begin{equation}
\label{eq:residual_R}
\begin{bmatrix}
a_0 & & &  & \\
& a_1 & & & \\
& & & \ddots & \\
& & & & a_{M-1}
\end{bmatrix}
\mathbf{R} = \bm{0}.
\end{equation}

Note the equivalence between \Cref{eq:residual_R} and \Cref{eq:X_K_Y_stacked}. 
Now, we consider the case when the Fourier spectrum is sparse with $P$ non-zero coefficients, $P \in \mathbb{N}$ and $P \le M$.  Moreover, it is consistent with the finite point spectral resolution of Koopman operator appears in the laminar unsteady flows~\cite{mezic2013analysis}.
Denote the set of wave numbers associated with non-zero coefficients as,   \begin{equation}
    {\mathcal{I}}^P_{M} \triangleq \{ a_i \neq 0 \vert i \in \mathcal{I}_M \} = \{ i_p \}_{p=0}^{P-1},
\end{equation}
with ascending order 
$ 0 \le i_{0} < i_{1} < \ldots < i_{P-1} \le M-1$, 
where 
$\vert \hat{\mathcal{I}}^P_{M} \vert = P \in \mathbb{N}$.  Note that there is a reflective symmetry restriction on the Fourier spectrum. 

The feasibility of using the number of time delays $L$ to ensure the existence of a \emph{real} solution $\mathbf{K}$ for the linear system is equivalent to the existence of the linear system $\mathbf{R} = \mathbf{0}$ after removing the rows that correspond to zero Fourier modes in $\mathbf{R}$, denoted as $\mathbf{R}_{\mathcal{I}_{M}^P}$, 
\begin{equation}
\label{eq:linear_system_sparse}
\mathbf{R}_{\mathcal{I}_{M}^P} = \mathbf{0} \iff \mathbf{A}_{{\mathcal{I}}^P_{M},L} \mathbf{K} = \mathbf{b}_{{\mathcal{I}}^P_{M}},
\end{equation}
where 
\begin{equation}
\mathbf{A}_{{\mathcal{I}}^P_{M},L} =   
\begin{bmatrix}
1 & \omega^{-i_0} & \ldots  & \omega^{-Li_0}\\ 
1 & \omega^{-i_1} & \ldots  & \omega^{-Li_1}\\ 
1 & \omega^{-i_2} & \ldots  & \omega^{-Li_2}\\ 
\vdots & \vdots & \ddots & \vdots \\ 
1 & \omega^{-i_{P-1}} & \ldots  & \omega^{-Li_{P-1}}\\ 
\end{bmatrix} \in \mathbb{C}^{P\times (L+1)}, 
\end{equation}
and
\begin{equation}
\mathbf{b}_{{\mathcal{I}}^P_{M}} = \begin{bmatrix}
\omega^{i_0} \\ 
\omega^{i_1} \\ 
\omega^{i_2} \\
\vdots \\ 
\omega^{i_{P-1}}
\end{bmatrix} \in \mathbb{C}^{P\times 1}.
\end{equation}

Before presenting the main theorem~\Cref{thm:sparse_time_delay}, we define the Vandermonde matrix in \Cref{def:vandermonde} and introduce  \Cref{lem:vdm_rank} and \Cref{lem:complex_solution}.
\begin{definition} 
\label{def:vandermonde}
Vandermonde matrix with nodes as $ \alpha_0, \alpha_1, \ldots, \alpha_{M-1} \in \mathbb{C}$ of order $N$ is defined as, \[\mathbf{V}_N(\alpha_0, \alpha_1, \ldots, \alpha_{M-1}) \triangleq 
    \begin{bmatrix}
    1 & \alpha_0 & \ldots  & \alpha_0^{N-1}\\ 
    1 & \alpha_1 & \ldots & \alpha_1^{N-1}\\ 
    \vdots & \vdots & \ddots & \vdots \\ 
    1 &  \alpha_{M-1} & \ldots &  \alpha_{M-1}^{N-1}
    \end{bmatrix}.\]
\end{definition}
\begin{lemma}
\label{lem:vdm_rank}
$\forall M,N \in \mathbb{N}$, the Vandermonde matrix $\mathbf{A} = \mathbf{V}_N(\alpha_0, \alpha_1, \ldots, \alpha_{M-1})$ constructed from distinct $\{\alpha_i\}_{i\in \mathcal{I}_M}, \alpha_i \in \mathbb{C}$, has the two properties,
\begin{enumerate}
    \item $\rank(\mathbf{A}) = \min(M,N)$,
    \item if $\mathbf{A}$ has full column rank, $\forall Q \in \mathbb{N}, Q \le M$, the rank of the submatrix $\mathbf{A^\prime}$ by arbitrarily selecting $Q$ rows is $\min(Q,N)$.
\end{enumerate}
\end{lemma}
\begin{proof}
See \Cref{apdx:vdm_rank_proof}.
\end{proof}
\begin{lemma}
\label{lem:complex_solution}
$\forall m,n \in \mathbb{N}, \mathbf{A} \in \mathbb{R}^{m\times n}, \mathbf{b} \in \mathbb{R}^{m \times 1}$, 
$\exists \mathbf{x} \in \mathbb{C}^{n \times 1}$ s.t. $\mathbf{Ax=b} \iff \exists \mathbf{x}^\prime \in \mathbb{R}^{n \times 1}$ s.t. $\mathbf{Ax^\prime =b}$. Further, when the solution is unique, the above still holds and the solution is real.
\end{lemma}

\begin{proof}
See \Cref{apdx:complex_solution_proof}.
\end{proof}

\begin{theorem}
{\label{thm:sparse_time_delay}}
For a uniform sampling of $S_M(t)$ with length $M$ and $P$ non-zero coefficients in the Fourier spectrum, the minimal number of time delays $L$ for a perfect prediction, i.e., one that satisfies
\Cref{eq:X_K_Y_stacked} is $P-1$. Moreover, when $L=P-1$, the solution is unique.
\end{theorem}

\begin{proof}
See \Cref{apdx:sparse_time_delay_thm_proof}.
\end{proof}

From the above \Cref{thm:sparse_time_delay}, we can easily derive \Cref{prop:single,prop:full} that are intuitive. 
\begin{proposition}
\label{prop:single}
If there is only one frequency in the Fourier spectrum of $S_M(t)$, simply one time delay in the linear model is enough to perfectly recover the signal.
\end{proposition}
\begin{proposition}
\label{prop:full}
If the Fourier spectrum of $S_M(t)$ is dense, then the maximum number of time delays, i.e., over the whole period $M-1$ is necessary to perfectly recover the signal.
\end{proposition}

In retrospect, the result of the minimal number of time delays for a scalar time series is rather intuitive:
any scalar signal with $R$ frequencies corresponds to a certain observable of a $2R$-dimensional linear system. Since more time delays in linear model increases the number of eigenvalues in the corresponding linear system, one requires a minimum of $L=2R-1=P-1$ to match the number of eigenvalues.

\subsection{Exact solution for the delay transition matrix $\mathbf{K}$}
Two interesting facts have to be brought to the fore: 
\begin{enumerate}
    \item From \Cref{eq:linear_system_sparse}, it is clear that $\mathbf{K}$ is independent of the \emph{quantitative value} of the Fourier coefficients, but only depends on the \emph{pattern in the Fourier spectrum}. 
    \item For $L = P-1$,  $\mathbf{A}_{\mathcal{I}_M^P,L}$ is an invertible Vandermonde matrix, which implies the uniqueness of the solution $\mathbf{K}$.
\end{enumerate}

Consider the general explicit formula for the inverse of a Vandermonde matrix~\cite{petersen2008matrix}. Note that $\mathbf{A}_{\mathcal{I}_M^P,P-1} = \mathbf{V}_P(\omega^{-i_0}, \ldots, \omega^{-i_{P-1}})$. 

Thus 
\begin{widetext}
\begin{align}
    \mathbf{A}^{-1}_{\mathcal{I}_M^P,P-1} &= \mathbf{V}_P^{-1}(\omega^{-i_0}, \ldots, \omega^{-i_{P-1}}). \\
    \nonumber
    \mathbf{V}_P^{-1}(\omega^{-i_0}, \ldots, \omega^{-i_{P-1}})_{mn}
    & = (-1)^{m+1} \frac{\displaystyle \sum_{\substack{ 0 \le k_1 < \ldots < k_{P-m} \le P-1 \\  k_1,\ldots,k_{P-m} \neq n-1 }} \omega^{-(i_{k_1} + \ldots + i_{k_{P-m}}) }}{\displaystyle \prod_{0 \le l \le P-1, l\neq n-1} \omega^{-i_{l}} - \omega^{-i_{n-1} } } .\\
    \label{eq:exact_spectral_solution}
    \mathbf{K}_{m} 
    &= \mathbf{V}_P^{-1}(\omega^{-i_0}, \ldots, \omega^{-i_{P-1}})_{mn} \mathbf{b}_{\mathcal{I}_M^P,L,n} \\
    \nonumber
    &= \sum_{n=1}^{P} (-1)^{m+1} \frac{\displaystyle \sum_{\substack{ 0 \le k_1 < \ldots < k_{P-m} \le P-1 \\  k_1,\ldots,k_{P-m} \neq n-1 }} \omega^{-(i_{k_1} + \ldots + i_{k_{P-m}}) }}{\displaystyle \prod_{0 \le l \le P-1, l\neq n-1} \omega^{-i_{l}} - \omega^{-i_{n-1} } } \omega^{i_{n-1}}\\
    \nonumber
    &= \sum_{n=1}^{P} (-1)^{m+1} \frac{\displaystyle \sum_{\substack{ 0 \le k_1 < \ldots < k_{P-m} \le P-1}} e^{\frac{j2\pi(i_{k_1} + \ldots + i_{k_{P-m}})}{M} }}{\displaystyle \prod_{0 \le l \le P-1, l\neq n-1} e^{\frac{j2\pi{i_l}}{M}} - e^{\frac{j2\pi i_{n-1}}{M}} }. 
\end{align}
\end{widetext}
where $ 1 \le m,n \le P$ and $\mathbf{K}_{m} \equiv K_{m-1}$.

Despite the explicit form, the above expression is  not useful in practice. Without loss of generality, considering $P$ is even, the computational complexity at least grows as $ \binom{P}{P/2}$. As an example, for a moderate system with 50 non-sparse modes, $\binom{50}{25} \approx 1.2 \times 10^{14}$.

\subsection{Eigenstructure of the companion matrix}

The eigenstructure of the companion matrix formed with time delays is closely related to the Koopman eigenvalues and eigenfunctions under ergodicity assumptions~\cite{arbabi2017ergodic}. From the viewpoint of HAVOK~\cite{brunton2017chaos}, for a general time delay $L$, the corresponding Koopman eigenvalues are eigenvalues of the companion matrix $\mathbf{K}_{comp}$ defined as $\mathbf{Y}^\top_{k+1}  = \mathbf{Y}^\top_k \mathbf{K}_{comp}$, where
\begin{equation}
    \label{eq:augmented_scalar_Yk_Ykp1}
    \mathbf{K}_{comp} = \begin{bmatrix}
    K_0 & 1 & 0 & \ldots & 0 \\
    K_1 & 0 & 1 & \ldots & 0 \\
    \vdots & \vdots & \vdots & \ddots & \vdots \\
    K_{L-1} & 0 & 0 & \ldots & 1 \\
    K_{L} & 0 & 0 & \ldots & 0
    \end{bmatrix} \in \mathbb{R}^{(L+1) \times (L+1)}.
\end{equation}
The corresponding eigenvalues satisfy $\det(\lambda \mathbf{I} - \mathbf{K}_{comp}) =0$, i.e., $
    \lambda^{L+1} - K_0 \lambda^{L} - \ldots - K_{L} = 0.$
The corresponding eigenstructure  is fully determined by the eigenvalues~\cite{drmavc2018data}, $\lambda_0,\ldots,\lambda_{L}$, i.e., $
    \mathbf{K}_{comp} = \mathbf{Q}^{-1} \mathbf{\Lambda} \mathbf{Q}$,
where $\mathbf{\Lambda} = \diag(\lambda_0,\ldots,\lambda_{L})$, $\mathbf{Q} = \mathbf{V}_{L+1}(\lambda_0,\ldots,\lambda_{L})$.


\subsubsection{Special case: dense Fourier spectrum}

Note that $\omega^{-M} = 1$ and $P=M$. Consider $L=P-1=M-1$, so that the last column of $A_{\mathcal{I}^P_M,L}$ becomes 
\begin{equation}
    \begin{bmatrix}
    1 \\
     \omega^{-(M-1)}\\
     \omega^{-2(M-1)}\\
     \vdots \\
     \omega^{-(M-1)(M-1))}
    \end{bmatrix} = 
    \begin{bmatrix}
    1 \\
     \omega\\
     \omega^{2}\\
     \vdots \\
     \omega^{M-1}
    \end{bmatrix} = \mathbf{b}_{\mathcal{I}^M_M}.
\end{equation}
Therefore, the unique solution can be found from observations as  
\begin{equation}
\label{eq:solution_dense}
    \mathbf{K} = \begin{bmatrix}
    0 & \ldots & 0 & 1
    \end{bmatrix}^\top.
\end{equation}
The companion matrix~\cite{arbabi2017ergodic} associated with the Koopman operator is in the form of a special circulant matrix~\cite{meyer2000matrix}, for which analytical eigenvalues and eigenvectors can be easily determined. In \Cref{eq:augmented_scalar_Yk_Ykp1}, we have
\begin{equation}
\mathbf{K}_{comp} = 
    \begin{bmatrix}
    0 & 1 & 0 & \ldots & 0 \\
    0 & 0 & 1 & \ldots & 0 \\
    \vdots & \vdots & \vdots & \ddots & \vdots \\
    0 & 0 & 0 & \ldots & 1 \\
    1 & 0 & 0 & \ldots & 0
    \end{bmatrix} \in \mathbb{R}^{M \times M},
\end{equation}
which has eigenvalues evenly distributed on the unit circle
\begin{equation}
\forall i \in \mathcal{I}_M, \quad \lambda_i = e^{-j\frac{2\pi i}{M}} = \omega^i,    
\end{equation}
and normalized eigenvectors as
\begin{equation}
    \nu_i = \frac{1}{\sqrt{M}} 
    \begin{bmatrix}
    1 & \omega^{-i} & \omega^{-2i} & \ldots & \omega^{-(M-1)i}
    \end{bmatrix}^\top.
\end{equation}

\subsection{Analysis in the time domain}

Projection of the trajectory onto a Fourier basis implies that  at least one period of training data has to be obtained to be able to construct a linear system that has a unique solution corresponding to $\mathbf{K}^*$. However, we will show that in the time domain, a full period of data is not necessary to determine the solution $\mathbf{K}^*$ if the Fourier spectrum is sparse. 

Denote the number of non-zero Fourier coefficients as $P \in \mathbb{N}$, and its index set as $\mathcal{I}_M^P$ as before. Instead of having a full period of data, without loss of generality, we consider $L$ time delays and select the $Q$ rows in \Cref{eq:X_K_Y_stacked}, for which the index is denoted as $ 0 \le k_0 < \ldots < k_{Q-1} \le M-1$, and $Q \in \mathbb{N}, L + Q \le M$. Therefore, we have the following equation in the time domain,
\begin{equation}
    \label{eq:X_K_Y_stacked_sparse}
    \begin{bmatrix}
        \mathbf{Y}_{k_0}^\top \\
        \mathbf{Y}_{k_1}^\top \\
        \vdots \\
        \mathbf{Y}_{k_{Q-2}}^\top \\
        \mathbf{Y}_{k_{Q-1}}^\top
    \end{bmatrix}
    \mathbf{K} = 
    \begin{bmatrix}
        x_{\mathcal{P}(k_0+1)} \\ 
        x_{\mathcal{P}(k_1+1)} \\ 
        \vdots \\ 
        x_{\mathcal{P}(k_{Q-2}+1)} \\ 
        x_{\mathcal{P}(k_{Q-1}+1)}
    \end{bmatrix}.
\end{equation}
Consider a Fourier transform and 
recall \Cref{eq:general_physics_to_freq}.  Choosing $k$ over $k_0,\ldots, k_{Q-1}$,  the above equation can be equivalently rewritten as
\begin{align}
    \begin{bmatrix}
    a_0 & \omega^{k_0} a_1 & \omega^{2k_0} a_2 & \ldots & \omega^{(M-1)k_0} a_{M-1} \\
    a_0 & \omega^{k_1} a_1 & \omega^{2k_1} a_2 & \ldots & \omega^{(M-1)k_1} a_{M-1} \\
    a_0 & \omega^{k_2} a_1 & \omega^{2k_2} a_2 & \ldots & \omega^{(M-1)k_2} a_{M-1} \\
    \vdots & \vdots & \vdots & \ddots & \vdots\\
    a_0 & \omega^{k_{Q-1}} a_1 & \omega^{2k_{Q-1}} a_2 & \ldots & \omega^{(M-1)k_{Q-1}} a_{M-1}
    \end{bmatrix} \mathbf{R}
    &=  \bm{0}.
\end{align}
Recall that only $P$ Fourier coefficients are non-zero, and thus the above equation that constrains $\mathbf{K}$ equivalently becomes 
\begin{widetext}
\begin{align}
    \label{eq:physics_to_fourier}
    & \begin{bmatrix}
    a_{i_0} & \omega^{k_0} a_{i_1} & \omega^{2k_0} a_{i_2} & \ldots & \omega^{(P-1)k_0} a_{i_{P-1}} \\
    a_{i_0} & \omega^{k_1} a_{i_1} & \omega^{2k_1} a_{i_2} & \ldots & \omega^{(P-1)k_1} a_{i_{P-1}} \\
    a_{i_0} & \omega^{k_2} a_{i_1} & \omega^{2k_2} a_{i_2} & \ldots & \omega^{(P-1)k_2} a_{i_{P-1}} \\
    \vdots & \vdots & \vdots & \vdots & \vdots\\
    a_{i_0} & \omega^{k_{Q-1}} a_{i_1} & \omega^{2k_{Q-1}} a_{i_2} & \ldots & \omega^{(P-1)k_{Q-1}} a_{i_{P-1}}
    \end{bmatrix}
    \mathbf{R}_{\mathcal{I}_{M}^P}
    =  \bm{0} \\ & \iff 
    \begin{bmatrix}
    1 & \omega^{k_0} & \omega^{2k_0} & \ldots & \omega^{(P-1)k_0} \\
    1 & \omega^{k_1} & \omega^{2k_1} & \ldots & \omega^{(P-1)k_1}  \\
    1 & \omega^{k_2} & \omega^{2k_2} & \ldots & \omega^{(P-1)k_2}  \\
    \vdots & \vdots & \vdots & \ddots & \vdots\\
    1 & \omega^{k_{Q-1}} & \omega^{2k_{Q-1}} & \ldots & \omega^{(P-1)k_{Q-1}}
    \end{bmatrix}
    \begin{bmatrix}
    a_{i_0} & & &  & \\
    & a_{i_1} & & & \\
    & & a_{i_2} & & \\
    & & & \ddots & \\
    & & & & a_{i_{P-1}}
    \end{bmatrix}
    \mathbf{R}_{\mathcal{I}_{M}^P}
    \nonumber =  \bm{0} \\
    \label{eq:subsample_linear}
    & \iff 
    \mathbf{V}_{P}( \omega^{k_0},\ldots,\omega^{k_{Q-1}} ) \diag (a_{i_0},\ldots,a_{i_{P-1}}) \mathbf{R}_{\mathcal{I}_{M}^P} = \mathbf{0}. 
\end{align}
\end{widetext}
Since $\{ \omega^{k_j} \}_{j=0}^{Q-1} $ are distinct from each other, from \Cref{lem:vdm_rank}, $\textrm{rank}(\mathbf{V}_{P}(\omega^{k_0},\ldots, \omega^{k_{Q-1}}) ) = \min(P, Q)$. Therefore, if we choose to \emph{have training data points no less than the number of non-zero Fourier coefficients}, i.e., $Q \ge P$, then $\mathbf{V}_{P}(\omega^{k_0},\ldots, \omega^{k_{Q-1}})$ is full rank, which leads to $\mathbf{R}_{\mathcal{I}_{M}^P} = \mathbf{0}$. Meanwhile, the solution $\mathbf{K}$ is uniquely determined given $L = P-1$. Therefore, given $Q \ge P$, 
\begin{align}
    \label{eq:X_K_Y_stacked_Q}
    \begin{bmatrix}
        \mathbf{Y}_{k_0}^\top \\
        \mathbf{Y}_{k_1}^\top \\
        \vdots \\
        \mathbf{Y}_{k_{Q-2}}^\top \\
        \mathbf{Y}_{k_{Q-1}}^\top
    \end{bmatrix}
    \mathbf{K} = 
    \begin{bmatrix}
        x_{\mathcal{P}(k_0+1)} \\ 
        x_{\mathcal{P}(k_1+1)} \\ 
        \vdots \\ 
        x_{\mathcal{P}(k_{Q-2}+1)} \\ 
        x_{\mathcal{P}(k_{Q-1}+1)}
    \end{bmatrix}
    & \iff 
    \mathbf{R}_{\mathcal{I}_{M}^P} = \mathbf{0}  {\overset{L=P-1}{\iff}} \textrm{$\mathbf{K} = \mathbf{K}^*$}
\end{align}

For the case with minimal number of data samples, i.e., $Q=P$, a natural choice is to construct $P$ rows of the future state from the $P$-th to $2P-1$-th rows in \Cref{eq:X_K_Y_stacked}. In the above setting, in order to construct the linear system in time domain that has the unique solution $\mathbf{K}^*$ of \Cref{eq:linear_system_sparse}, we only require access to the first $2P$ snapshots of data. 
The key observation is that when the signal is sparse, instead of constructing the classic unitary DFT matrix (\Cref{eq:above_residual_R} to \Cref{eq:residual_R}), a random choice of $P$ rows will be sufficient to uniquely determine a real solution $\mathbf{K}^*$. It has to be  mentioned, however, that randomly chosen data points might not be optimal. For example, in \Cref{eq:physics_to_fourier}, the particular choice of sampling (i.e. the choice of $Q$ rows),  will determine the condition number of the complex Vandermonde matrix $\mathbf{V}_{P}(\omega^{k_0},\ldots, \omega^{k_{Q-1}})$. The  necessary and sufficient condition for \emph{perfect} conditioning of a Vandermonde matrix is when $\{ \omega^{k_j} \}_{j=0}^{Q-1}$ are uniformly spread on the unit circle~\cite{berman2007perfect}.


{\color{red} 
}

At first glance, our work might appear to be in the same vein as  compressed sensing (CS)~\cite{donoho2006compressed,candes2006near} where a complete signal is extracted from only a few measurements. However, it should be emphasized that CS requires \emph{random} projections from the whole field to extract information about a broadband signal in each measurement, while we simply follow the setup in modeling dynamical systems where only \emph{deterministic} and sequential point measurements are available, and limited to a certain time interval.

Moreover, the above instance of accurately recovering the dynamical system without using a full period of data on the attractor is also reported elsewhere, for instance in sparse polynomial regression for data-driven modeling of dynamical systems~\cite{champion2018discovery}. Indeed, this is one of the key ideas behind SINDy~\cite{brunton2016discovering}: one can leverage the prior knowledge of the existence of a sparse representation (for instance, in a basis of monomials), such that sparse regression can significantly reduce the amount of data required with no loss of information. 

\section{Extension of the analysis to the vector case}
\label{sec:vec_case}
In this section, we extend the above analysis to the case of a vector dynamical system. Assuming the state  vector has $J$ components, given the time series of $l$-th component, $\{x^{(l)}_k\}_{k=0}^{M-1}$, $l =1,\ldots,J$, we have, $\forall k \in \mathcal{I}_M$
\begin{equation}
    \label{eq:Y_stacked_vector}
    \tilde{x}_{\mathcal{P}(k+1)} = 
    \begin{bmatrix}
    x_{\mathcal{P}(k+1)}^{(1)}\\
    \vdots \\
    x_{\mathcal{P}(k+1)}^{(J)}
    \end{bmatrix} \in \mathbb{R}^{J \times 1},
\end{equation}
where $k\in \mathcal{I}_M, \forall 1 \le l \le J, l\in\mathbb{N}, x^{(l)}_{\mathcal{P}(k)} \in \mathbb{R}$, $J \in \mathbb{N}$. Rewrite \Cref{eq:X_K_Y} in a vector form: 
\begin{equation}
\label{eq:X_K_Y_vector}
\tilde{x}_{\mathcal{P}(k+1)} = \mathbf{\tilde{K}}^\top \mathbf{\tilde{Y}}_{k},  \ \ \forall{k} \in \mathcal{I}_M,
\end{equation}
where $\tilde{x}_{\mathcal{P}(k+1)} \in \mathbb{R}^{J}$, $\mathbf{\tilde{K}} \in \mathbb{R}^{J(L+1) \times J}$ and
\begin{equation}
\label{eq:X_stacked_vector}
\mathbf{\tilde{Y}}_k = 
\begin{bmatrix}
\mathbf{Y}^{(1)}_k \\
\vdots \\
\mathbf{Y}^{(J)}_k
\end{bmatrix} \in \mathbb{R}^{J(L+1) \times 1},
\end{equation}
where $\mathbf{Y}^{(l)}_k$ are the $L$ time-delay embeddings defined in \Cref{eq:Y_k} for the $l$-th component of the state. In the present work, we treat the time-delay uniformly across all components. 

Following  similar procedures as before, denoting the Fourier coefficient of $l$-th component as $\mathbf{a}^{(l)} \in \mathbb{C}^{M\times 1}$, the following lemma which is an analogy to \Cref{eq:residual_R} in the scalar case. 
\begin{lemma}
\label{lem:vector_residual_R}
The necessary and sufficient condition for the existence of a real solution $\mathbf{\tilde{K}}$ in \Cref{eq:X_K_Y_vector} is equivalent to the existence of a solution for the following linear system: 
\begin{align}
    \label{eq:lem_vector_residual_R}
    & 
    \begin{bmatrix}
    \diag(\mathbf{a}^{(1)})  & 
    \ldots & 
    \diag(\mathbf{a}^{(J)})
    \end{bmatrix}
    \Bigg(
    \begin{bmatrix}
    \mathbf{b}_{\mathcal{I}_M^M} & &  \\
    & \ddots & \\
    & & \mathbf{b}_{\mathcal{I}_M^M}
    \end{bmatrix}
    \nonumber
    - \\
    & \begin{bmatrix}
    \mathbf{A}_{\mathcal{I}_M^M,L} & &  \\
    & \ddots & \\
    & & \mathbf{A}_{\mathcal{I}_M^M,L} 
    \end{bmatrix}
     \mathbf{\tilde{K}}
    \Bigg )
    = \mathbf{0}. 
\end{align}
The existence of the above solution is equivalent to the following relationship, 
    \begin{widetext}
    \begin{align}
        \label{eq:vec_lem_rank}
        & \rank \left(
        \begin{bmatrix}
        \diag(\mathbf{a}^{(1)})\mathbf{A}_{\mathcal{I}_M^M,L}   & 
        \ldots & 
        \diag(\mathbf{a}^{(J)})\mathbf{A}_{\mathcal{I}_M^M,L}
        \end{bmatrix}
        \right)\nonumber  \\
         & = 
        \rank \left(
        \begin{bmatrix}
        \diag(\mathbf{a}^{(1)})\mathbf{A}_{\mathcal{I}_M^M,L}   & 
        \ldots & 
        \diag(\mathbf{a}^{(J)})\mathbf{A}_{\mathcal{I}_M^M,L} & 
        \diag(\mathbf{a}^{(1)})\mathbf{b}_{\mathcal{I}_M^M}    & 
        \ldots & 
        \diag(\mathbf{a}^{(J)})\mathbf{b}_{\mathcal{I}_M^M} 
        \end{bmatrix}
        \right).
    \end{align}
    \end{widetext}
\end{lemma}

\begin{proof}
See \Cref{apdx:vector_proof}.
\end{proof}

Next, with the introduction of the  Krylov subspace in \Cref{def:krylov_subs} \textcolor{black}{which frequently appears in the early literatures of DMD~\cite{rowley2009spectral,schmid2010dynamic}}, we present \Cref{rem:geo} and \Cref{rem:intp} from  \Cref{eq:lem_vector_residual_R} that interprets and reveals the possibility of using \emph{less} embeddings than the corresponding  sufficient condition for the scalar case in \Cref{thm:sparse_time_delay}.
\begin{definition}[Krylov subspace]
\label{def:krylov_subs}
For $n,r \in \mathbb{N}$, $\mathbf{A} \in \mathbb{C}^{n \times n}$, $\mathbf{b} \in \mathbb{C}^{n \times 1}$, Krylov subspace is defined as
\begin{equation}
    \mathcal{K}_{r}(\mathbf{A},\mathbf{b}) = \spn \{ \mathbf{b}, \mathbf{Ab}, \ldots, \mathbf{A}^{r-1} \mathbf{b}\}.
\end{equation}
\end{definition}

\begin{figure}[htbp]
  \centering
    \includegraphics[width=1\linewidth]{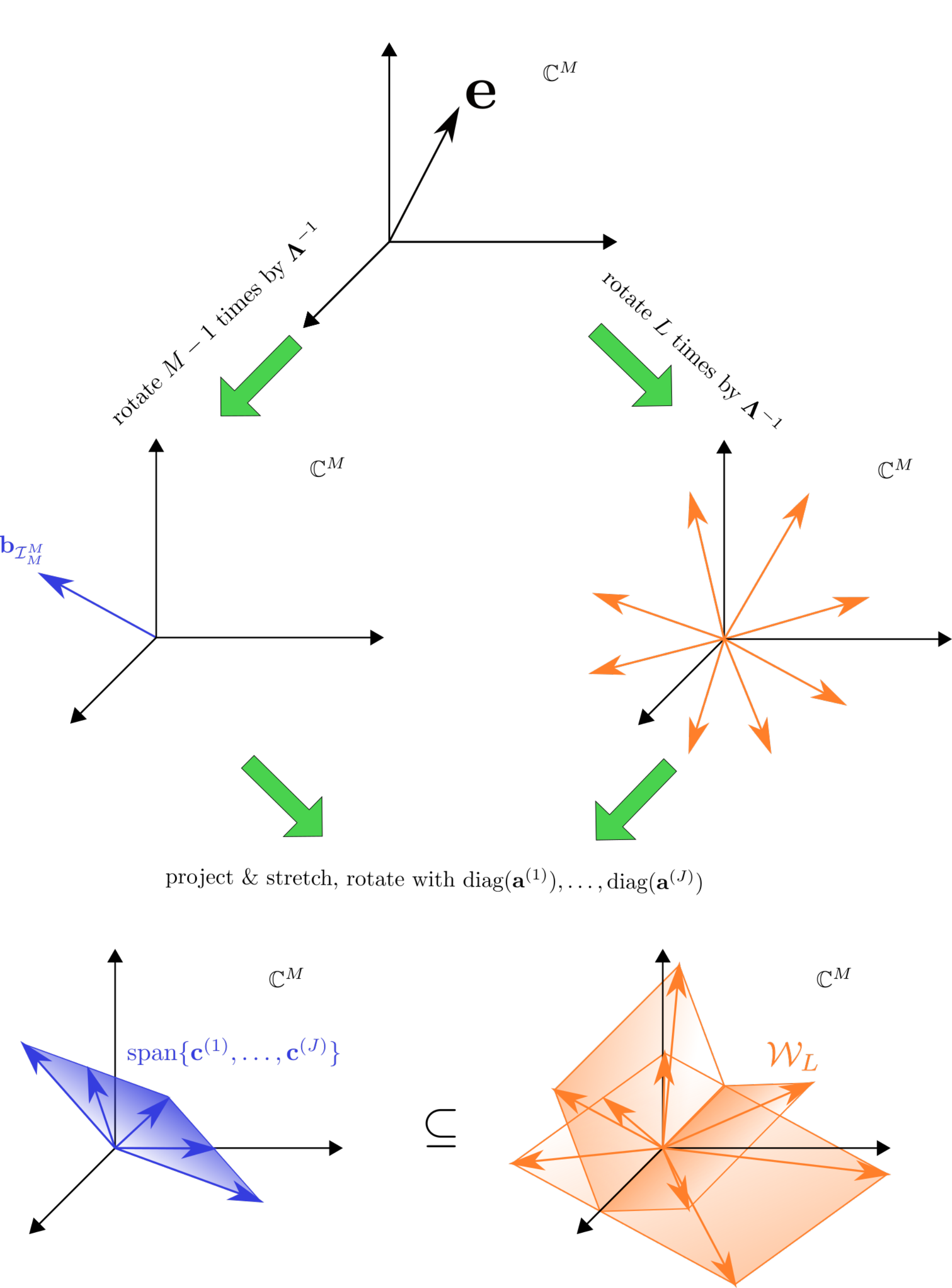}
  \caption{Illustration of the geometrical interpretation of \Cref{lem:vector_residual_R}.}
  \label{fig:geom_meaning}
\end{figure}

\begin{remark}[Geometric interpretation]
\label{rem:geo}
For $j=1,\ldots,J$, define $\mathbf{c}^{(j)} \triangleq \diag(\mathbf{a}^{(j)}) \mathbf{b}_{\mathcal{I}_M^M}$, and $\mathcal{E}_L^{(j)}$ as the column space of  $ \diag(\mathbf{a}^{(j)})\mathbf{A}_{\mathcal{I}_M^M,L}$. The  existence of the solution in \Cref{eq:lem_vector_residual_R} is then equivalent to
\begin{align}
& \forall j \in \{ 1,\ldots,J\}, 
    \nonumber
    \mathbf{c}^{(j)} \in \mathcal{W}_L \triangleq \mathcal{E}_L^{(1)}  \oplus \ldots \oplus \mathcal{E}_L^{(J)} \\
     \label{eq:geometric}
    & \iff \spn{\{\mathbf{c}^{(1)},\ldots, \mathbf{c}^{(J)}\}} \subseteq \mathcal{W}_L,
\end{align}
where $\mathcal{W}_L$ is the column space from all components, and $\oplus$ is the direct sum operation between vector spaces. 
Note that the column space of $\mathbf{A}_{\mathcal{I}_M^M,L}$ can represented as a Krylov subspace  $\mathcal{K}_{L+1}(\mathbf{\Lambda}^{-1},\mathbf{e} )$, where 
\begin{equation}
    \label{eq:e}
    \mathbf{e} \triangleq \begin{bmatrix} 1 & \ldots & 1 \end{bmatrix}^\top,
\end{equation}
\begin{equation}
    \label{eq:lambda}
    \mathbf{\Lambda} \triangleq \diag(\omega^0,\ldots,\omega^{M-1}).
\end{equation}

A geometric interpretation of the above expressions is shown in \Cref{fig:geom_meaning}: for each $j$,  $\mathbf{b}_{\mathcal{I}_M^M} = \mathbf{\Lambda}^{-(M-1)} \mathbf{e}$ and  $\mathbf{e}$  are \emph{projected}, 
\emph{stretched} and \emph{rotated}  using the $j$-th Fourier spectrum diagonal matrix $\diag(\mathbf{a}^{(j)})$ yields $\mathcal{E}_L^{(j)}$ and its total column subspace $\mathcal{W}_L$. If all of the projected and stretched $\mathbf{b}_M$'s are contained in $\mathcal{W}_L$, a real solution exists for \Cref{eq:X_K_Y_vector}.
Notice that in \Cref{eq:geometric}, $\forall i \neq j$, $\mathcal{E}_L^{(i)}$ expands the column space $\mathcal{E}_L^{(j)}$  to include $\mathbf{c}^{(j)}$. Thus, the minimal number of time delays required  in the vector case as in \Cref{eq:X_K_Y_vector} can be smaller than that of the scalar case.
\end{remark}

\begin{remark}[Interplay between Fourier spectra]
\label{rem:intp}

The vector case involves the interaction between the $J$ different Fourier spectra corresponding to each component of the state. This complicates the derivation of an explicit result for the minimal number of time delays as in the scalar case (\Cref{thm:sparse_time_delay}).
We note two important observations that illustrate the impact of the interplay between the $J$ Fourier spectra:
\begin{itemize}
    
    \item To ensure $\mathbf{c}^{(j)}$ lies in $\mathcal{W}_L$, each $\mathcal{E}_L^{(j)}$ should provide distinct vectors to maximize the dimension of $\mathcal{W}_L$. If a  linear dependency is present in  $\{\mathbf{a}^{(j)}\}_{j=1}^J$, \Cref{eq:geometric} no longer holds.
    
    \item Since $\mathbf{c}^{(j)}$ is projected using $\diag(\mathbf{a}^{(j)})$,  if $\mathbf{a}^{(i)\top} \mathbf{a}^{(j)} = 0$,  $\mathcal{E}_L^{(i)}$ will not contribute to increasing the dimension of $\mathcal{W}_L$.
\end{itemize}

\end{remark}

Drawing insight from the representation of the column space of $\mathbf{A}_{\mathcal{I}_M^M,L}$ as the Krylov subspace in \Cref{rem:geo}, we present a connection between the \emph{output controllability} from linear system control theory~\cite{kreindler1964concepts}, and the number of time delays required for linear models in a general sense.

\begin{definition}[Output controllability]
\label{def:output_c}
Consider a linear system with state vector $\mathbf{x}(t) \in \mathbb{C}^{M \times 1}$, $M \in \mathbb{N}$, $t \in \mathbb{R}^{+}$,
\begin{align}
    \label{eq:def_ou_lin}
    \dot{\mathbf{x}} &= \mathbf{A} \mathbf{x} + \mathbf{B} \mathbf{u}, \\
    \label{eq:def_ou_lin_2}
    \mathbf{y}  &= \mathbf{Cx} + \mathbf{Du},
\end{align}
where $\mathbf{A} \in \mathbb{C}^{M \times M}$, $\mathbf{B} \in \mathbb{C}^{M \times N}$, $\mathbf{C} \in \mathbb{C}^{P \times M}$, $\mathbf{D} \in \mathbb{C}^{P \times N}$. $\mathbf{y}(t) \in \mathbb{C}^{P \times 1}$ is the output vector. The above system is said to be output controllable if for any $ \mathbf{y}(0), \mathbf{y}^\prime \in \mathbb{C}^{P \times 1}$, there exists $t_1 \in \mathbb{R}^{+}, t_1 < +\infty$ and $ \mathbf{u}^{\prime} \in \mathbb{C}^{N \times 1}$, such that under such input and initial conditions, the output vector of the linear system can be transferred from $\mathbf{y}(0)$ to $\mathbf{y}^\prime = \mathbf{y}(t_1)$. 
\end{definition}

Recall that the necessary and sufficient condition~\cite{kreindler1964concepts,gruyitch2018observability} for a linear system to be output controllable is given in \Cref{def:output_control}. A natural definition for the output controllability index that is similar to the controllability and observability index is given in \Cref{def:output_c_index}. 
We summarize the conclusion in \Cref{thm:control_vec_connection} that the output controllability index minus one is a tight upper bound for the number of time delays required for the linear model in the general sense. We again emphasize that the particular linear system with input and output in \Cref{thm:control_vec_connection} is solely \emph{induced} by the Fourier spectrum of the nonlinear dynamical system on the attractor.

\begin{definition}[Output controllability test]
\label{def:output_control}
The system in \Cref{eq:def_ou_lin,eq:def_ou_lin_2} is output controllable if and only if 
$
    \mathcal{OC}(\mathbf{A,B,C,D}; M) \triangleq 
    \begin{bmatrix}
    \mathbf{CB} & \mathbf{CAB} & \ldots & \mathbf{CA}^{M-1}\mathbf{B} & \mathbf{D}
    \end{bmatrix} 
$
is full rank. Note that when $\mathbf{D} = \mathbf{0}$, we omit $\mathbf{D}$ in the notation.
\end{definition}
\begin{definition}[Output controllability index]
\label{def:output_c_index}
If the system in \Cref{eq:def_ou_lin,eq:def_ou_lin_2} is output controllable, then the output controllability index is defined as the least integer $\mu$ such that 
    $\mathcal{OC}(\mathbf{A,B,C,D}; \mu) \in \mathbb{C}^{P \times (\mu+1)N}$ 
 is full rank.
\end{definition}

\begin{lemma}
\label{lem:EF_lemma}
For any matrix $\mathbf{A}$ that is a horizontal stack of diagonal matrices, the row elimination matrix $\mathbf{E}$ that removes any row that is a zero vector leads to a full rank matrix with the rank of original matrix. Moreover, $\mathbf{E^\top EA} = \mathbf{A}$.
\end{lemma}
\begin{proof}
See \Cref{apdx:EF_proof}. 
\end{proof}

\begin{theorem}
\label{thm:control_vec_connection}
Following definitions in \Cref{eq:e,eq:lambda}, consider the following induced linear dynamical system with output controllability index $\mu$:
\begin{align*}
    \dot{\mathbf{Z}} &= \mathbf{A} \mathbf{Z} + \mathbf{Bu}\\
    \mathbf{y} &= \mathbf{CZ} 
\end{align*}
with 
\[
\mathbf{A} = 
\begin{bmatrix}
\mathbf{\Lambda}^{-1} & & \\
& \ddots & \\
& & \mathbf{\Lambda}^{-1}
\end{bmatrix} \in \mathbb{C}^{MJ \times MJ},
\] 
\[
\mathbf{B} = 
\begin{bmatrix}
\mathbf{e} & & \\
& \ddots & \\
& & \mathbf{e}
\end{bmatrix} \in \mathbb{C}^{MJ\times J},
\]
\[
\mathbf{C}^{\prime} = 
\begin{bmatrix}
\diag(\mathbf{a}^{(1)}) & 
\ldots & 
\diag(\mathbf{a}^{(J)})
\end{bmatrix} \in \mathbb{C}^{M \times JM},
\]
\[\mathbf{C} = \mathbf{EC^\prime} \in \mathbb{C}^{P \times JM},\]
where $P$ is the number of non-zero row vectors in $\mathbb{C}^\prime$,  and $\rank \left(\mathbf{C}\right) = \rank \left(\mathbf{C}^{\prime}\right) = P$ as indicated by \Cref{lem:EF_lemma}.
Then,
$\mu-1$ is a tight upper bound on the minimal number of time delays that ensures the existence of solution of \Cref{eq:lem_vector_residual_R}, and  thus a perfect reconstruction of the dynamics.
\end{theorem}

\begin{proof}
See \Cref{apdx:thm_control_vec_connection}.
\end{proof}

\section{Dynamic mode decomposition of a linear model with 
time-delays}
\label{sec:connect_koopman}

As indicated earlier, the trajectory predicted by linear models with time-delay can be viewed as an observable from an associated high dimensional linear system. To see this, consider a uniformly sampled trajectory data  of length $M$, $\{\mathbf{x}_j\}_{j=0}^{M-1}$. The $L$ time-delay vector for a $J$-dimensional nonlinear system $\mathbf{x} \in \mathbb{R}^{J}$ is defined as, 
\begin{equation}
\mathbf{h}_k = 
\begin{bmatrix}
\mathbf{x}_{k-L} \\ \vdots \\ \mathbf{x}_{k} 
\end{bmatrix}, \quad L \le k \le M-1. 
\end{equation}
If the trajectory data can be well approximated by a linear model with $L$ time-delays of the form in \Cref{eq:linear_ass}, then one has the so-called high order dynamic mode decomposition~\cite{brunton2017chaos,le2017higher} for $L \le k \le M-2$, 
\begin{align}
\mathbf{h}_{k+1}  &\approx \mathbf{A}_L \mathbf{h}_{k},\\
\mathbf{x}_{k+1} = \mathbf{E}_L \mathbf{h}_{k+1} &\approx \mathbf{E}_L \mathbf{A}_L \mathbf{h}_{k} = \mathbf{W}_L \mathbf{x}_{k-L} + \ldots + \mathbf{W}_0 \mathbf{x}_{k} \\
 \mathbf{x}_{k+1} = \mathbf{E}_L \mathbf{h}_{k+1} &\approx \mathbf{E}_L \mathbf{A}_L^{k+1-L} \mathbf{h}_L = \mathbf{Q}_{L} \mathbf{\Lambda}^{k+1-L} \mathbf{P}_L \\ 
 \label{eq:hodmd} \mathbf{x}_{k+1} & \approx  \sum\nolimits_{i=1}^{J(L+1)}\lambda_i^{k+1-L}  \mathbf{q}_i  \mathbf{p}^{\top}_i \mathbf{h}_L
\end{align}
where $\mathbf{E}_L \triangleq \begin{bmatrix} \mathbf{0} & \ldots & \mathbf{0} & \mathbf{I} \end{bmatrix} \in \mathbb{R}^{J \times J(L+1)}$, and
$\mathbf{A}_L \in \mathbb{R}^{J(L+1) \times J(L+1)}$ is known as the block companion matrix, 

\begin{equation}
\mathbf{A}_L = 
\begin{bmatrix}
             & \mathbf{I}        &                  &               &                \\
             &                   &  \mathbf{I}      &               &                \\
             &                   &                  &  \ddots       &                \\
             &                   &                  &               & \mathbf{I}     \\
\mathbf{W}_L & \mathbf{W}_{L-1}  &  \mathbf{W}_{L-2}&  \ldots       &   \mathbf{W}_0        
\end{bmatrix} = \mathbf{P}_L \mathbf{\Lambda}_L \mathbf{P}_L^{-1},
\end{equation}
and 
\begin{equation}
\mathbf{P}_L^{-1} \triangleq \begin{bmatrix}
\mathbf{p}^{\top}_1 \\
\vdots\\
\mathbf{p}^{\top}_{J(L+1)}
\end{bmatrix}, \quad
\mathbf{Q}_L \triangleq  \mathbf{E}_L \mathbf{P}_{L} = \begin{bmatrix}
\mathbf{q}_1 & \ldots &  \mathbf{q}_{J(L+1)}
\end{bmatrix}.
\end{equation}

Note that the above decomposition in \Cref{eq:hodmd} reduces to the standard DMD when $L=0$, i.e., 
\begin{equation}
\mathbf{x}_{k+1}=\sum\nolimits_{i=1}^{J}\lambda_i^{k+1-L}  \mathbf{q}_i  \mathbf{p}^{\top}_i \mathbf{x}_0, \quad \forall L \le k \le M-2, 
\end{equation} 
where $\mathbf{q}_i$ and $\{ \lambda_i^{k+1-L} \mathbf{p}^{\top}_i \mathbf{x}_0\}_{k=0}^{M-2}$ are sometimes referred to as the $i$-th \emph{spatial modes} and \emph{temporal modes} respectively.
With more time-delays $L$, the maximal number of linear waves in the model increases with $J(L+1)$. 
As a side note, the above modal decomposition can be interpreted as an \emph{approximation} to the Koopman mode decomposition on the trajectory with $L$ time-delays as observables~\cite{brunton2017chaos,arbabi2017ergodic,arbabi2017study}.



\section{{Verification and practical consideration}}
\label{sec:experiments}


In this section, we start with a simple example and discuss  practical numerical considerations.

\subsection{ 5-mode sine signal}
\label{sec:toy_5modes}

First, an explicit time series consisting of five frequencies with a long period $T=100$ is considered:
\begin{align}
    \nonumber
    x(t) & = 
    0.3 \cos(\frac{2\pi t}{100}) + 
    0.5 \sin(\frac{4\pi t}{100}) + 
    0.9 \cos(\frac{8\pi t}{100}) \\ 
     & \label{eq:5_mode_sine} + 
    1.6 \sin(\frac{16\pi t}{100}) + 
    1.2 \cos(\frac{24\pi t}{100}).
\end{align}
Such a signal may be realized, for instance, by observing the first component of a 10-dimensional linear dynamical system. 
The sampling rate is set at 1 per unit time, which is arbitrary and considered for convenience, and the signal is sampled for two periods from $n=0$ to $n=199$. Thus we have a discretely sampled time series of length 200 as $\{x_n\}_{n=0}^{199}$ with $x_n = x(t)|_{t=n}$. Only the first 20\% of the original signal is used, which is 40\% of a full period with around 20 to 30 data points sampled. The variation in the number of data points is due to the fact that we fix the use of first 20\% of trajectory, and then reconstruct the signal with a different number of time delays. We solve the least squares problem in the time domain with the iterative least squares solver \texttt{scipy.linalg.lstsq}~\cite{jones2014scipy} with lapack driver as \texttt{gelsd}, and cutoff for small singular values as $10^{-15}$. 
\begin{figure}[htbp]
  \centering
    \includegraphics[width=\linewidth]{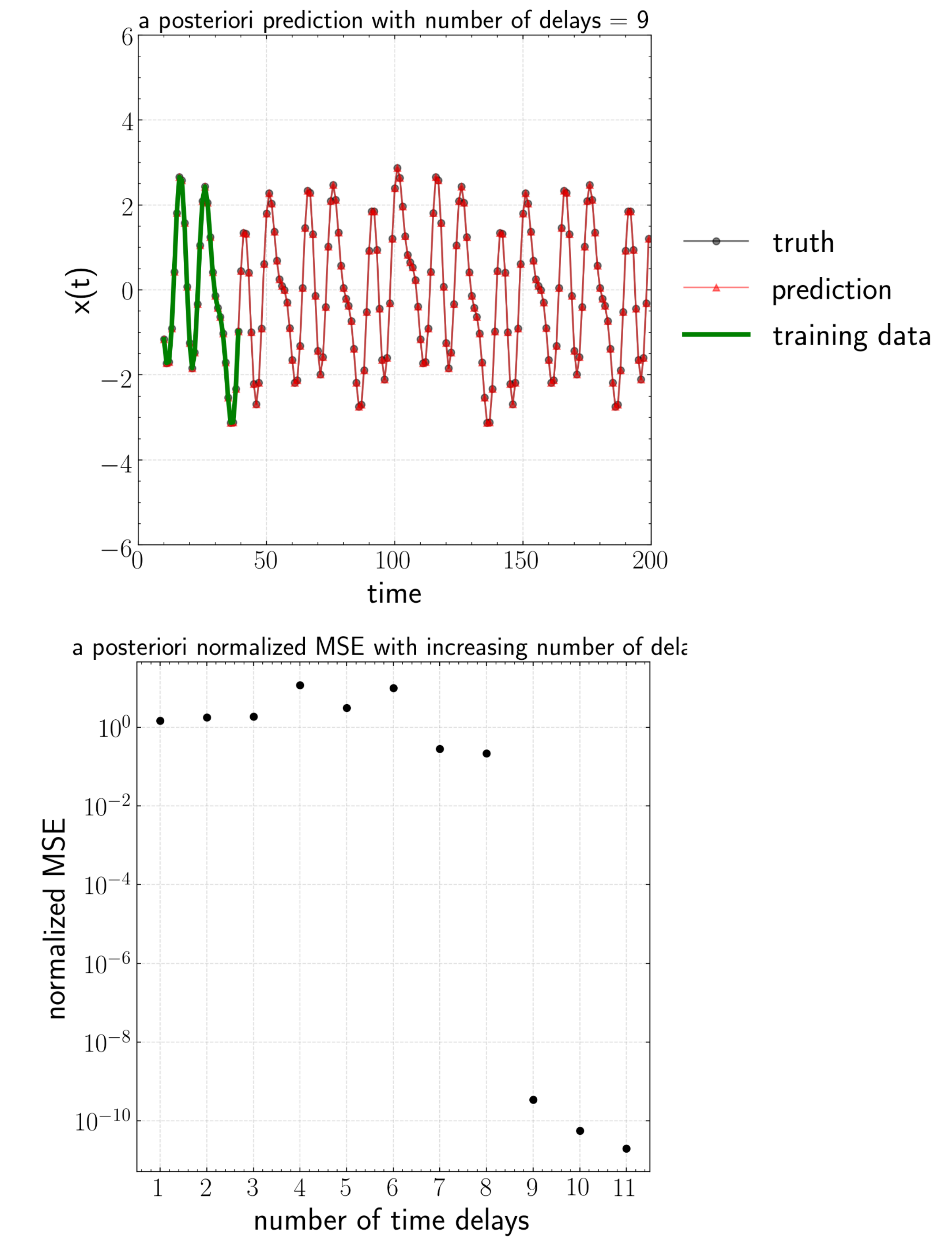}
  \caption{Top: A posteriori prediction vs ground truth, time delayed linear model with number of delays $L=9$. Bottom: A posteriori MSE normalized by standard deviation of $x(t)$ vs number of time delays.}
  \label{fig:result_1_delay_9}
\end{figure}
The analysis in \Cref{thm:sparse_time_delay} implies that one can avoid using the \emph{full period} of data for exact prediction. Numerical results are presented in \Cref{fig:result_1_delay_9} with  number of time delays $L=9$. These results  show that time delayed DMD, unlike non-linear models such as neural networks, avoid the requirement of a \emph{full period} of data  when the dynamics is expressible by a set of sparse harmonics.
From \Cref{thm:sparse_time_delay}, the 5-mode signal has $P = 10$ non-zero Fourier coefficients in the Fourier spectrum, and thus the least number of delays is $L= P-1 = 9$, which agrees well with  \Cref{fig:result_1_delay_9} which shows the a posteriori mean square error normalized by the standard deviation of the data , between prediction and ground truth.  \Cref{fig:result_1_delay_9}  clearly shows that a sharp decrease of a posteriori error when the number of delays $L=9$. 


Now we will consider a different scenario. As explained earlier, linear time delayed models can  avoid the use of a \emph{full period} of data if  there is enough information to determine the solution within the first $P$ states. Thus, if one increases the sampling rate, less data will be required to  recover an accurate solution. However, one still needs to numerically compute the solution of a linear system, while the condition number grows with increasing sampling rates.  As displayed in \Cref{fig:result_3_mse_vs_M}, the condition number increases in both time and spectral domain formulations, with increasing sampling rate.

Using \texttt{scipy.linalg.lstsq}~\cite{jones2014scipy} and a time domain  formulation, we found that there is no visual difference between the truth and a posteriori prediction when the condition number is below $10^{13}$, i.e., $M \le 300$ in the spectral domain, or $M \le 200$ in the time domain. However, as the condition number grows beyond $10^{13}$ (i.e. machine precision noise of even $10^{-16}$ can  contaminate digits around 0.001), a posteriori prediction error can accumulate when $M=400$  (\Cref{fig:result_2_bad_result}). 

\begin{figure}[bthp]
  \centering
    \includegraphics[width=1\linewidth]{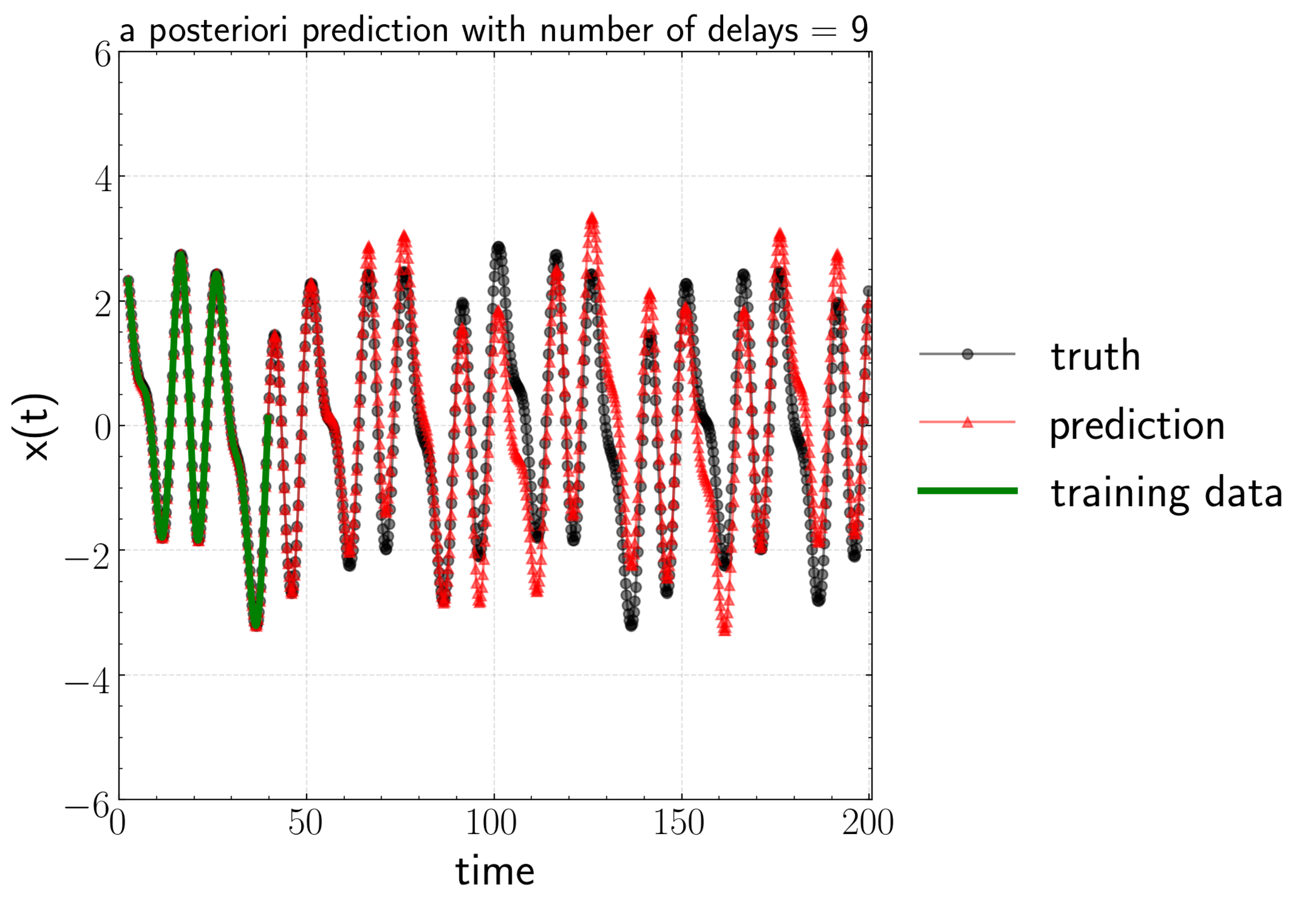}
  \caption{Prediction vs ground truth when sampling rate is excessive, e.g., $M=400$}
  \label{fig:result_2_bad_result}
\end{figure}





\subsection{Numerical considerations}
\label{sec:method}
In practical terms, one can pursue two general formulations to numerically compute the delay transition matrix $\mathbf{K}$ in \Cref{eq:argmin_W}:
\begin{enumerate}
\item \textit{Formulation in time domain:} If all available delay vectors and corresponding future states are stacked, the direct solution of \Cref{eq:argmin_W} is a least square problem in the time domain with the requirement of at least $P$ samples.
\item \textit{Formulation in spectral domain:} In this approach,  the Fourier signals from a full period of data is extracted and  \Cref{eq:linear_system_sparse} is numerically solved. 
\end{enumerate}

\subsubsection{Ill-conditioning due to excessive sampling rate}

Consider signals that consist of a finite number of harmonics with the index set of Fourier coefficients as $\mathcal{I}_M^P$. Since the first half of the indices  $i_0,\ldots,i_{P/2-1}$ is determined by the inherent period of each harmonic, these indices are independent of the number of samples per period $M$, as long as $M$ satisfies the Nyquist condition. It is thus tempting to choose a relatively large sampling rate. However, this may not be favorable from a  numerical standpoint. When $L=P-1$ and the sampling rate is excessive compared to the potentially lower frequency dynamics of the system, each column could become nearly linearly dependent. We will now explore the circumstances under which the corresponding linear system in either the spectral or time domain can become ill-conditioned. It has to also be recognized that the denominator in \Cref{eq:exact_spectral_solution} consists of the difference between different nodes on the unit circle, and can therefore impact numerical accuracy.

The condition number of the Vandermonde matrix with complex nodes \Cref{eq:linear_system_sparse} is also pertinent to the present discussion. It is well known that the condition number of a Vandermonde matrix grows exponentially with the order of matrix $n$ when the nodes are real positive or symmetrically distributed with respect to the origin~\cite{cordova1990vandermonde}. When the nodes are complex, the numerical conditioning of a Vandermonde matrix can be as perfect as that of a  DFT matrix, 
or as poor as that of the quasi-cyclic sequence~\cite{gautschi1990stable}. Specifically,  it has been shown that a large square Vandermonde matrix is  ill-conditioned unless its nodes are nearly uniformly spaced on or about the unit circle~\cite{pan2016bad}. Interestingly, for a  rectangular Vandermonde matrix with $n$ nodes and order $N$, i.e., $\mathbf{V}_N(z_1,\ldots,z_n)$, Kunis and Nagel~\cite{kunis2018condition} provided a lower bound on the 2-norm condition number of the Vandermonde matrix that contains ``nearly-colliding" nodes: 
\begin{equation}
\label{eq:kunis}
    \kappa_2(\mathbf{V}_N(z_1,\ldots,z_n)) \ge \frac{\sqrt{6}}{\pi \tau} \approx \frac{0.77}{\tau}, 
\end{equation}
for all $\tau \le 1$, i.e., ``nearly colliding", 
where $\tau \triangleq N \min_{j \neq l} |t_j - t_l|_{\mathbb{T}}$, $|t_j - t_l|_{\mathbb{T}} \triangleq \min_{r \in \mathbb{Z}} |t_j - t_l + r|$. Applying the above result to \Cref{eq:linear_system_sparse}, when $M$ is large enough so that $\tau \le 1$ is satisfied\footnote{since $\tau = O(1/M)$},  the lower bound of the 2-norm condition number will increase proportionally with the number of samples per period $M$. Thus, the tightly clustered nodes due to excessive sampling  will lead to the ill-conditioning of the linear system in \Cref{eq:linear_system_sparse}.



\subsubsection{Sub-sampling within Nyquist limits}

\Cref{eq:kunis} shows that  the tight clustering of nodes due to excessive sampling can lead to ill-conditioning. A straightforward fix would thus be to filter out unimportant harmonics, and re-sample the signal at a smaller sampling rate that can still capture the highest frequency retained in the  filtering process.
In this way,  the nodes can be more favorably redistributed on the unit circle. Recall that, if the complex nodes of the Vandermonde matrix are  uniformly distributed on a unit circle, then one arrives at a perfect conditioning of the Vandermonde matrix with condition number of one similar to the DFT matrix~\cite{pan2016bad}. Without any loss of generality, we assume the number of samples per period $M$ is even. The wave numbers of sparse Fourier coefficients are denoted by $\mathcal{I}_M^P$. The sorted wave numbers are symmetrical with respect to $M/2$ and recall that the values of the first half of $\mathcal{I}_M^P$, i.e., $i_0, \ldots, i_{\frac{P}{2}-1}$ is independent of $M$, as long as the Nyquist condition is satisfied~\cite{landau1967sampling}. Then,  a continuous signal $x(t)$ is sub-sampled uniformly. Due to symmetry, the smallest number of samples per period $M^*$ that preserves the signal is $2(i_{\frac{P}{2}-1}+1)$.

\subsubsection{Effect of sampling rate, formulation domain, and numerical solver on model accuracy}

To compare the impact of different solution techniques, we choose several off-the-shelf numerical methods to compute $\mathbf{K}$ in either the time domain or spectral domain. These methods include: 

(i) \texttt{mldivide} from \texttt{MATLAB}~\cite{MATLAB_2010}, i.e., backslash operator which effectively uses QR/LU solver in our case;

(ii) \texttt{scipy.linalg.lstsq}~\cite{jones2014scipy}, which by default calls \texttt{gelsd} from  \texttt{LAPACK}~\cite{laug} to solve the minimum 2-norm least squares solution with SVD, and an algorithm based on divide and conquer;

(iii) Bj\"{o}rck \& Pereyra (BP) algorithm~\cite{bjorck1970solution} which is designed to solve the Vandermonde system exactly in an efficient way exploiting the inherent structure. For a $n\times n$ matrix, instead of the standard Gaussian elimination with $O(n^3)$  arithmetic operations and $O(n^2)$ elements for storage, the BP algorithm only requires $n(n+1)(2O_M+3O_A)/2$\footnote{$O_A$ and $O_M$ denote addition/subtraction and multiplication/division.} for arithmetic operations and no further storage than storing the roots and right hand side of the system.

As shown in \Cref{fig:result_3_mse_vs_M}, the condition number increases exponentially with increasing number of samples per period $M$, leading to a significant deterioration of accuracy.
Comparing the  time and spectral domain formulations,  \Cref{fig:result_3_mse_vs_M} shows that the solution for the spectral case is more accurate than the time domain solution when the sampling rate is low. This is not unexpected as one would need to perform FFT on a full period of data to find the appropriate Fourier coefficients in the spectral case. 
When $M>600$, however, the spectral domain solutions obtained by BP and \texttt{mldivide} algorithms blow up, while the time domain solution is more robust in that the error is bounded. Note that the singular value decomposition - in \texttt{lstsq} and in \texttt{mldivide} that removes the components of the solution in the subspace spanned by less significant right singular vectors - is extremely sensitive to noise. Further, from \Cref{eq:physics_to_fourier}, the difference between the formulations in the spectral  and time domains can be attributed to $\mathbf{V}_P(\omega^{k_0},\ldots,\omega^{k_{Q-1}})$ and $\diag(a_{i_0},\ldots,a_{i_{P-1}})$, which could be ill-conditioned. Thus, regularization in the time domain formulation is more effective.  \Cref{fig:result_3_mse_vs_M} also shows that, when the system becomes highly ill-conditioned, i.e., $M>600$, \texttt{lstsq} with thresholding $\epsilon=10^{-15}$ results in a more stable solution than \texttt{mldivide}.

It should be  mentioned that the condition number computed in \Cref{fig:result_3_mse_vs_M} around the inverse of machine precision, i.e., $O(10^{16})$, should be viewed in a qualitative rather than quantitative sense~\cite{drmavc2018data}.

\begin{figure}[bthp]
  \centering
    \includegraphics[width=\linewidth]{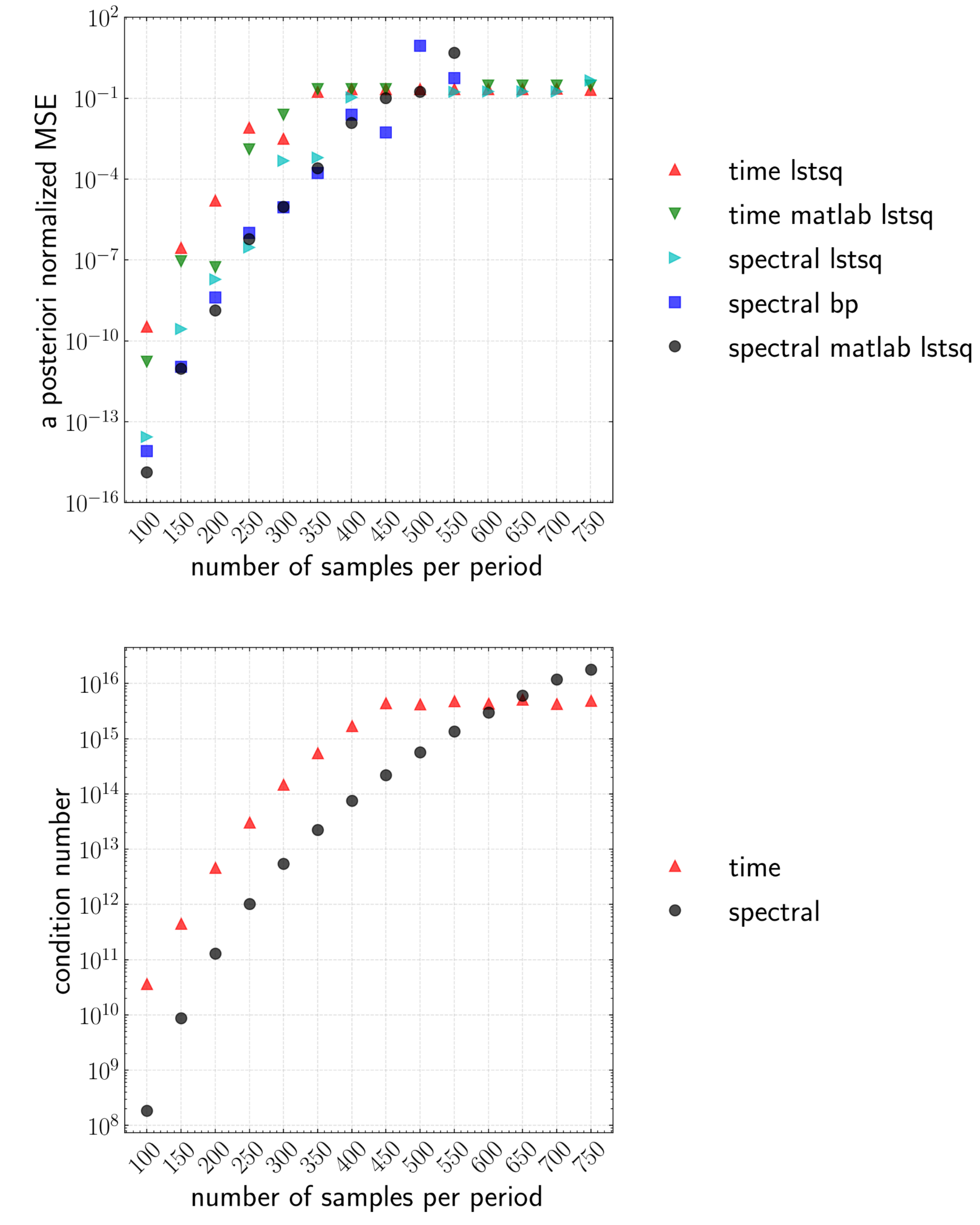}
  \caption{Top: A posteriori MSE normalized by the standard deviation of $x(t)$ with increasing sampling rate and different numerical solvers. Bottom: Numerical condition number with increasing sampling rate}
  \label{fig:result_3_mse_vs_M}
\end{figure}

\subsubsection{Effect of the number of time delays $L$ on condition number}
\label{sec:subsection_cond}
By adding more time delays than the theoretical minimum, the dimension of the  solution space grows, along with the features for least squares fitting. Accordingly, the null space becomes more dominant, and thus one should expect non-unique  solutions with lower residuals. Note that, for simplicity, the following numerical analysis assumes the scalar case, i.e., $J=1$. 

For the complex Vandermonde system in \Cref{eq:linear_system_sparse}, following Baz\'{a}n's work~\cite{bazan2000conditioning}, we discovered very distinct features of the asymptotic behavior of the solution, and the corresponding system in \Cref{eq:linear_system_sparse} when the number  of time delays $L \rightarrow \infty$.

(i) The norm of the minimum 2-norm solution of \Cref{eq:linear_system_sparse} $\lVert \mathbf{\hat{K}}_L \rVert_2 \rightarrow 0$ , as shown in \Cref{prop:norm_fn}.

(ii) An upper bound for the convergence rate of $\lVert \mathbf{\hat{K}}_L \rVert_2^2$ is derived in \Cref{lem:upper_bound}.

(iii) An upper bound on  the 2-norm condition number of \Cref{eq:linear_system_sparse} is shown in \Cref{prop:k2_cond_convergence} to scale with $1+O(1/\sqrt{L})$.  

\begin{proposition}
\label{prop:norm_fn}
$\displaystyle \lim_{L \rightarrow \infty}\lVert \mathbf{\hat{K}}_L \rVert_2 = 0$, where $\mathbf{\hat{K}}_L$ is the minimum 2-norm solution of \Cref{eq:linear_system_sparse}. 
\end{proposition}
\begin{proof}
See \Cref{apdx:prop_norm_fn}.
\end{proof}

\begin{lemma}
\label{lem:upper_bound}
$\forall L\ge P-1$, denote $\mathbf{\hat{K}}_L$ as the minimum 2-norm solution of \Cref{eq:linear_system_sparse}. The following tight upper bound can be derived
\begin{equation}
    \lVert \mathbf{\hat{K}}_L \rVert_2^2 \le \frac{\lVert \mathbf{\hat{K}}_{P-1} \rVert_2^2}{1 + \left \lfloor{\frac{L - P + 1}{M}}\right \rfloor}.
\end{equation}
\end{lemma}
\begin{proof}
See \Cref{apdx:proof_upper_bound}.
\end{proof}

\begin{proposition}
\label{prop:k2_cond_convergence}
Let $P$ be the number of non-zero Fourier coefficients. $\forall L\ge P-1$, denote $ \mathbf{\hat{K}}_{P-1}$ as the unique solution of \Cref{eq:linear_system_sparse}. With the minimal number of time delays, the upper bound on the 2-norm condition number of the system is given by 
\begin{align}
    \nonumber
    \kappa_2(\mathbf{A}_{\mathcal{I}_M^P,L}) &= \kappa_2(\mathbf{V}_{L+1}(\omega^{-i_0},\ldots,\omega^{-i_{P-1}})) \\
    &     \label{eq:our_bounds}
 \le 1 + \frac{d}{2} \left[ 1 + \sqrt{1 + \frac{4}{d}} \right],
\end{align}
where
\begin{align}
    d &\triangleq P \left[\left( 1 + \frac{\lVert \mathbf{\hat{K}}_{P-1} \rVert_2^2}{(P-1)  ({1 + \left \lfloor{\frac{L - P + 1}{M}}\right \rfloor}) \delta^2 } \right)^{\frac{P-1}{2}} - 1\right], \\
    \displaystyle \delta & \triangleq \min_{0 \le j < k \le P-1} |\omega^{-i_j} - \omega^{-i_k}|.
\end{align}
Further, if $L \rightarrow \infty$, then $\kappa_2(\mathbf{A}_{\mathcal{I}_M^P,L}) \rightarrow 1$, i.e.,  perfect conditioning is achieved.
\end{proposition}
\begin{proof}
See \Cref{apdx:k2_cond_convergence}.
\end{proof}

\begin{remark}
\label{rem:bound_upper_discussion}
Note that the bound in \Cref{prop:k2_cond_convergence} does not demand a potentially restrictive condition on the number of time delays, i.e., $L + 1 > 2(P-1)/\delta$ that is required in Baz\'{a}n's work, which utilizes the Gershgorin circle theorem for the upper bound of the 2-norm condition number~\cite{bazan2000conditioning}. More  recently, this constraint has been defined in the  context of the  nodes being ``well-separated"~\cite{kunis2018condition}. Applying such a result to our case, we  have 
\begin{equation}
    \label{eq:bazan_k2_norm_asl}
    \kappa_2(\mathbf{A}_{\mathcal{I}_M^P,L}) \le \sqrt{1 + \frac{2}{\frac{\delta(L+1)}{2P-2} - 1}}
\end{equation}
since we have an estimation for the convergence rate of the minimal 2-norm solution. 
However, although our upper bound in \Cref{prop:k2_cond_convergence} holds\footnote{and is more general than Baz\'{a}n's upper bound \Cref{eq:bazan_k2_norm_asl}} for all $L \ge P-1$, it is too conservative compared to Baz\'{a}n's upper bound when $L\rightarrow \infty$. To see this, denote $k_m \triangleq \min_{i,j \in \mathcal{I}_M^P, i,\neq j } \{ |k| | k = (i-j) \bmod{M} \} $, i.e., the minimal absolute difference between any pair of distinct indices in $\mathcal{I}^P_M$, in the sense of modulo $M$. Assuming that the number of samples per period is large enough so that $M \gg 2\pi k_m$, we have $\delta = \sqrt{2\left[1-\cos(2\pi k_m/M)\right]} \approx 2\pi k_m/M = O(1/M)$.
If we assume that the system with time delay $L$ is far from being perfectly conditioned, we have $\kappa_F(\mathbf{V}_{L+1})  \gg P + 2$, which leads to the following approximation for our upper bound, 
\begin{align}
    \nonumber
    & \kappa_2(\mathbf{V}_N) \le \frac{1}{2} \Big[ \kappa_F(\mathbf{V}_{L+1}) - P + 2 \\ \nonumber
    & + \sqrt{(\kappa_F(\mathbf{V}_{L+1}) - P +2)^2 -4} \Big] \approx  \kappa_F(\mathbf{V}_{L+1}) - P + 2 \\ 
    & \le d + 2.
\end{align}
Hence, for an excessively sampled case, if $L$ is small enough such that $\kappa_F(\mathbf{V}_{L+1}) \ge \kappa_2(\mathbf{V}_{L+1}) \gg P+2$ holds but large enough such that
\begin{equation}
    \label{eq:approx_assumption}
     \frac{\lVert \mathbf{\hat{K}}_{P-1} \rVert_2^2}{(P-1)  ({1 + \left \lfloor{\frac{L - P + 1}{M}}\right \rfloor}) \delta^2}  \ll 1, 
\end{equation}
then the approximated upper bound becomes
\begin{align}
    & 2 + d = 2 + P \left[\left( 1 + \frac{\lVert \mathbf{\hat{K}}_{P-1} \rVert_2^2}{(P-1)  ({1 + \left \lfloor{\frac{L - P + 1}{M}}\right \rfloor}) \delta^2 } \right)^{\frac{P-1}{2}} - 1\right], \nonumber \\
    & \nonumber \approx 2 +  \frac{P \lVert \mathbf{\hat{K}}_{P-1} \rVert_2^2}{2  \delta^2  ({1 + \left \lfloor{\frac{L - P + 1}{M}}\right \rfloor}) } \approx 2 + \frac{P \lVert \mathbf{\hat{K}}_{P-1} \rVert_2^2}{8 \pi^2 k_m^2/M^2  ({1 + \left \lfloor{\frac{L - P + 1}{M}}\right \rfloor}) }\\ &  = 2+ O\left(\frac{M^3}{L}\right).
\end{align}
Meanwhile, when $L$ is very large, and thus $\delta(L+1) > 2(P-1)$ is  satisfied, Baz\'{a}n's bound in \Cref{eq:bazan_k2_norm_asl} scales with $1 + O \left( \sqrt{M} / \sqrt{L} \right)$ for $L/M \gg 1$. Thus, to retain the same upper bound of condition number, one only needs to increase the number of time delays at the \emph{same} same rate as the sampling. 
\end{remark}

 \Cref{fig:result_4_mse_vs_L} shows that the residuals from the least squares problem in both the time and spectral domains  decrease exponentially with the addition of  time delays. Further, the a posteriori MSE shows significant improvement with the addition of time delays. 
\begin{figure}[htbp]
  \centering
    \includegraphics[width=\linewidth]{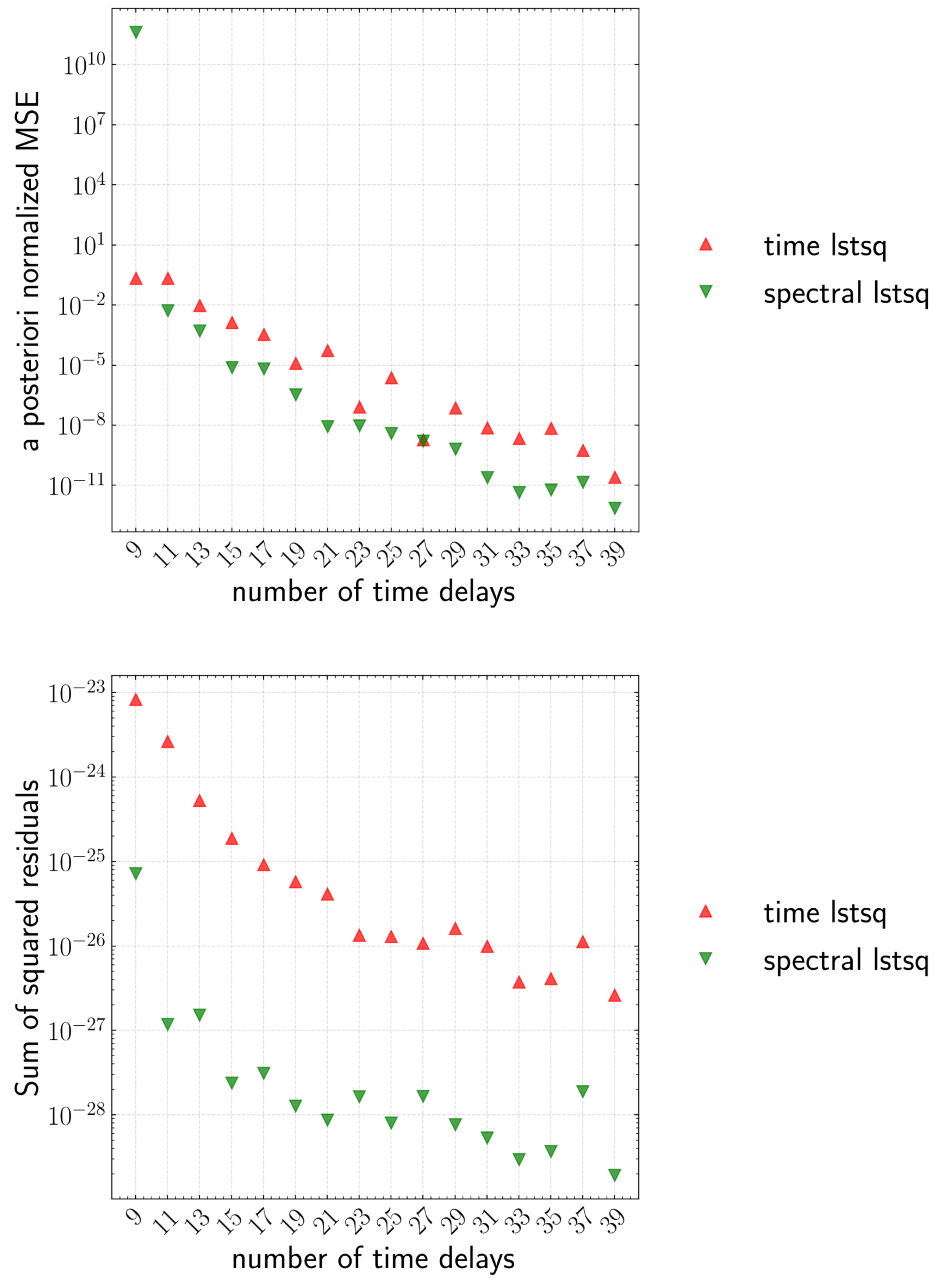}
  \caption{Effect of time delay $L$ on $M=500$ oversampling case. Top: A posteriori MSE normalized by standard deviation of $x(t)$ with increasing time delays. Bottom: Sum of squared residuals with increasing time delays.}
  \label{fig:result_4_mse_vs_L}
\end{figure}

 \Cref{fig:result_4_cond_vs_L} shows the trend of the  2-norm condition number in both the time and spectral domains. The condition number decays exponentially in the spectral case, but increases in the time domain case. This appears to be contradictory since the condition number is typically reflective of the quality of the solution. However, since SVD regularization is implicit in  \texttt{scipy.linalg.lstsq} with \texttt{gelsd} option, computing the 2-norm condition number in the same way as in the numerical solver, i.e., effective condition number~\footnote{i.e., SVD with the same thresholding ($\epsilon = 10^{-15}$) such that any singular value below $\epsilon \cdot \sigma_{max}$ is removed} is a more relevant measure of the quality of the solution of the SVD truncated system. Thus, the reasons for  improved predictive accuracy are due to a) the increasing dimension of the solution space for a potentially under-determined system with more time delays, and b) the  well conditioned system after SVD truncation as shown in \Cref{fig:result_4_cond_vs_L}. The large condition number in the time domain with increasing number of  delays is a result of  the ill-conditioning of $\mathbf{V}_P(\omega^{k_0},\ldots,\omega^{k_{Q-1}})$ and $\diag(a_{i_0},\ldots,a_{i_{P-1}})$ in \Cref{eq:subsample_linear}.  
\begin{figure}[htbp]
  \centering
    \includegraphics[width=\linewidth]{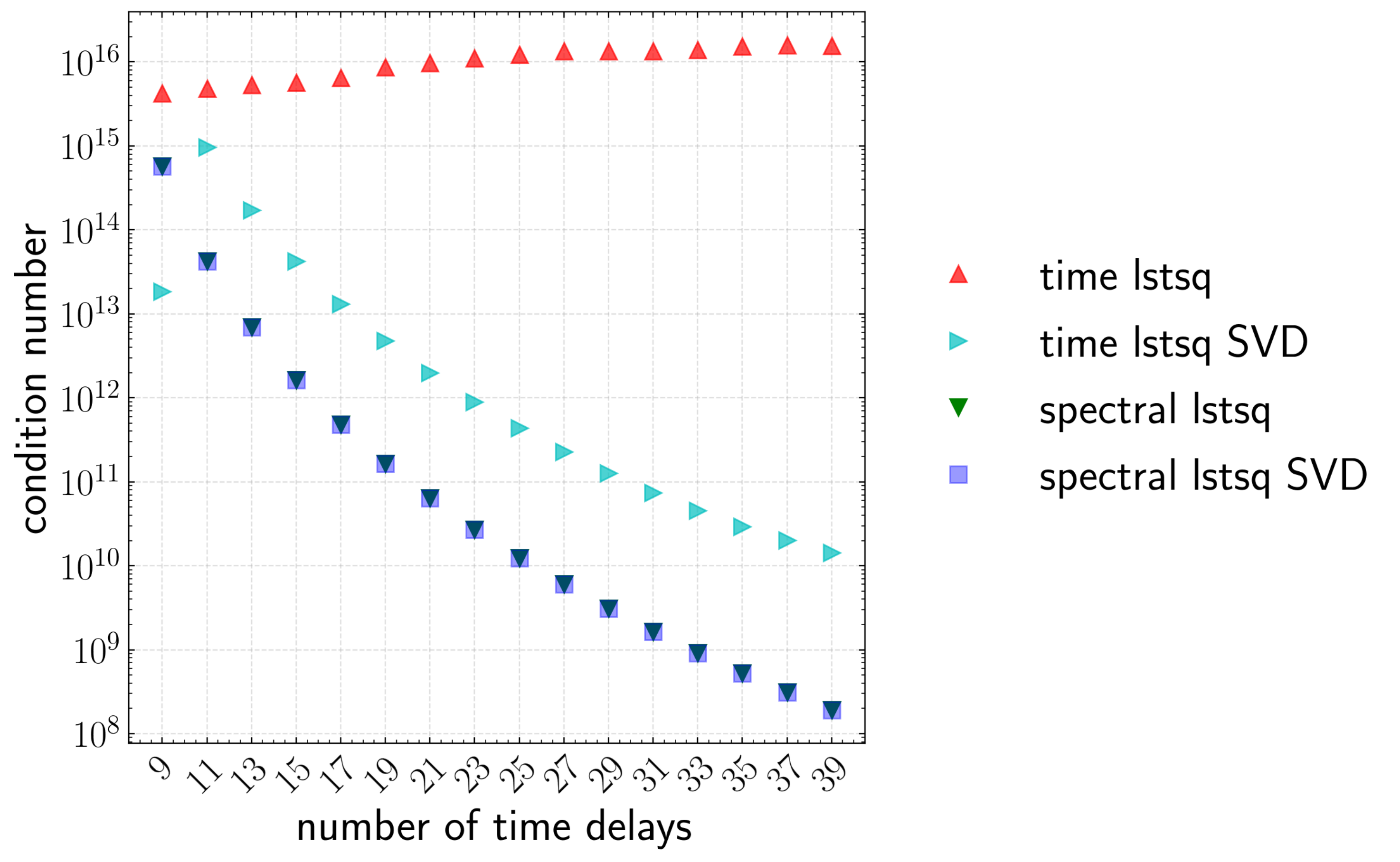}
  \caption{$M=500$ oversampling case: effective condition number decreases with increasing time delay $L$}
  \label{fig:result_4_cond_vs_L}
\end{figure}







\subsubsection{Effect of subsampling on model performance}



As indicated in \Cref{rem:bound_upper_discussion}, reducing the number of samples per period $M$ is shown to  decrease  the upper bound on the condition number. For a given signal, however, there is a restriction on the minimum possible $M$ compared to the number of time delays $L$. In the above case for the 5-mode sine signal,  $i_{\frac{P}{2}-1} = 12$, and thus the minimal sampling per period that one can use to perfectly preserve the original signal in the subsampling is $M=26$. The condition number with $M$ ranging from $26$ to $98$ is shown in \Cref{fig:result_5_cond_vs_M}. This shows the effectiveness of subsampling in reducing the condition number significantly. Correspondingly, the a posteriori normalized MSE is also reduced as shown in \Cref{fig:result_5_cond_vs_M}. 

The previous two subsections  demonstrated the role of numerical conditioning  on model performance. We note that explicit stabilization techniques~\cite{le2017higher,champion2018discovery} require further investigation. 


\begin{figure}[htbp]
  \centering
    \includegraphics[width=\linewidth]{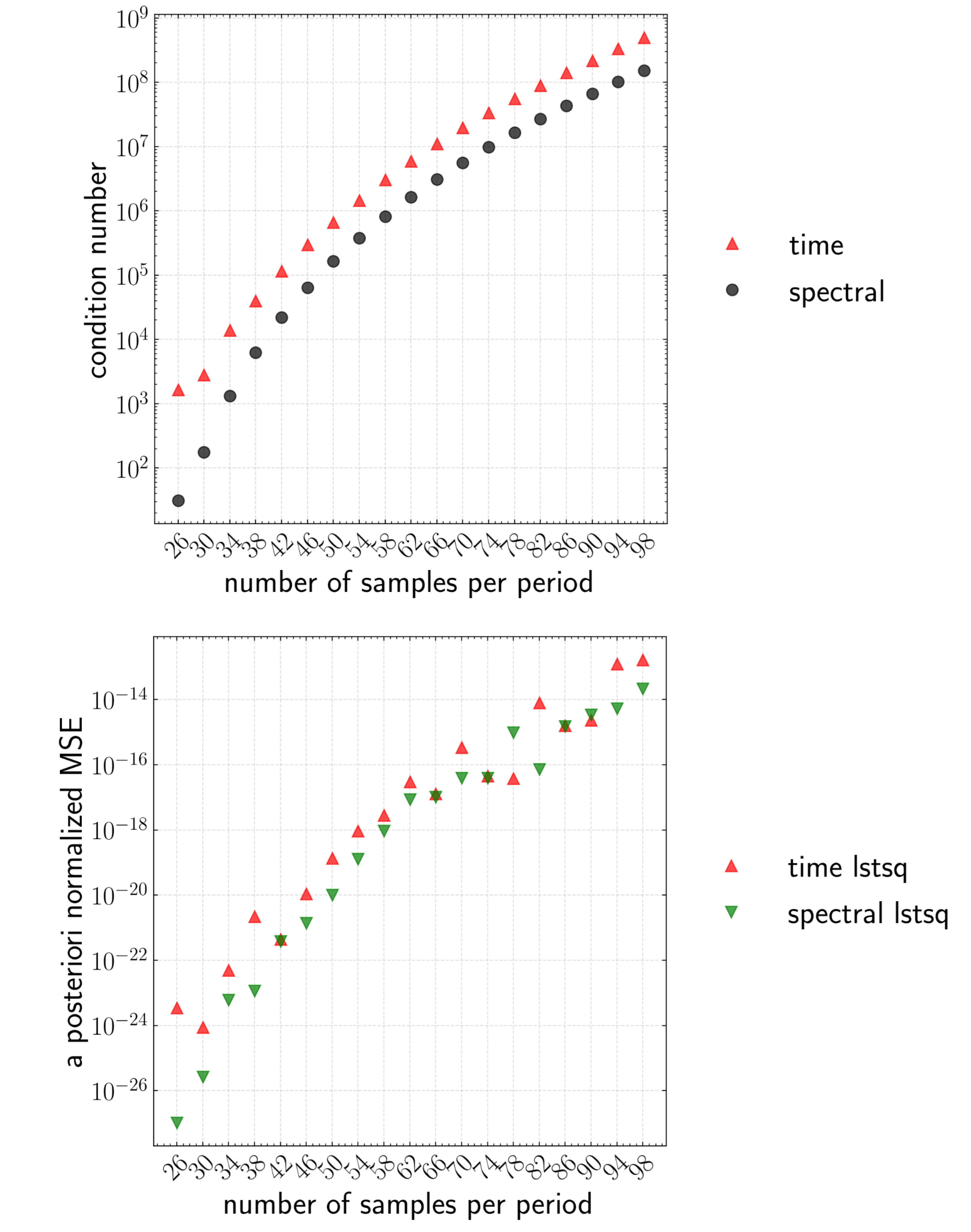}
  \caption{Top: Condition number as a function of sampling rate. Bottom: A posteriori normalized MSE with  sampling rate.}
  \label{fig:result_5_cond_vs_M}
\end{figure}


\subsection{Issues in large-scale chaotic dynamical systems}
\label{sec:opt_delay}

Lnear time delayed models have been investigated for chaotic dynamics on an attractor (for instance, ~\cite{brunton2017chaos}). The main challenges are two fold: 
a) Chaotic systems may require an infinite number of waves to resolve the continuous Koopman spectrum~\cite{mezic2005spectral}, and b)  Practical chaotic systems of interest in science and engineering science are large-scale. For example,  realistic fluid flow simulations, may be very large even after dimension reduction,  especially for advection-dominated problems~\cite{lee2020model}. This would further limit the expressiveness of linear models with time delay. 

To illustrate this, consider dimension reduction using SVD on the trajectory data $\{\mathbf{x}_j\}_{j=0}^{M-1}$. One can extract a reduced $r$-dimensional trajectory, $\{\hat{\mathbf{x}}_j\}_{j=0}^{M-1}$, i.e.,
\begin{equation}
\begin{bmatrix}
\mathbf{x}_0 & \ldots & \mathbf{x}_{M-1} 
\end{bmatrix} \approx \mathbf{U}_r \mathbf{\Sigma}_r \mathbf{V}_r^\top, \quad \hat{\mathbf{x}}_j = \mathbf{U}_r^\top \mathbf{x}_j \in \mathbb{R}^{r}.
\end{equation}
Recalling \Cref{eq:argmin_W,eq:argmin_W_minimal}, we have a similar analytic SVD-DMD solution on the time delay data matrix of the reduced $r$-dimensional system, i.e.,
\begin{equation}
\widehat{\mathbf{A}}_L = \mathbf{Q}_{r^{'}}^{\top} \mathbf{U}_r^{\top} \begin{bmatrix}
\mathbf{h}_{L+1} & \ldots & \mathbf{h}_{M-1} 
\end{bmatrix} \mathbf{Z}_{r^{'}} \mathbf{\Sigma}_{r^{'}}^{-1} \in \mathbb{R}^{r^{'} \times r^{'}},
\end{equation} 
with the following $r^{'}-$SVD regularization purely for numerical robustness
\begin{equation}
\mathbf{U}_r^{\top} \begin{bmatrix}
\mathbf{h}_{L} & \ldots & \mathbf{h}_{M-2} 
\end{bmatrix} \approx \mathbf{Q}_{r^{'}} \mathbf{\Sigma}_{r^{'}} \mathbf{Z}_{r^{'}}^{\top}.
\end{equation}
Note that ${\mathbf{A}}_L = \mathbf{Q}_{r^{'}}  \widehat{\mathbf{A}}_L  \mathbf{Q}_{r^{'}}^{\top} \in \mathbb{R}^{r(L+1)\times r(L+1)}$ with $\rank({\mathbf{A}}_L) = r^{'}$. Following the notations of the mode decomposition in \Cref{sec:connect_koopman}, we have
\begin{equation}
\mathbf{x}_{k+1} \approx \sum\nolimits_{i=1}^{r^{'}} \lambda_i^{k+1-L}  \mathbf{U}_r \mathbf{q}_i  \mathbf{p}^{\top}_i \mathbf{h}_L,
\end{equation}
where $\mathbf{U}_r \mathbf{q}_i$ and $\{ \lambda_i^{k+1-L} \mathbf{p}^{\top}_i \mathbf{h}_L\}_{k=0}^{M-2}$ are the spatial and temporal modes respectively.

Now we can describe the constraints on the maximal number of modes in the linear model $r^{'}$ from the time delay $L$. From the restrictions on matrix rank, we have 
\begin{equation}
    r \le \min \{ J,M \}, \quad r^{'} \le \min \{r(L+1), M-1-L\},
\end{equation}
as illustrated in \Cref{fig:wave_L_cond}. Clearly, we see the maximal number of waves $r^{'}$ stops increasing after the time delay $L$ surpasses the intersection point where $L_{*} = \frac{M}{r+1}-1$, $r^{'}_{*} = \frac{r}{r+1}M$. This relation indicates that keeping more POD modes in the dimension reduction increases the upper limit of the number of waves in the resulting linear models.  The corresponding time delay would decrease with respect to the peak. Interestingly,   for $L > \frac{M}{r+1}-1$, called ``overdelay", might yield an underdetermined linear system as in \Cref{eq:argmin_W_minimal}. For example, we can choose $L_{opt} = \lceil{\frac{M}{r+1}}\rceil$. The solution of that system would, however, result in a least square residual near machine precision, leading to overfitting even in a posteriori sense.
Note that practical problems may require denoising on the trajectory data. 

\begin{figure}[htbp]
  \centering
    \includegraphics[width=\linewidth]{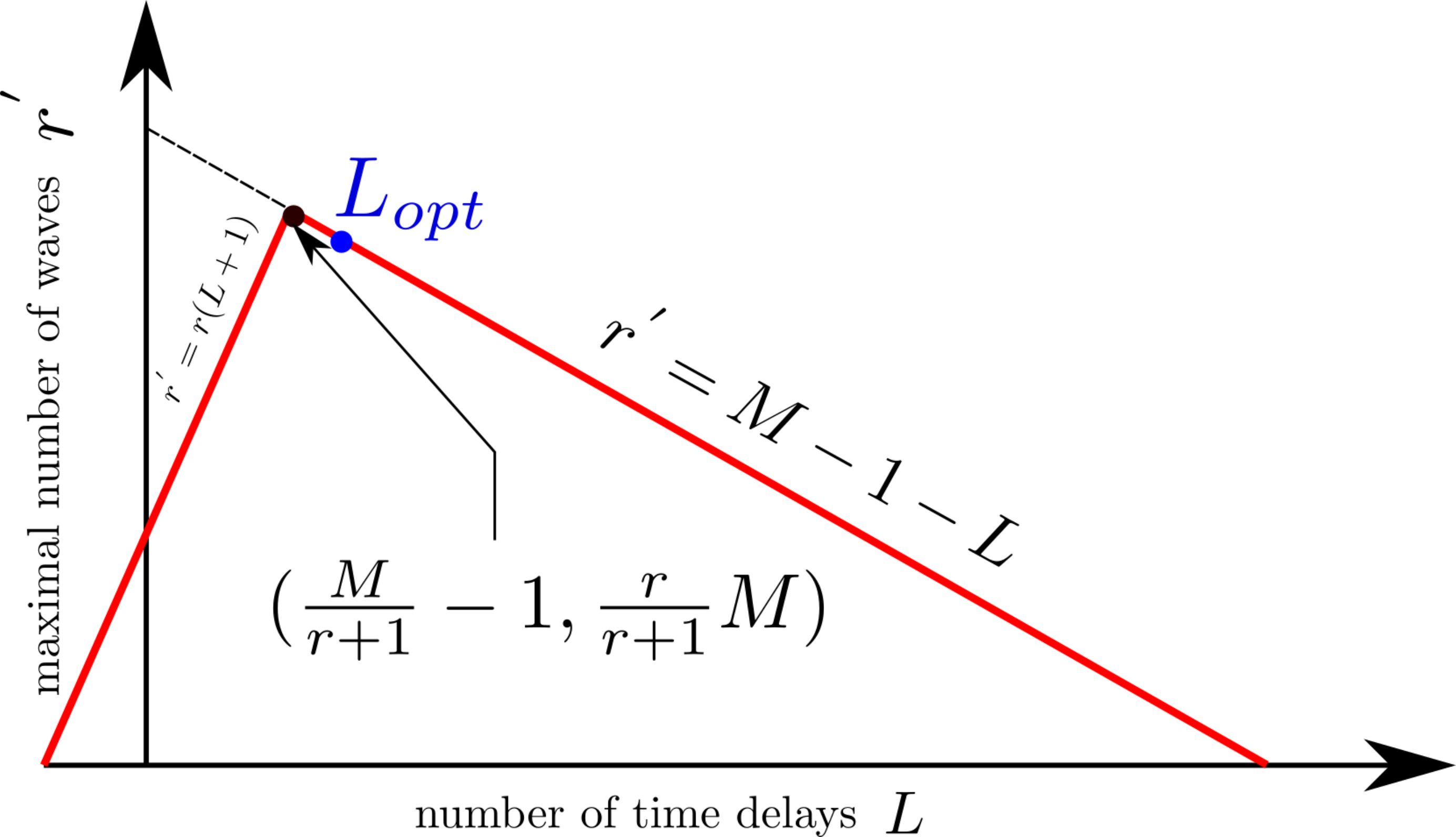}
  \caption{Constraints on maximal number of waves $r^{'}$ in the linear model with time delays.}
  \label{fig:wave_L_cond}
\end{figure}


\section{Applications}
\label{sec:application_nonlinear_system}


\subsection{Van der Pol oscillator}

Now we consider the Van der Pol oscillator (VdP) with forward Euler time discretization:
\begin{equation}{\label{eq:2d_vdp_system_disc}}
    \begin{bmatrix}
    x^{n+1}_1\\ x^{n+1}_2
    \end{bmatrix}
    = 
    \begin{bmatrix}
    x^{n}_1\\ x^{n}_2
    \end{bmatrix}
    + \Delta  t
    \begin{bmatrix}
    x^{n}_2 \\
    \mu (1 - x_1^{n}x_1^{n}) x_2^{n} - x_1^{n}
    \end{bmatrix},
\end{equation}
where $\mu = 2$, $x_1^{0} = 1$, $x_2^{0} = 0$, $\Delta t = 0.01$. After 530 time steps, the system approximately falls on the attractor with an approximate period of 776 steps. Total data is collected after the system falls on the attractor for 4 periods.

As shown in \Cref{fig:result_vdp_spectrum}, Fourier spectrum for each component of VdP system shows that the exhibition of an approximate sparse spectrum with $P=10$ and $P=18$ for $x_1$ and $x_2$ respectively. As indicated from \Cref{thm:sparse_time_delay}, the corresponding time delay and minimal sampling rate is summarized in \Cref{tab:vdp}. 

\begin{figure}[bhtp]
   \centering
   \includegraphics[width=0.7\linewidth]{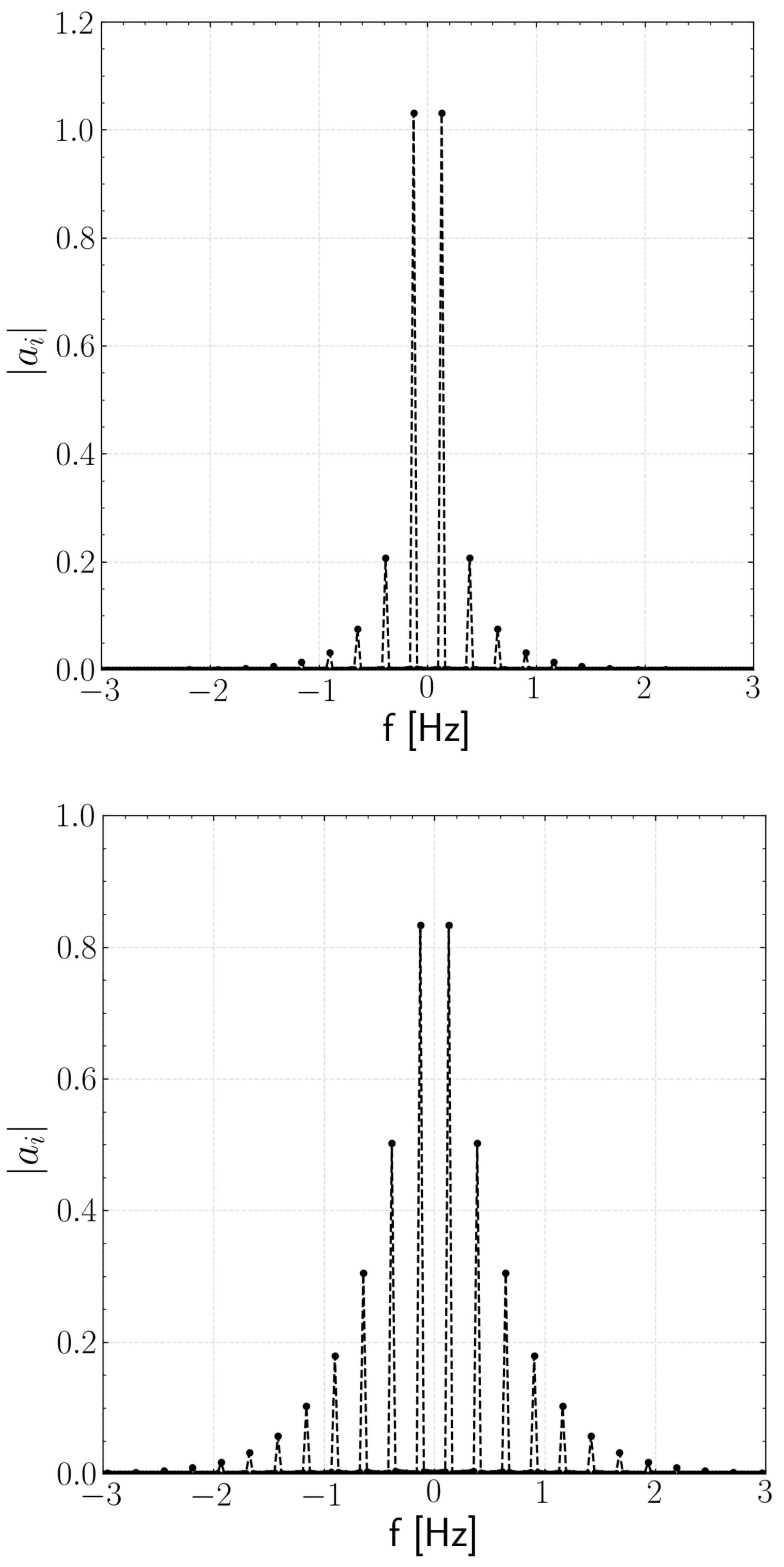}
   \caption{Fourier spectrum for VdP system. Top: $x_1$. Bottom $x_2$.}
   \label{fig:result_vdp_spectrum}
\end{figure}


\begin{table}
\caption{\label{tab:vdp}Summary of the structure of time delay embedding for VdP.}
\begin{ruledtabular}
\begin{tabular}{lcccc}
&  $P$ & $L$ & $i_{P/2-1}$ & $M_{min}$\\
\hline
    $\tilde{x}_1(t)$     & 10 & 9  & 9   & 20 \\
    $\tilde{x}_2(t)$     & 18 & 17 & 18  & 38 \\
    $\tilde{x}_{1,2}(t)$ &    &  8 &     & 38 \\
\end{tabular}
\end{ruledtabular}
\end{table}

\subsubsection{Prediction of the VdP system without a full period of data: scalar case}
\label{sec:vdp_scalar}

From \Cref{tab:vdp}, it is clear that the smallest number of samples per period is significantly smaller than the original number of samples per period, i.e., $M=776$. The analysis in  the previous section also showed that the choice of a smaller number of samples per period is helpful in reducing the condition number. Thus, we choose a moderately subsampled representation without any loss in reconstruction compared to the filtered representation. Individually treating the first and second components, we choose $M=200,100$ with theoretical minimum time delays $L=9,17$, respectively. 

Numerical results displayed in \Cref{fig:result_7_vdp_1} show that, even using training data that covers less than 25\% of the period for the first component, and 50\% of the period for the first component, the linear model with minimal time delays is still able to accurately predict the dynamics over the entire time period of  the limit cycle. Note that a similar predictive performance is expected for the original (unfiltered) VdP system.   

\begin{figure}[htbp]
  \centering
    \includegraphics[width=\linewidth]{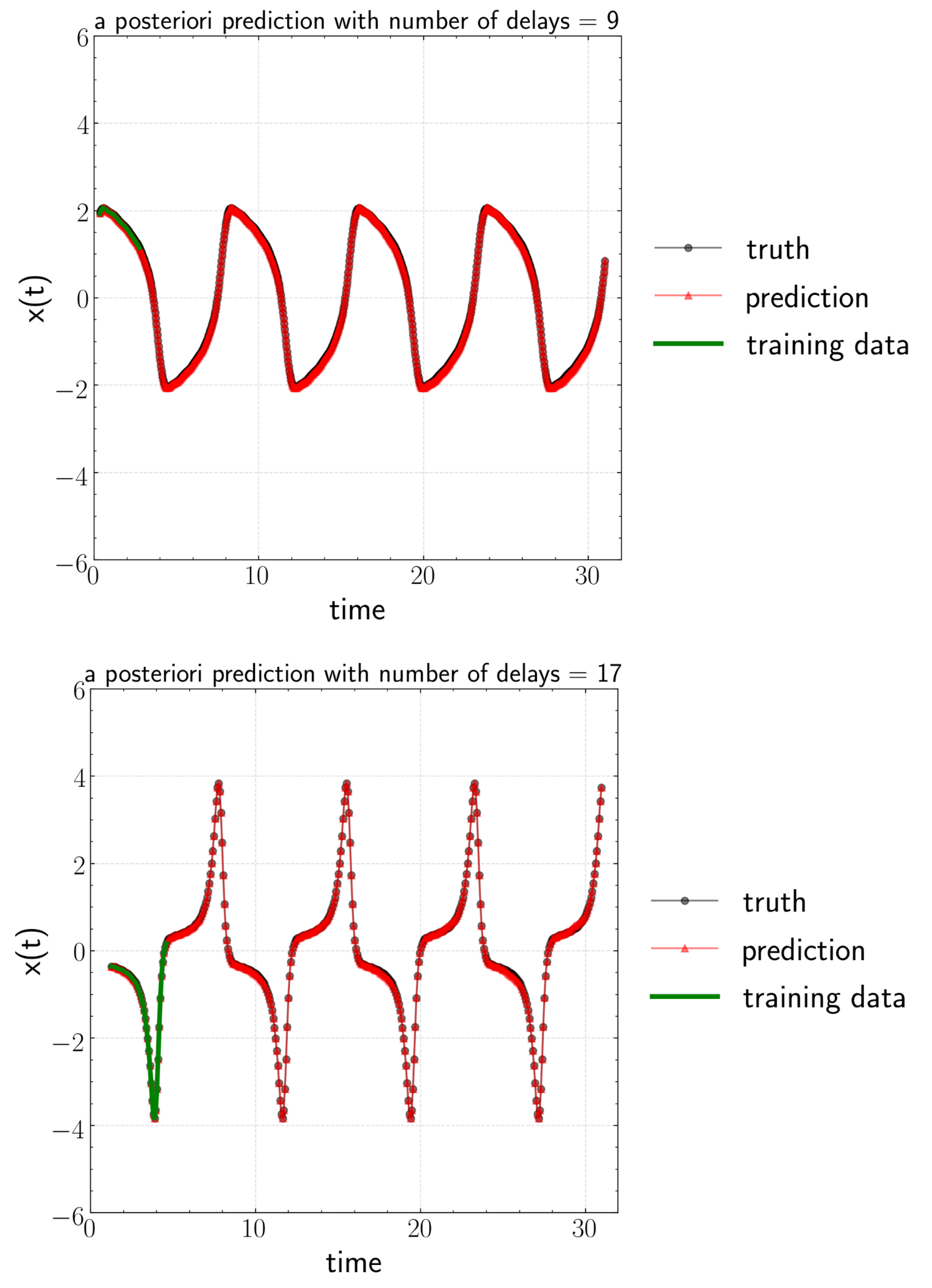}
  \caption{Prediction vs ground truth for each component of VdP. Top: first component. Bottom: second component.}
  \label{fig:result_7_vdp_1}
\end{figure}


\subsubsection{Prediction of VdP system without a full period of data: vector case}

As given in \Cref{tab:vdp}, \Cref{lem:vector_residual_R} predicts that the consideration of both components requires only 8 delays. The effectiveness of the criterion developed in \Cref{lem:vector_residual_R}  is confirmed to a resounding degree in \Cref{fig:result_11_vdp_vector_cond_mse}. The top figure shows the predictive performance of the time delayed linear model for the minimum number of delays and the bottom figure shows the behavior of the a posteriori normalized MSE versus the number of time delays. It should be recognized that in contrast to the scalar case, in which the minimal time delay can be directly inferred from the Fourier spectrum, the vector case requires \emph{iterative} evaluations of the rank test in \Cref{lem:vector_residual_R}.

\begin{figure}[htbp]
  \centering
    \includegraphics[width=1\linewidth]{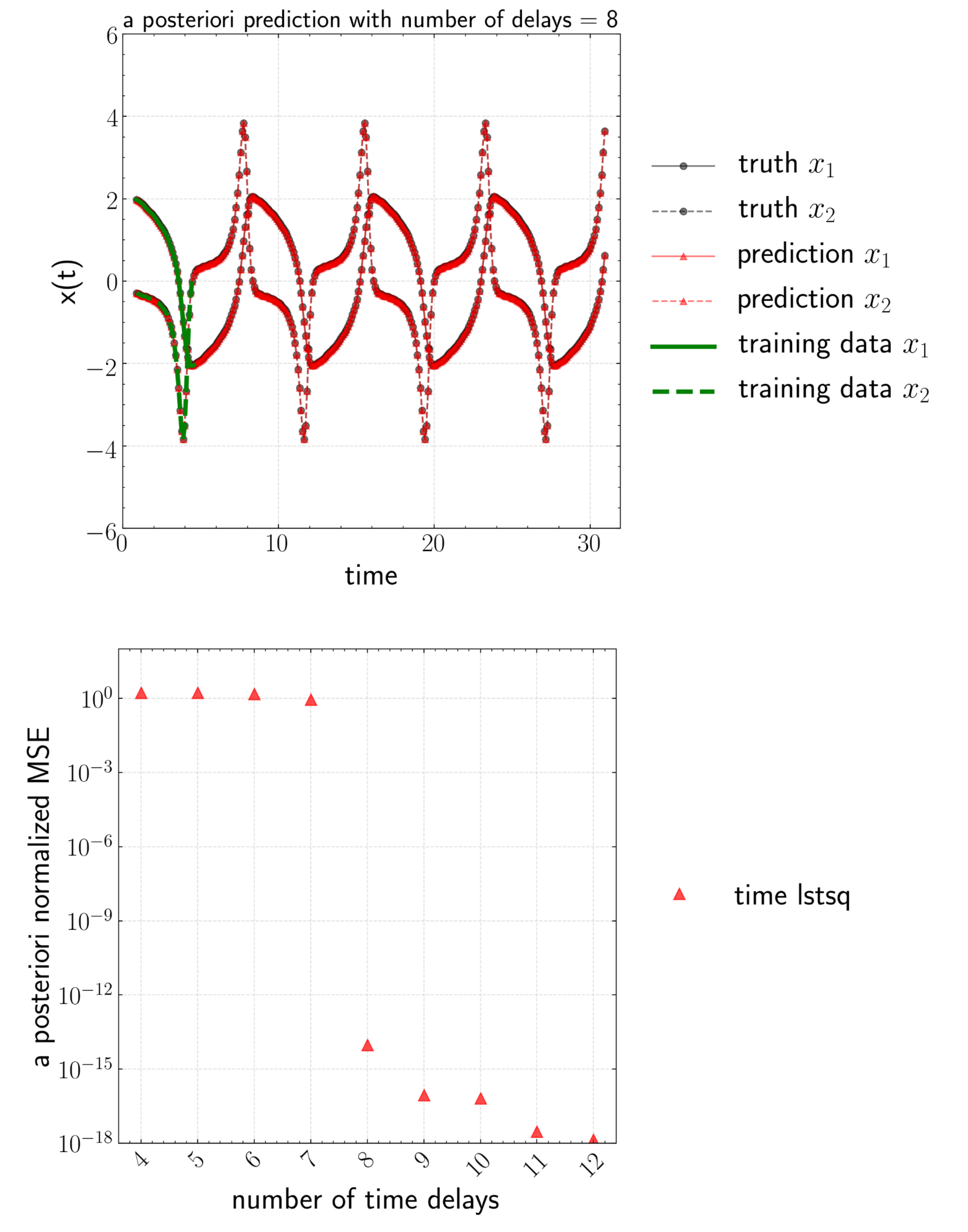}
  \caption{Top: Prediction vs ground truth with $M=80$ for VdP system. Bottom: A posteriori MSE normalized by standard deviation with as a function of the number of time delays for the vector case.}
  \label{fig:result_11_vdp_vector_cond_mse}
\end{figure}

\subsection{{Quasi-periodic signal}}

As indicated in Laudau's route to chaos~\cite{landau1944problem}, quasi-periodic systems play an important role in the transition from a limit cycle to fully chaotic flow.We consider the following  quasi-periodic signal 
\begin{equation}
\label{eq:quasi_periodic}
x(t) = \cos(\sqrt{2}t/2) \sin(\sqrt{3}t/2) \cos(t),
\end{equation}
where $t \in [0,40]$. Consider a sampling interval $\Delta t = 0.1$, we consider the linear model trained on the first 60 snapshots, i.e., $t\in [0,6]$.

\begin{figure}[htbp]
  \centering
    \includegraphics[width=1\linewidth]{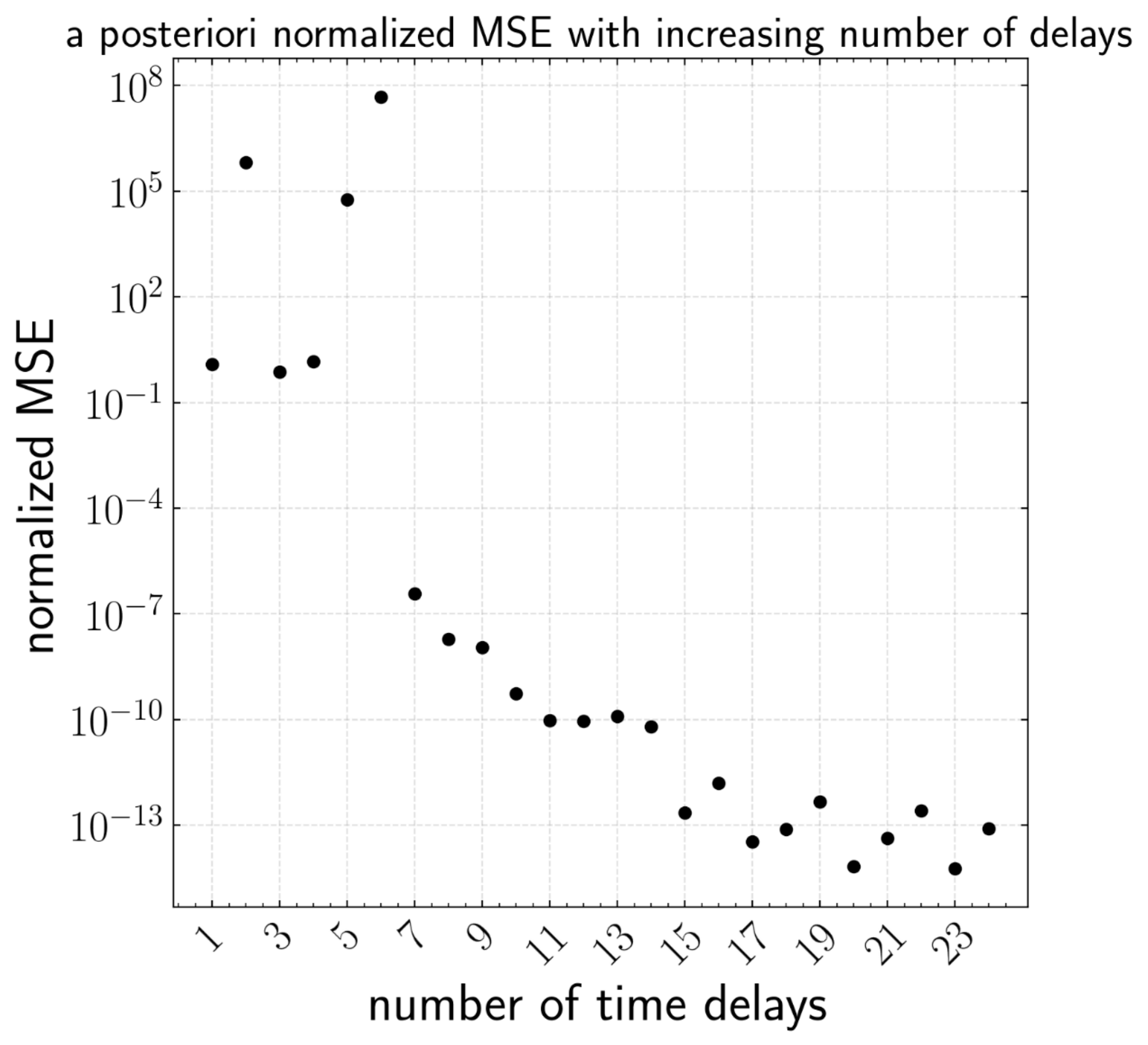}
  \caption{Top: Prediction vs ground truth for the toy quasi-periodic signal. Bottom: A posteriori MSE normalized by standard deviation with as a function of the number of time delays.}
  \label{fig:true_quasi}
\end{figure}

As shown in \Cref{fig:true_quasi}, the linear model with $L=7$ accurately predicts the future state behavior of the quasi-periodic system with only a fraction of data limited in the range $[-0.25, 0.55]$ while the whole data ranges from $[-0.944, 0.902]$. Indeed, the minimal time delay $L=7$ is determined by the number of frequencies in the signal.  The analysis on the minimal number of time delays for scalar time series as in  \Cref{sec:theory} can be extended to quasi-periodic system. Consider the trigonometric identity, we have the following equivalent equation of \Cref{eq:quasi_periodic}, 
\begin{align*}
x(t) = & \frac{1}{4} \bigg( \sin \Big(\frac{(\sqrt{2}  + \sqrt{3} + 2)t}{2}\Big) + 
  \sin \Big(\frac{(\sqrt{2} + \sqrt{3} - 2)t}{2} \Big) \\ & -
   \sin \Big(\frac{(\sqrt{2} - \sqrt{3} + 2)t}{2} \Big) -
    \sin \Big(\frac{(\sqrt{2} - \sqrt{3} - 2)t}{2} \Big)  
   \bigg).
    \end{align*}
Therefore, we require $L=P-1=7$ time delays to fully recover the signal which is confirmed in \Cref{fig:true_quasi}.




%

\subsection{Analysis of noise effect with pseudospectra}
\textcolor{black}{
Note that our analysis and experiments thus far have been based on noise-free assumptions. When additive noise is present in the data, the \emph{minimal} number of time delays as given by the results in \Cref{sec:theory} can be optimistic as we will confirm shortly. Alternatively, one might  de-noise the data as by using for instance, optimal SVD thresholding~\cite{gavish2014optimal} for the delay matrix with i.i.d. Gaussian noise.
To illustrate the effect of noise, the toy 5-mode sine signal in \Cref{sec:vdp_scalar} is considered, but the training horizon is increased to one complete period of data. Consider  additive i.i.d. Gaussian noise with signal-to-noise ratio (with respect to the standard deviation) of 1\%. To assess the influence of noise rigorously, we take an ensemble of 500 data trajectories and train a linear model with ordinary least squares on such data. In other words, for each sample trajectory, we have a slightly perturbed linear model associated with the data. The influence of noise is evaluated in the resulting distribution of eigenvalues (a priori sense) and long-time predictions (a posteriori sense)}. 
\textcolor{black}{
As shown in \Cref{fig:noise_aprior,fig:noise_aposter}, the theoretical optimality of $L=9$ does not hold  as the model becomes overly dissipative. Instead,   $L=20$ is required to have a reasonable prediction. It should be noted that the noise in the training data is too small to be observed in \Cref{fig:noise_aposter}, while the impact on the linear model is significant, as represented from the red shaded region. Moreover, as $L$ increases, it is observed that the ``cloud" of eigenvalues shifts from the left half plane towards the imaginary. Interestingly, the ``clouds" associated with spurious modes are much more scattered than those of the exact modes on the imaginary axis, i.e., the spurious modes are \emph{more sensitive} to the noise in the data. As $L$ becomes increasingly large, e.g., $L=39$, those clouds merge together along the imaginary axis, resulting in higher uncertainty due to the possibility of unstable modes. This is also reflected in the a posteriori predictions in \Cref{fig:noise_aposter}. Interestingly, the ensemble average of a posteriori prediction appears to show better predictions, even though each individual prediction can be divergent. This implies that an appropriate Bayesian reformulation could make the model more robust to noise~\cite{pan2020physics}. 
}
\begin{figure}[htbp]
 \centering
    \includegraphics[width=\linewidth]{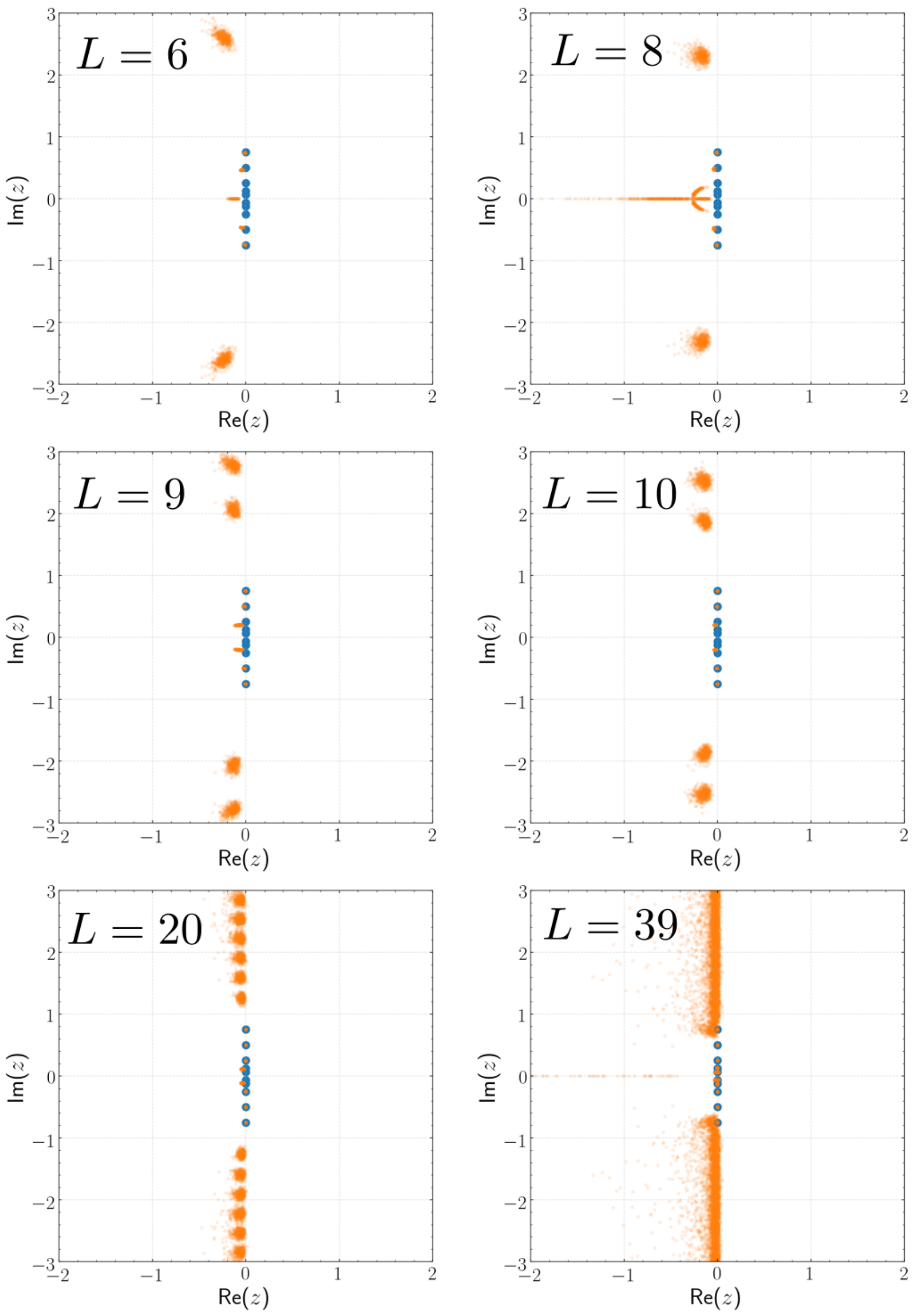}
 \caption{Eigenvalue distribution of linear model from noisy data with signal-to-noise ratio as 0.01 (orange) and noise-free data (blue). Time delay ranges from $L=6$ to $L=39$. }
 \label{fig:noise_aprior}
\end{figure}

\begin{figure}[htbp]
 \centering
    \includegraphics[width=\linewidth]{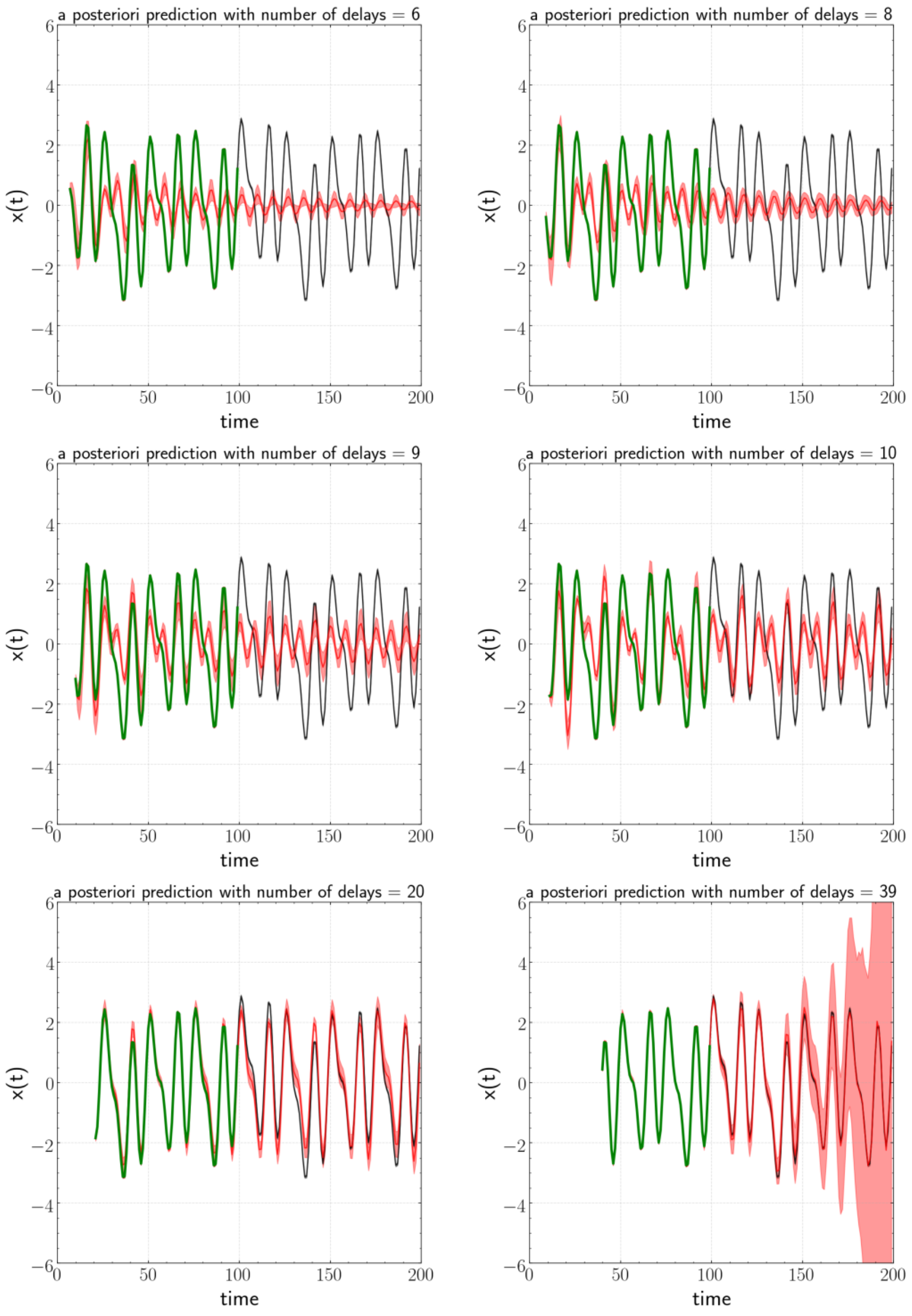}
 \caption{A posteriori prediction from noisy data with signal-to-noise ratio of 0.01. Green: training data. Black: whole data. Red: prediction from linear model. Shaded regions represents the uncertainty range of $\pm2$ standard deviations. Note that all of training, whole and predictions contain shaded region but the noise on training/whole data is too small to be observed. }
 \label{fig:noise_aposter}
\end{figure}

\textcolor{black}{
Next, we will analyze the robustness of the linear time delayed model with respect to noise in a more general sense.
Recall that the previous analysis on condition number in \Cref{sec:subsection_cond} with periodic assumptions  indicates robustness to noise  with increasing time delays. For a more stringent description of the robustness, we introduce the concept of \emph{pseudospectra}~\cite{trefethen1993hydrodynamic}. Here we define the $\epsilon$-pseudospectra of the block companion matrix $\mathbf{A}_{L}$ in \Cref{sec:toy_5modes} as $\Lambda_{\epsilon}$ in \Cref{eq:pseudo_spectra}.
}

\begin{equation}
\label{eq:pseudo_spectra}
    \Lambda_{\epsilon}(\mathbf{A}_{L}) = \{ z \in \mathbb{C}: \sigma_{\textrm{min}} (z\mathbf{I} - \mathbf{A}_{L}) \le \epsilon \},
\end{equation}
\textcolor{black}{
where $\sigma_{\textrm{min}}$ represents the minimal singular value. 
As shown in \Cref{fig:noise_spectra}, it is observed that the robustness of the solution decreases  the increasing $L$ and becomes most sensitive to noise at the noise-free optimal $L=9$, following which  the robustness improves as $L$ increases, which is consistent with previous analysis on condition number.
}

\begin{figure}[htbp]
 \centering
    \includegraphics[width=\linewidth]{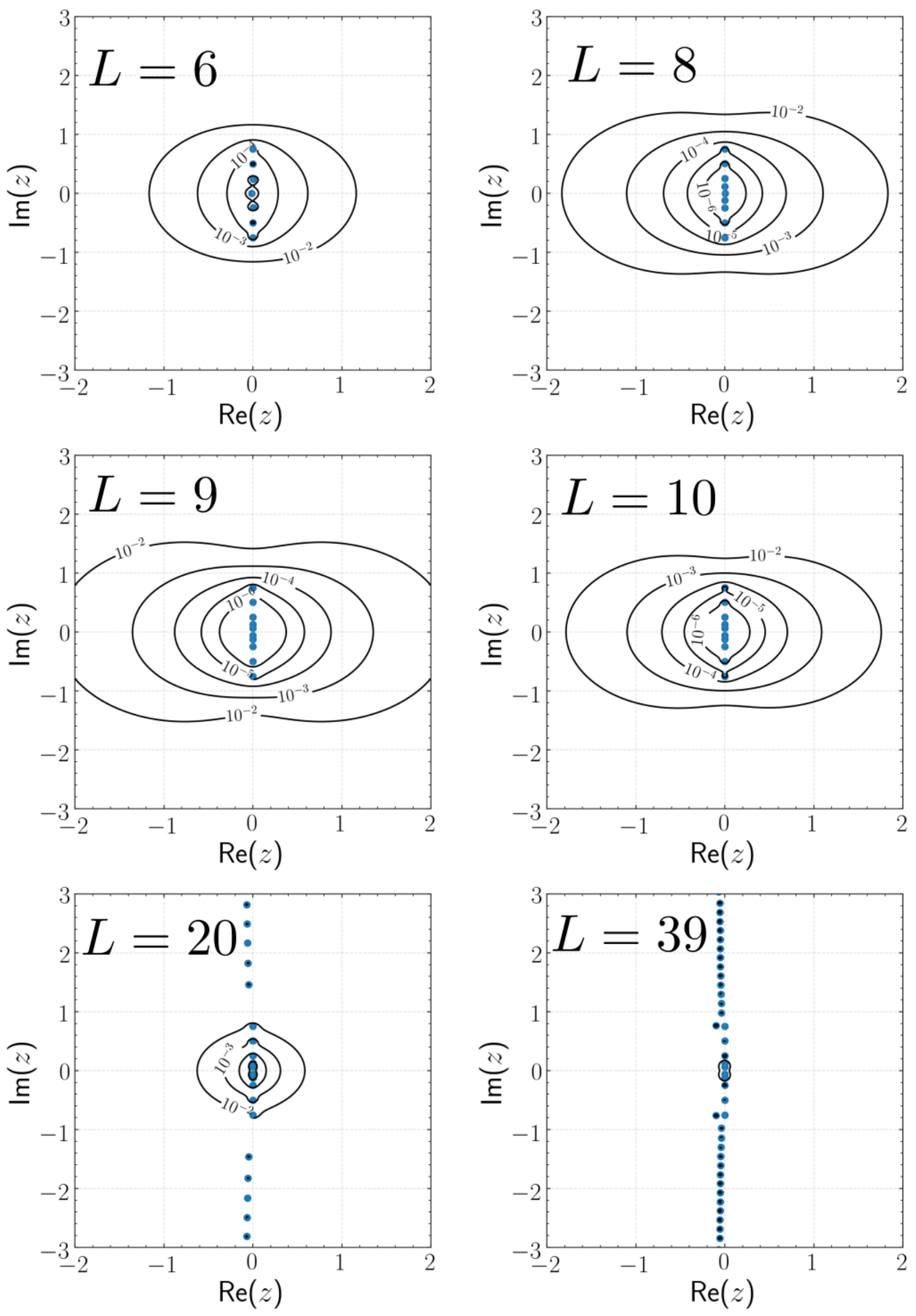}
 \caption{Isocontours of pseudospectra at $\epsilon=10^{-2}$, $10^{-3}$, $10^{-4}$, $10^{-5}$, $10^{-6}$ for different time delays $L$ for the toy 5 waves case.}
 \label{fig:noise_spectra}
\end{figure}

\subsection{Turbulent Rayleigh-B\'{e}nard convection}
\label{sec:rb}


As a final test case, we consider  Rayleigh-B\'{e}nard convection, which is a problem of great interest to the fluid dynamics community. As displayed in \Cref{fig:rb_svd}, the fluid is confined between two infinite horizontal planes with a hotter lower plane. The Rayleigh number, which represents the strength of buoyancy with respect to momentum and heat diffusion is defined as $Ra = {U_f^2 H^2}/{\nu \kappa} = {\alpha g \Delta T H^3}/{\nu \kappa}$ where $\alpha$ is the thermal expansion coefficient, $\kappa$ is the thermal diffusivity, $\Delta T$ is the temperature difference between hot and cold planes, and $U_f \triangleq \sqrt{\alpha g \Delta T H}$ is the so-called free-fall velocity of a fluid parcel. Additional parameters that govern the dynamics are aspect ratio $\Gamma \triangleq L/H$, the Prandtl number $Pr = \nu/ \kappa$. $L$ is the horizontal length scale of the domain. The computational domain is taken as a rectangular box with periodic side walls. We set $Ra = 10^7$ for fully turbulence; $H = \pi L_x = \pi L_y$ and $Pr = 1$. This domain is discretized uniformly in $x$ and $y$ direction with $128\times 128$ grid points and in $z$ direction with 128 grid points highly refined near the wall. The thickness of thermal boundary layer is sufficiently resolved~\cite{verzicco2003numerical} since  $\delta_{\theta}/H \sim {1}/{2Nu} \approx 10 \Delta z $, where $\Delta z$ is the grid size in $z$ direction closest to the wall.  

The simulation is performed by solving 3D incompressible Navier-Stokes equations with a Boussinesq approximation using  OpenFOAM~\cite{jasak2007openfoam}. Linear heat conduction, i.e., an unstable equilibrium state is set as initial condition. The simulation is performed over four thousand characteristic advection time units, approximately $ 1.264\tau_{\textrm{diff}}$, where $\tau_{\textrm{diff}}\triangleq {H^2}/{\nu}$, $\tau_{\textrm{adv}} \triangleq \sqrt{{H}/{\alpha g \Delta T}}$. The sampling interval is $\Delta t = 4\tau_{\textrm{adv}}$. Note that this  dynamical system contains approximately 2 million degrees of freedom. Here we perform dimension reduction on the sampled system state $u,v,w,T$ similar to~\cite{pan2020sparsity}. First, normalization for each component and mean subtraction is performed. Second, as shown in the bottom subfigure in the \Cref{fig:rb_svd}, more than 99\% of variance for the nonlinear system is retained in the first $r=800$ POD modes on the normalized data. After removing the effect of initial condition (the first 100 snapshots), we use  900 snapshots~\cite{2020_RB_data} for analysis.

\begin{figure}[htbp]
 \centering
    \includegraphics[width=\linewidth]{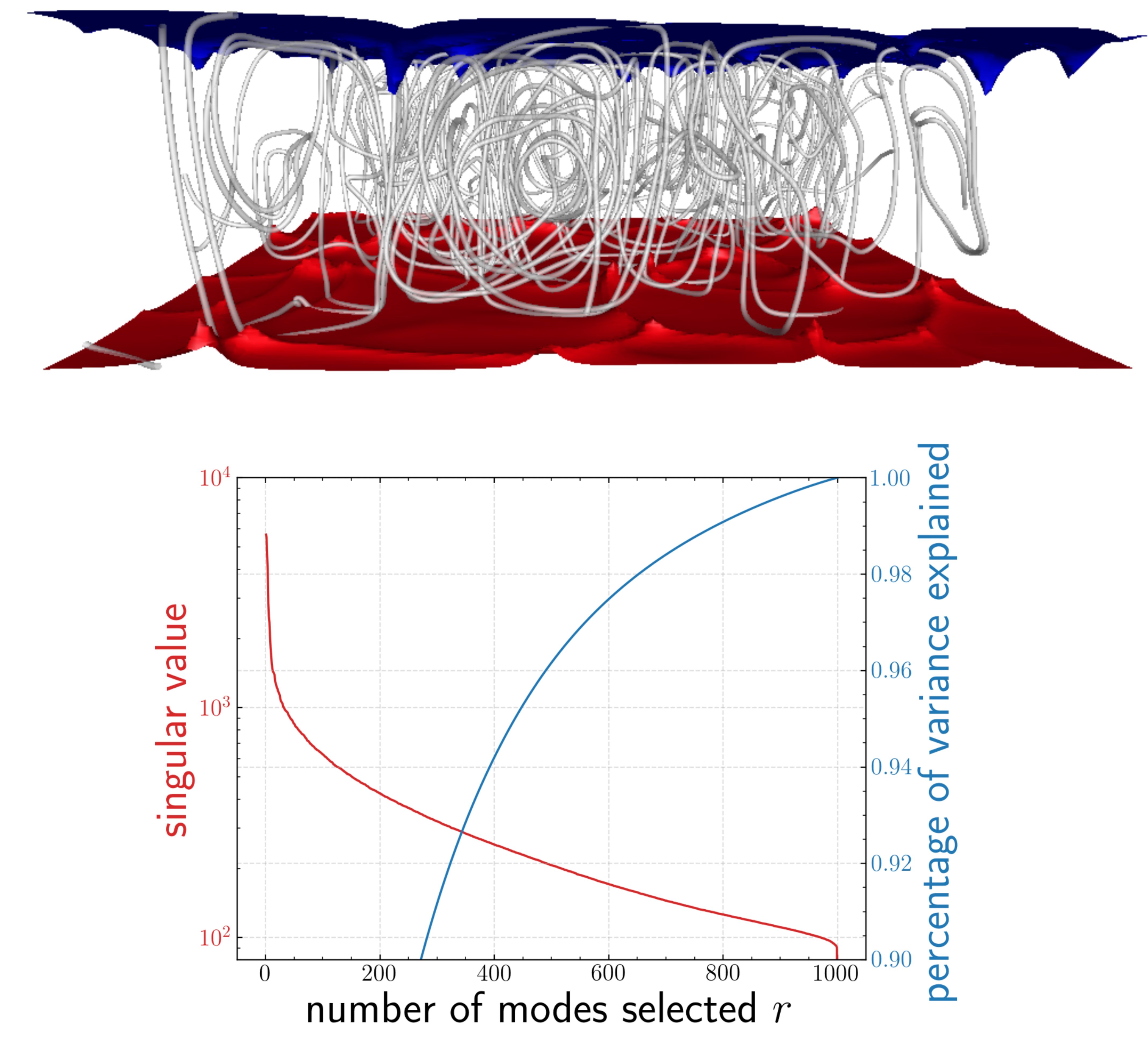}
 \caption{Top: Iso-surfaces of temperature at $T=295$ (red) and $T=285$ (blue) with streamlines of velocity field (grey) at $t=7.28$ for the  Rayleigh-B\'{e}nard turbulent convection at $Ra = 10^7$. Bottom: Singular value distribution and percentage of variance explained.}
 \label{fig:rb_svd}
\end{figure}

We consider the first 800 out of 900 snapshots as training data. Then we perform a posteriori evaluation for 900 steps to examine the reconstruction performance and predictions on future time steps.  
As shown in \Cref{fig:compare_with_svd_dmd_x1}, performing SVD-DMD ($L=0$) on this dataset with  $r=800$ results in a set of unstable eigenvalues,  leading to undesired blow up in a posteriori evaluation after $180\Delta t$. While the model with time delay $L=1$, overfits to the training data from $0$ to approximately $800\Delta t$, it yields  stable predictions. Note that in this case $L_{opt} = \lceil{\frac{M}{r+1}}\rceil = 1$. 

\begin{figure*}[htbp]
 \centering
    \includegraphics[width=0.7\linewidth]{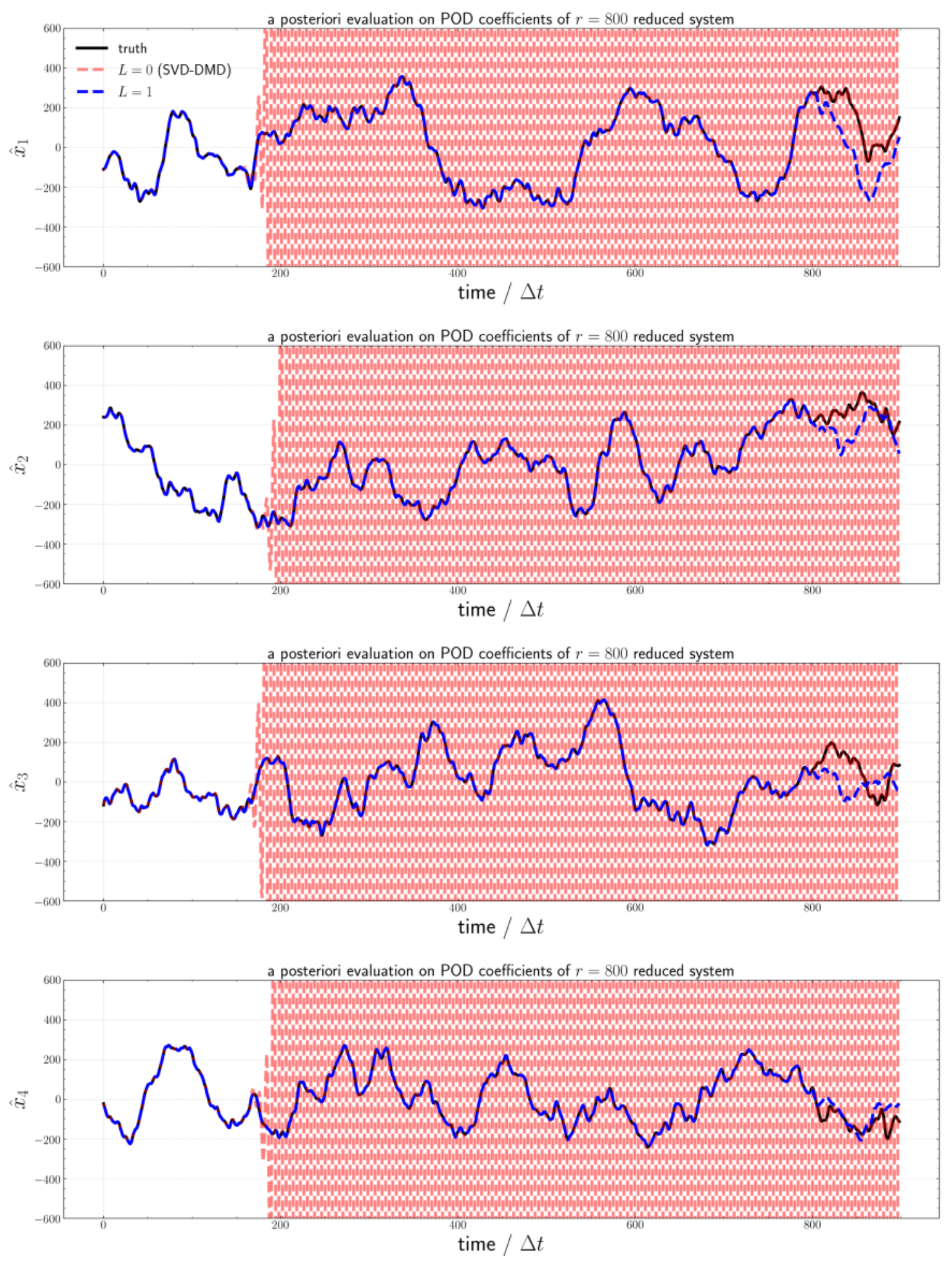}
 \caption{Comparison of a posteriori evaluation between linear model without/with time delay $L=1$ for the reduced system with $r=800$. Note that $0\le t \le 800$ is training horizon while $800 < t \le 900$ is testing horizon.}
 \label{fig:compare_with_svd_dmd_x1}
\end{figure*}

We then take the entire 900 snapshots trajectory as training data to investigate the impact of of time delays $L$ on stabilizing the reconstruction at various $r$. As shown in \Cref{fig:reconstruction_cond_mse_delay}, we first observe that as $r$ decreases, the numerical condition number increases simply as a consequence of retaining more small singular values. Secondly, we observe a general trend that, for each $r$, model performance worsens as $L$ increases from $0$ to $L_{opt}-1$, i.e., the transient point where linear systems approximately change from over-determined to under-determined. 
For the current data specifically, we observe that the system becomes stable as $L$ increases as the system becomes under-determined. Thirdly, we observe that the  condition number shares a similar pattern with the reconstruction performance for each $r$.

\begin{figure}[htbp]
 \centering
    \includegraphics[width=1\linewidth]{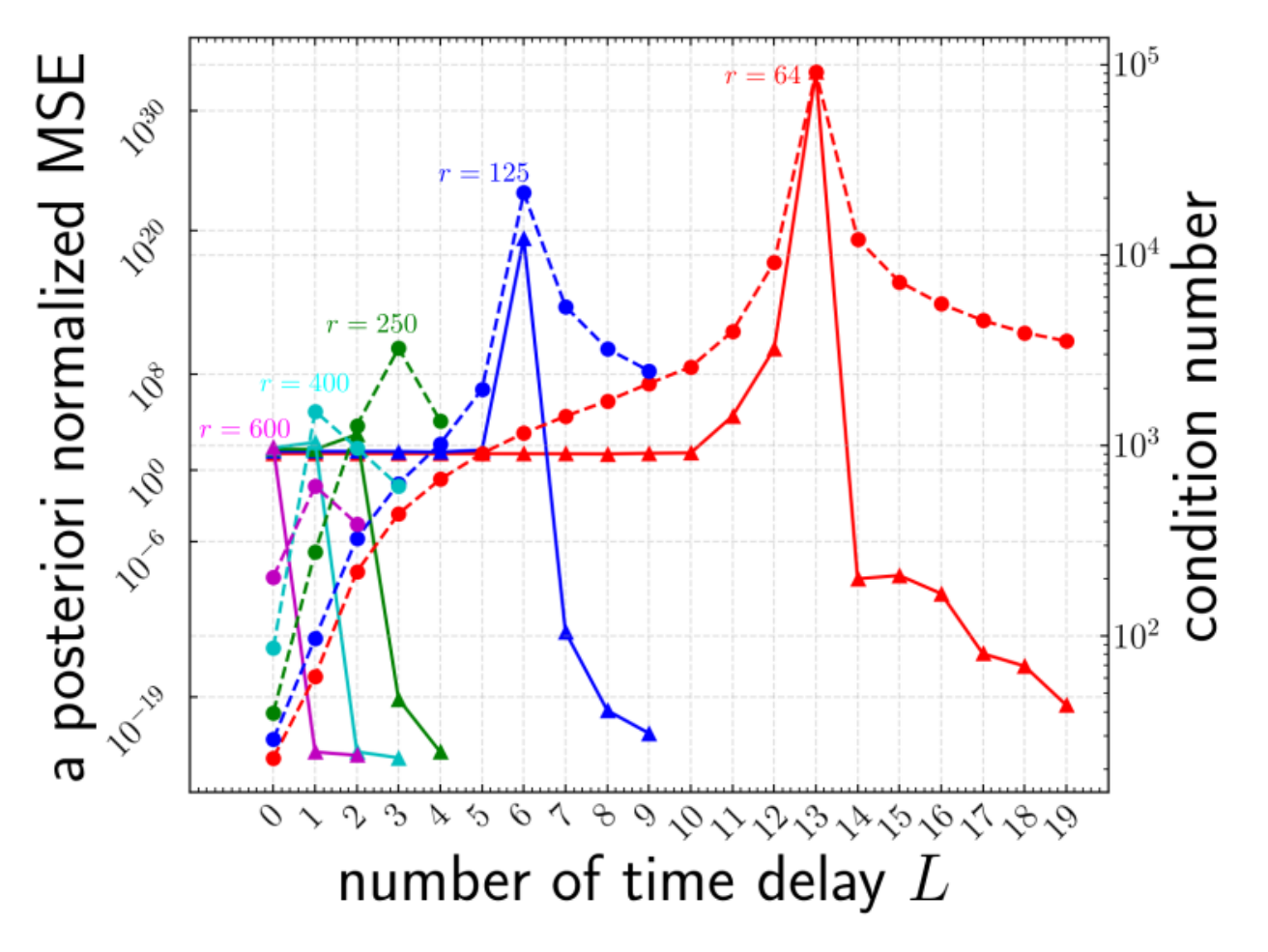}
 \caption{Dependency of model reconstruction performance and condition number on the number of time delays $L$ with varying reduced dimension $r$ for turbulent Rayleigh-B\'{e}nard convection. Solid line: normalized mean-squared-error. Dashed line: condition number.}
 \label{fig:reconstruction_cond_mse_delay}
\end{figure}

\section{Conclusions}
\label{sec:conclusions}

In summary, this work addressed fundamental questions regarding the structure and conditioning of linear time delay models of non-linear dynamics on an attractor. The following are the main contributions of this work:
\begin{enumerate}

    \item 
    We proved that for non-linear scalar dynamical systems, the number of time delays  required by linear models to perfectly recover limit cycles  is determined by the sparsity in the Fourier spectrum.
    
    \item In the vector case, we proved that the minimal number of time delays has a tight upper bound that is precisely the output controllability index of a related  linear system.  
    
    \item We developed an equivalent representation of the linear time delayed model in the spectral domain and provided the exact solution of the delay transition matrix $\mathbf K$ for the scalar case.
    
    \item We derived an upper bound on the 2-norm condition number as a function of the sampling rate and the number of time delays. Thus, ill-conditioning can be mitigated by increasing the number of time delays and/or subsampling the original signal.
    
    \item We explicitly showed that the dynamics over the full period can be perfectly recovered by training the linear time delayed model over just a partial period. 
    
    \item \textcolor{black}{Influences of the noises are evaluated with ensemble realizations. We further analyzed the stability of the model with the concept of pseudospectra. The results are consistent with our finding on the stabilizing role of the number of time delays.}
    
    \item Numerical experiments on simple problems were shown to confirm each of the above theoretical results. 
    
    \item The impact of time delays on linear modeling of large-scale chaotic systems  was investigated, and Hankel DMD was confirmed to produce stable and accurate results given enough time delays.
    
\end{enumerate}
    
\noindent A few observations are pertinent to  the above conclusions:
    \begin{itemize}
        \item Due to  accuracy considerations on the numerical integrator, the sampling rate in the raw data may be excessively high. We believe that instabilities in prediction arise from choices that lead to poor numerical conditioning. Thus, as an alternate to pursuing explicit stabilization techniques~\cite{le2017higher,champion2018discovery}, appropriate sub-sampling and time delays can  be employed. Indeed, when noise is present in the data, explicit stabilization, \textcolor{black}{Bayesian inference}, or denoising techniques~\cite{rudy2019deep} may be warranted.
        \item     
    The effectiveness of linear time delayed models of non-linear dynamics is that - by leveraging Fourier interpolation - an arbitrarily close trajectory from a high dimensional linear system can be derived. This also intuitively explains the ability of the model - when the signal has  a sparse spectrum - to perform ``true'' predictions  without training on a full period of data. 

    \end{itemize}

\begin{acknowledgments}
We would like to thank Mr. Nicholas Arnold-Medabalimi for visualizing and preparing the SVD of the Rayleigh-Bernard turbulence. This work was supported by DARPA under the grant titled {\em Physics Inspired Learning and Learning the Order and Structure Of Physics,} (Technical Monitor: Dr. Jim Gimlett), and US Air Force Office of Scientific Research through the
Center of Excellence Grant FA9550-17-1-0195 (Technical Monitors: Mitat Birkan \& Fariba Fahroo).
\end{acknowledgments}

\section*{Data availability}
The data that support the findings of this study are openly available in
\url{https://github.com/pswpswpsw/2020_Time_Delay_Paper_Rayleigh-Benard}

\appendix




\section{Proofs}

\subsection{Proof of \Cref{thm:sparse_time_delay}}
\label{apdx:sparse_time_delay_thm_proof}
\begin{proof}
Consider the discrete Fourier spectrum of $S_M(t)$ with $M$ uniform samples  per period. The perfect prediction using a time-delayed linear model requires the existence of a real $\mathbf{K}$ that satisfies \Cref{eq:X_K_Y_stacked}, which is equivalent to \Cref{eq:residual_R}. Therefore, \Cref{eq:X_K_Y_stacked} and \Cref{eq:residual_R} share the same solutions in $\mathbb{C}^{(L+1)\times 1}$. Since the Fourier spectrum contains only $P$ non-zero coefficients, \Cref{eq:residual_R} is equivalent to \Cref{eq:linear_system_sparse}.  The necessary and sufficient condition to have a solution (not necessarily real)  $\mathbf{K}$ for \Cref{eq:linear_system_sparse}  
 follows from the Rouch\'e-Capelli theorem~\cite{meyer2000matrix},
\begin{equation}
\label{eq:rank_equality}
\textrm{rank}\left(
\begin{bmatrix} \mathbf{A}_{\mathcal{I}^P_M,L}  & \mathbf{b}_{\mathcal{I}^P_M}  \end{bmatrix} \right) 
= 
\textrm{rank}\left(
\mathbf{A}_{\mathcal{I}^P_M,L} \right).
\end{equation}
Using the first property in \Cref{lem:vdm_rank}, rank$(\mathbf{A}_{\mathcal{I}_M^P,L}) = \min(P,L+1)$. While for the augmented matrix, 
\begin{align}
    & \textrm{rank}\left(
    \begin{bmatrix}
    \mathbf{A}_{\mathcal{I}^P_M,L} &\mathbf{b}_{\mathcal{I}^P_M}
    \end{bmatrix} 
    \right) 
    = 
    \textrm{rank}\left(
    \begin{bmatrix}
    \mathbf{b}_{\mathcal{I}^P_M} & \mathbf{A}_{\mathcal{I}^P_M,L}
    \end{bmatrix} 
    \right) \\
    \nonumber
    &= 
    \textrm{rank} \left(
    \begin{bmatrix}
    \omega^{i_0} & 1 & \omega^{-i_0} & \ldots  & \omega^{-Li_{0}} \\ 
    \omega^{i_1} & 1 & \omega^{-i_1} & \ldots & \omega^{-Li_1}\\ 
    \vdots &\vdots & \vdots & \ddots & \vdots \\ 
    \omega^{i_{P-1}} &1 &  \omega^{-i_{P-1}} & \ldots &  \omega^{-Li_{P-1}}
    \end{bmatrix} \right) \\
    &= \nonumber
    \textrm{rank} \left(
    \begin{bmatrix}
    \omega^{i_0} & & & & \\
    & \omega^{i_1} & & & \\
    & & & \ddots & \\
    & & & & \omega^{i_{P-1}}
    \end{bmatrix} 
    \begin{bmatrix}
    1 & \omega^{-i_0} & \ldots  & \omega^{-(L+1)i_0} \\ 
1 & \omega^{-i_1} & \ldots  & \omega^{-(L+1)i_1}\\ 
    \vdots & \vdots & \ddots & \vdots \\ 
    1 & \omega^{-i_{P-1}} & \ldots  & \omega^{-(L+1)i_{P-1}}
    \end{bmatrix} \right) \\
    \nonumber
    & \nonumber = \textrm{rank} \left(\diag( \omega^{i_0},\ldots,\omega^{i_{P-1}} ) \mathbf{V}_{L+2}( \omega^{-i_0},\ldots,\omega^{-i_{P-1}})\right)\\
    &= \nonumber \textrm{rank} \left( \mathbf{V}_{L+2}( \omega^{-i_0},\ldots,\omega^{-i_{P-1}})\right) \\
    &= \nonumber \min(P, L+2).
\end{align}
Therefore, if $L+2 \le P$, i.e., $L \le P - 2$, $\min(P,L+2) = L+2 \neq L+1 = \min(P,L+1)$. If $L+1 \ge P$, i.e., $L \ge P -1$, then $\min(P,L+2) = P = \min(P,L+1)$. So the minimal $L$ for \Cref{eq:rank_equality} to hold is $P-1$, which makes $\mathbf{A}_{\mathcal{I}_M^P, L}$ an invertible Vandermonde square matrix. Thus the solution is unique in $\mathbb{C}^{(L+1)\times 1}$. From \Cref{lem:complex_solution}, consider \Cref{eq:X_K_Y_stacked}, the solution is real.
\end{proof}

\subsection{Proof of \Cref{thm:control_vec_connection}}
\label{apdx:thm_control_vec_connection}
\begin{proof}

Consider
\begin{widetext}
\begin{align}
    & 
    \mathcal{OC}(\mathbf{A,B,C}; \mu) =
    \mathbf{C} \begin{bmatrix}
    \mathbf{B} & \mathbf{AB} & \ldots & \mathbf{A}^{\mu-1}\mathbf{B}
    \end{bmatrix} \\
    \nonumber &= \mathbf{C} 
    \begin{bmatrix}
    \mathbf{I} & \mathbf{A} & \ldots & \mathbf{A}^{\mu-1}
    \end{bmatrix}
    \begin{bmatrix}
    \mathbf{B} & & \\
    & \ddots & \\
    & & \mathbf{B}
    \end{bmatrix} \\
    &= 
    \nonumber
    \mathbf{EC}^\prime
    \begin{bmatrix}
    \mathbf{I} &  & &  & \mathbf{\Lambda}^{-(\mu-1)} & &\\
    & \ddots &  &  \ldots &  & \ddots & \\
    & &  \mathbf{I} & & &  & \mathbf{\Lambda}^{-(\mu-1)}
    \end{bmatrix}
    \begin{bmatrix}
    \mathbf{e} & & \\
    & \ddots & \\
    & & \mathbf{e}
    \end{bmatrix}\\
    \nonumber
    & = 
    \mathbf{E}
    \begin{bmatrix}
    \diag(\mathbf{{a}}^{(1)}) \mathbf{e} & 
    \ldots & 
    \diag(\mathbf{{a}}^{(J)}) \mathbf{e}
    & 
    \ldots 
    & 
    \diag(\mathbf{{a}}^{(1)}) \mathbf{\Lambda}^{-(\mu-1)
    }
    \mathbf{e} & 
    \ldots & 
    \diag(\mathbf{{a}}^{(J)}) \mathbf{\Lambda}^{-(\mu-1)
    }\mathbf{e}
    \end{bmatrix}.
\end{align}
\end{widetext}
Following \Cref{def:output_c}, for any integer $i \ge \mu$, $\mathcal{OC}(\mathbf{A,B,C}; i)$ is full rank. Thus, $\forall v \in \mathbb{C}^{P \times 1}$, $v$ lies in the column space of $\mathcal{OC}(\mathbf{A,B,C}; i)$. Therefore, $\mathbf{F}v$ should lie in the column space of $\mathbf{F} \mathcal{OC}(\mathbf{A,B,C}; i)$. Noticing \Cref{lem:EF_lemma} and \Cref{rem:geo}, we have 
\begin{equation}
    \mathbf{F}v \in \textrm{Col}(\mathbf{F} \mathcal{OC}(\mathbf{A,B,C}; i)) = \mathcal{W}_{i-1}.
\end{equation}
Now, consider $\forall j =1,\ldots,J$, $v^{(j)} = \mathbf{E} \diag(\mathbf{a}^{(j)}) \mathbf{b}_{\mathcal{I}_M^M} \in \mathcal{C}^{P \times 1}$, from the above, we have
\begin{equation}
    \mathbf{F}v^{(j)} = \mathbf{FE}  \diag(\mathbf{a}^{(j)}) \mathbf{b}_{\mathcal{I}_M^M} = \diag(\mathbf{a}^{(j)}) \mathbf{b}_{\mathcal{I}_M^M} = \mathbf{c}^{(j)} \in \mathcal{W}_{i-1}.
\end{equation} 
Since the minimal $i$ for $\mathcal{OC}(\mathbf{A,B,C}; i)$ to be full rank is $\mu$, the output observability index is $\mu$. Correspondingly, when the number of time delays $L = \mu - 1$, a solution exists for \Cref{eq:lem_vector_residual_R}, which makes $\mu-1$ an upper bound for the minimal time delay in \Cref{lem:vector_residual_R}. Finally, to show that the bounds are tight, consider that when $J=1$, \Cref{thm:control_vec_connection} reverts to  \Cref{thm:sparse_time_delay} where $\mu = P$, and thus $\mu-1=P-1$ is essentially the minimal number of time delays required.
\end{proof}

\subsection{Proof of \Cref{lem:vdm_rank}}
\label{apdx:vdm_rank_proof}
\begin{proof}
\begin{equation}
    \mathbf{A} = \mathbf{V}_N( \alpha_0, \alpha_1, \ldots, \alpha_{M-1}) = 
    \begin{bmatrix}
    1 & \alpha_0 & \ldots  & \alpha_0^{N-1}\\ 
    1 & \alpha_1 & \ldots & \alpha_1^{N-1}\\ 
    \vdots & \vdots & \ddots & \vdots \\ 
    1 &  \alpha_{M-1} & \ldots &  \alpha_{M-1}^{N-1}
    \end{bmatrix}
\end{equation}
If $M\ge N$, then
\begin{equation}
\mathbf{V}_N(\alpha_0, \alpha_1, \ldots, \alpha_{M-1}) = 
    \begin{bmatrix}
    \mathbf{V}_N(\alpha_0, \alpha_1, \ldots, \alpha_{N-1}) \\
    \mathbf{V}_N(\alpha_{N}, \ldots, \alpha_{M-1})
    \end{bmatrix}
\end{equation}
Since $\{\alpha_i\}_{i\in \mathcal{I}_M}$ are distinct, $\mathbf{V}_N(\alpha_0, \alpha_1, \ldots, \alpha_{N-1})$ is full rank with rank $N$. Since $M \ge N$, the row space of $\mathbf{V}_N(\alpha_0, \alpha_1, \ldots, \alpha_{M-1})$ and is fully spanned by the first $N$ rows, and is thus full rank. Likewise, if $M < N$, 
\begin{equation}
\mathbf{V}_N(\alpha_0, \alpha_1, \ldots, \alpha_{M-1}) = 
    \begin{bmatrix}
    \mathbf{V}_M(\alpha_0, \alpha_1, \ldots, \alpha_{M-1}) &
     * \quad
    \end{bmatrix}
\end{equation}
Similarly, the first $M$ columns are full rank  and $\mathbf{V}_N(\alpha_0, \alpha_1, \ldots, \alpha_{M-1})$ is also full rank. 
Thus in either case, $\mathbf{V}_N(\alpha_0, \alpha_1, \ldots, \alpha_{M-1})$ is full rank with rank as $\min(M,N)$. To show the the second property, one can simply replace $\{\alpha_i\}_{i\in \mathcal{I}_M}$ with $\{\alpha_i\}_{i\in \mathcal{J}}$ in the above arguments.
Since $|\mathcal{J}|=Q$, $
\rank\left(\mathbf{V}_N(\{ \alpha_i \}_{i \in \mathcal{J}})\right) 
= \min(Q,N)$.
\end{proof}

\subsection{Proof of \Cref{lem:complex_solution}}
\label{apdx:complex_solution_proof}
\begin{proof}
First, let's prove from left to right. If $\exists \mathbf{x} \in \mathbb{C}^{n\times 1}$, we have $\mathbf{Ax} = \mathbf{b}$. Note that $\widebar{\mathbf{Ax}} = \widebar{\mathbf{A}} \widebar{\mathbf{x}} = \mathbf{A}\widebar{\mathbf{x}} =  \widebar{\mathbf{b}} = \mathbf{b}$ then consider $\mathbf{x}^\prime = \frac{{\widebar{\mathbf{x}} + \mathbf{x}}}{2} \in \mathbb{R}^{n \times 1}$. $\mathbf{Ax^\prime} = (\mathbf{A}\mathbf{x} + \mathbf{A}\widebar{\mathbf{x}})/ 2 = (\mathbf{b} + \mathbf{b})/2 = \mathbf{b}$. Second, it is easy to show from right to left. Third, when uniqueness is added, note that $\mathbf{Ax=b} \iff \mathbf{A}\widebar{\mathbf{x}} = \mathbf{b}$, it is easy to show both directions since it is impossible to have complex solution being unique and not real. 
\end{proof}

\subsection{Proof of \Cref{lem:vector_residual_R}}
\label{apdx:vector_proof}
\begin{proof}
Given the definitions in \Cref{eq:Y_stacked_vector,eq:X_K_Y_vector,eq:X_stacked_vector}, note \Cref{eq:Y_k_stack}, we  have
\begin{equation}
    \label{eq:apdx_Yk}
    \mathbf{\tilde{Y}}_k = 
    \begin{bmatrix}
        \mathbf{\Omega}_{k,L} & & \\
        & \ddots & \\
        & & \mathbf{\Omega}_{k,L} 
    \end{bmatrix}
    \begin{bmatrix}
        \mathbf{a}^{(1)} \\
        \vdots \\
        \mathbf{a}^{(J)} 
    \end{bmatrix}. 
\end{equation}
Recall \Cref{eq:gamma_k}, note that 
\begin{equation}
    \label{eq:apdx_gamma_k}
    \mathbf{\Upsilon}_k = \mathbf{\Lambda}^k \mathbf{b}_{\mathcal{I}_M^M},
\end{equation}
where $\mathbf{\Lambda} \triangleq \begin{bmatrix}
1 & & & \\ 
& \omega & &  \\
& & \ddots &  \\
& & & \omega^{(M-1)}
\end{bmatrix}$.
    
Moreover, note that 
\begin{equation}
    \label{eq:apdx_omega_k}
    \mathbf{\Omega}^\top_{k,L} = \mathbf{\Lambda}^k \mathbf{A}_{\mathcal{I}^M_M,L}.
\end{equation}

We rewrite \Cref{eq:X_K_Y_vector} for a given $k$ using \Cref{eq:x_k_p_1} for the left hand side and \Cref{eq:apdx_Yk} for the right hand side in \Cref{eq:X_K_Y_vector}, 
\begin{equation}
    \begin{bmatrix}
        \mathbf{\Upsilon}_{k}^\top & & \\
        & \ddots & \\
        & & \mathbf{\Upsilon}_{k}^\top
    \end{bmatrix}    
    \begin{bmatrix}
        \mathbf{a}^{(1)} \\
        \vdots \\
        \mathbf{a}^{(J)} 
    \end{bmatrix}
    =
    \mathbf{\tilde{K}}^\top
    \begin{bmatrix}
        \mathbf{\Omega}_{k,L} & & \\
        & \ddots & \\
        & & \mathbf{\Omega}_{k,L} 
    \end{bmatrix}
    \begin{bmatrix}
        \mathbf{a}^{(1)} \\
        \vdots \\
        \mathbf{a}^{(J)} 
    \end{bmatrix}. 
\end{equation}

Using \Cref{eq:apdx_gamma_k,eq:apdx_omega_k} for the above, we  have
\begin{equation}
    \begin{bmatrix}
        \mathbf{a}^{(1)} \\
        \vdots \\
        \mathbf{a}^{(J)} 
    \end{bmatrix}^\top 
    \left(
        \begin{bmatrix}
        \mathbf{\Upsilon}_{k} & & \\
        & \ddots & \\
        & & \mathbf{\Upsilon}_{k} 
    \end{bmatrix} \right)
    - 
    \begin{bmatrix}
        \mathbf{\Omega}_{k,L}^\top & & \\
        & \ddots & \\
        & & \mathbf{\Omega}_{k,L}^\top 
    \end{bmatrix} \mathbf{\tilde{K}}
     = \mathbf{0},
\end{equation}
\begin{align}
\nonumber
    \begin{bmatrix}
        \mathbf{a}^{(1)} \\
        \vdots \\
        \mathbf{a}^{(J)} 
    \end{bmatrix}^\top
    \begin{bmatrix}
        \mathbf{\Lambda}^{k} & & \\
        & \ddots & \\
        & & \mathbf{\Lambda}^{k} 
    \end{bmatrix}
    & \Bigg(
        \begin{bmatrix}
        \mathbf{b}_{\mathcal{I}^M_M} & & \\
        & \ddots & \\
        & & \mathbf{b}_{\mathcal{I}^M_M} 
        \end{bmatrix}
        \\
        \label{eq:vector_proof_equation}
        & - 
        \begin{bmatrix}
        \mathbf{A}_{\mathcal{I}^M_M,L} & & \\
        & \ddots & \\
        & & \mathbf{A}_{\mathcal{I}^M_M,L}
    \end{bmatrix} \mathbf{\tilde{K}}
    \Bigg) = \mathbf{0}.
\end{align}

Considering $k = 0,1,\ldots, M-1$, we stack $\begin{bmatrix}
        \mathbf{a}^{(1)} \\
        \vdots \\
        \mathbf{a}^{(J)} 
    \end{bmatrix}^\top
    \begin{bmatrix}
        \mathbf{\Lambda}^{k} & & \\
        & \ddots & \\
        & & \mathbf{\Lambda}^{k} 
    \end{bmatrix}$
row by row as
\begin{align}
    & \nonumber
    \begin{bmatrix}
        a_0^{(1)} & \ldots & a_{M-1}^{(1)} & \ldots & a_0^{(J)} & \ldots & a_{M-1}^{(J)} \\
        a_0^{(1)} & \ldots & \omega^{M-1} a_{M-1}^{(1)} & \ldots & a_0^{(J)} & \ldots & \omega^{M-1} a_{M-1}^{(J)} \\
        \vdots & \ddots & \vdots & \ldots & \vdots & \ddots & \vdots \\
        a_0^{(1)} & \ldots & \omega^{(M-1)^2} a_{M-1}^{(1)} & \ldots & a_0^{(J)} & \ldots & \omega^{(M-1)^2} a_{M-1}^{(J)} 
    \end{bmatrix} \\
    & \nonumber = 
    \mathbf{V}_M(\{ \omega^{j} \}_{j=0}^{M-1})
    \begin{bmatrix}
        \mathbf{I} & \ldots & \mathbf{I} 
    \end{bmatrix}
    \diag(\{\mathbf{a}^{(l)}\}_{l=1}^{J}) \\
    &  = \mathbf{V}_M(\{ \omega^{j} \}_{j=0}^{M-1})
    \begin{bmatrix}
    \diag(\mathbf{a}^{(1)})  & 
    \ldots & 
    \diag(\mathbf{a}^{(J)})
    \end{bmatrix}.
\end{align}

Then plug the above equality into \Cref{eq:vector_proof_equation}, and notice the non-singularity of $\mathbf{V}_M(\{ \omega^{j} \}_{j=0}^{M-1})$, for $k = 0,1,\ldots,M-1$, \Cref{eq:vector_proof_equation} can be rewritten as 
\begin{align}
    \nonumber
    & \begin{bmatrix}
    \diag(\mathbf{a}^{(1)})  & 
    \ldots & 
    \diag(\mathbf{a}^{(J)})
    \end{bmatrix}
    \Bigg(
        \begin{bmatrix}
        \mathbf{b}_{\mathcal{I}^M_M} & & \\
        & \ddots & \\
        & & \mathbf{b}_{\mathcal{I}^M_M} 
        \end{bmatrix}
        \\
        \label{eq:apdx_vector_final}
        & - 
        \begin{bmatrix}
        \mathbf{A}_{\mathcal{I}^M_M,L} & & \\
        & \ddots & \\
        & & \mathbf{A}_{\mathcal{I}^M_M,L}
    \end{bmatrix} \mathbf{\tilde{K}}
    \Bigg) = \mathbf{0}.
\end{align}

From the Rouch\'e-Capelli theorem~\cite{meyer2000matrix}, the necessary and sufficient condition for the existence of   a complex solution to \Cref{eq:apdx_vector_final} is,
\begin{align}
     \nonumber
    \rank \Big(
    &
    \begin{bmatrix}
    \diag(\mathbf{a}^{(1)})\mathbf{A}_{\mathcal{I}^M_M,L}   & 
    \ldots & 
    \diag(\mathbf{a}^{(J)})\mathbf{A}_{\mathcal{I}^M_M,L} 
    \end{bmatrix}
    \Bigg)\\
     = 
     \label{eq:apdx_rouche}
    \rank \Big(
    \Big[
    & \diag(\mathbf{a}^{(1)})\mathbf{A}_{\mathcal{I}^M_M,L}   
    \ldots 
    \diag(\mathbf{a}^{(J)})\mathbf{A}_{\mathcal{I}^M_M,L} \\
    & \diag(\mathbf{a}^{(1)})\mathbf{b}_{\mathcal{I}^M_M}   
    \ldots 
    \diag(\mathbf{a}^{(J)})\mathbf{b}_{\mathcal{I}^M_M}
    \Big]
    \Big).
\end{align}
Note that since the above procedures are can be retained  in \Cref{eq:X_K_Y_vector}, \Cref{eq:X_K_Y_vector} and \Cref{eq:apdx_vector_final} share the same solution in $\mathbb{C}^{J(L+1) \times J}$. From \Cref{lem:complex_solution}, \Cref{eq:apdx_rouche} is also the necessary and sufficient condition for \Cref{eq:X_K_Y_vector} to have a real solution.
\end{proof}

\subsection{Proof of \Cref{lem:EF_lemma}}
\label{apdx:EF_proof}
\begin{proof}
For $n,J \in \mathbb{N}$, consider $J$ diagonal matrices in $\mathbf{A}$, for $j = 1,\ldots,J$, with the $j$-th diagonal matrices being $\diag(\mathbf{a}^{(j)}) \in \mathbb{C}^{n \times n}$. $\mathbf{a}^{(j)} = \begin{bmatrix}
\mathbf{a}^{(j)}_1 & \mathbf{a}^{(j)}_2 & \ldots \mathbf{a}^{(j)}_n
\end{bmatrix}^\top$. Thus 
\[
\mathbf{A} = 
\begin{bmatrix}
\diag(\mathbf{a}^{(1)}) & \diag(\mathbf{a}^{(2)}) & \ldots & \diag(\mathbf{a}^{(J)})
\end{bmatrix}
\in \mathbb{C}^{n\times nJ}.\] 

We define the following row index set that describes the row that is not a zero row vector in $\mathbf{A}$. 
\begin{equation}
    \Gamma = \{ l | l \in \{ 1,\ldots, n\}, \exists j \in \{1,\ldots, J\}, \mathbf{a}^{(j)}_l \neq 0 \},
\end{equation}
where we further order the index in $\Gamma$ as
\[1 \le \gamma_1 < \gamma_2 < \ldots < \gamma_{P} \le n,\]
where $P = |\Gamma|$. Now we construct the row elimination matrix $\mathbf{E} \in \mathbb{C}^{P \times n}$ from $\Gamma$ with
\begin{equation}
      i \in \{1,\ldots,P\}, j \in \{1,\ldots,n \}, \mathbf{E}_{ij} = \delta_{\gamma_i,j}.
\end{equation}
For $\mathbf{EA}$, since $\mathbf{E}$ only removes the zero row vector, the rank of the matrix $\mathbf{EA}$ is the same as $\mathbf{A}$. To show $\mathbf{EA}$ is full rank, simply consider the following  procedure: 

From  the definition of $\Gamma$, on each row with row index $i = 1,\ldots,P$, there are non-zero entries. Start by choosing an entry, denoted as $\mathbf{a}_{\gamma_i}^{j_i}$ that is non-zero (while the choice of $j_i$ is not unique). Then, one can simply perform column operations that switch the column with index ${j_i}$ corresponding to the non-zero entry of $i$-th row, with the current $i$-th column. These operations can be iteratively performed, after which the following matrix is obtained:
\begin{equation}
    \mathbf{EA}\mathbf{R} = \begin{bmatrix}
    \mathbf{a}_{\gamma_1}^{j_1} & & & & * & \\
    & \mathbf{a}_{\gamma_2}^{j_2} & & & * & \\
    & & \ddots & & * & \\
    & & & \mathbf{a}_{\gamma_P}^{j_P} &  * &
    \end{bmatrix},
\end{equation}
where $\forall i=1,\ldots,P, \mathbf{a}_{\gamma_i}^{j_i} \neq 0$ and $\mathbf{R}$ is the elementary column operation matrix. Thus  $\mathbf{EAR}$ is full rank, and $\mathbf{EA}$ is full rank.

Define $\mathbf{F} = \mathbf{E}^\top$, i.e., $\mathbf{F}_{jk} = \delta_{\gamma_k, j}$. Thus 
\begin{align}
\nonumber
    & i,j \in \{1,\ldots,n\}, \mathbf{G}_{ij} \triangleq \mathbf{F}_{ik} \mathbf{E}_{kj} = \delta_{\gamma_k, i} \delta_{\gamma_k, j} \\
    & = \sum_{k=1}^{P} \delta_{\gamma_k, i} \delta_{\gamma_k, j} = \begin{cases}
    1, \textrm{ $i=j \in \Gamma$, }\\
    0, \textrm{ otherwise. }
    \end{cases}
\end{align}
Therefore,  $\mathbf{G}$ is simply a diagonal matrix that keeps the row with index in $\Gamma$ unchanged, but makes the row zero when the  index is not in $\Gamma$. However, the row index that is not in $\Gamma$  corresponds to a zero row vector, and thus $\mathbf{GA}=\mathbf{A}$, i.e., $\mathbf{E}^\top \mathbf{EA}=\mathbf{A}.$
\end{proof}

\subsection{Proof of \Cref{lem:upper_bound}}
\label{apdx:proof_upper_bound}
\begin{proof}
For $q \in \mathbb{N}$, denote $L_q = qM + P - 1$. Note that in \Cref{eq:linear_system_sparse}, when $L=P-1$, the minimal 2-norm solution $\mathbf{\hat{K}}_{P-1}$ is also unique. Specifically we denote $\mathbf{\hat{K}}_{P-1} = \begin{bmatrix}
\hat{K}_0 & \ldots & \hat{K}_{P-1}
\end{bmatrix}$. Note that, for any $L \ge P -1$, we can find $q =  \left \lfloor{\frac{L - P + 1}{M}}\right \rfloor$, such that $L \in \mathcal{T}_q \triangleq [L_q, L_{q+1})$. From the definition of the minimal 2-norm solution, we have $\lVert \mathbf{\hat{K}}_L \rVert_2 \le \lVert \mathbf{\hat{K}}_{L_q} \rVert_2$.  

Consider $\mathbf{A}_{\mathcal{I}^P_M,L_q}$ and notice that for $q = 0$, i.e., $ L_0= P-1 \le L < L_1 = M + P-1$, so $\lVert \mathbf{\hat{K}}_L \rVert_2 \le \lVert \mathbf{\hat{K}}_{L_0} \rVert_2 = \lVert \mathbf{\hat{K}}_{P-1} \rVert_2$; for $q \ge 1$, for any $1 \le j \le P$, the $j$-th column of $\mathbf{A}_{\mathcal{I}^P_M,L_q}$ is duplicated with the $(j + kM)$-th column, $k = 1,\ldots,q$. For $q \ge 1$,  $\mathbf{A}_{\mathcal{I}^P_M,L_q}$ in  \Cref{eq:linear_system_sparse}, consider the following easily validated special class of real solutions, 
\begin{equation}
    \mathbf{K} = 
    \begin{bmatrix}
    K_0 \\ \vdots \\ K_{P-1} \\ 0 \\ \vdots  \\ 0 \\ K_{M} \\ \vdots \\ K_{L_1} \\ 0 \\ \vdots \\ 0 \\ \vdots \\ K_{qM} \\ \vdots \\ K_{L_q}
    \end{bmatrix}^\top
    \in \mathbb{R}^{1\times (L_{q+1})},
\end{equation}
with the constraint that for any $1 \le j \le P$, $\sum_{l=0}^{q} K_{j-1+lM} = {\hat{K}}_{j-1}$. To find the minimal 2-norm solution, note that we have
\begin{align}
    \min \lVert \mathbf{K} \rVert_2^2 = \sum_{j=1}^{P} \min \sum_{l=0}^{q} K^2_{j-1 + lM}.
\end{align}
From Jensen's inequality, $\forall j = 1,\ldots,P$,
\begin{align}
    \frac{\sum_{l=0}^{q} K^2_{j-1 + lM}}{q+1} &\ge \left( \frac{\sum_{l=0}^{q} K_{j-1+lM}}{q+1} \right)^2, \\
    \sum_{l=0}^{q} K^2_{j-1 + lM} &\ge \frac{\hat{K}_{j-1}^2}{q+1},
\end{align}
where the equality holds when $K_{j-1+lM} = {\hat{K}_{j-1}}/{(q+1)}$ for $l = 0,\ldots,q$. Thus $\min \lVert \mathbf{K} \rVert_2^2 = \sum_{j=1}^{P} \hat{K}_{j-1}^2/(q+1) = \lVert \mathbf{\hat{K}}_{P-1} \rVert_2^2/(q+1)$. Since the above minimal norm is found within a special class of solutions in \Cref{eq:linear_system_sparse},  the general minimal 2-norm is $$ \lVert \mathbf{\hat{K}}_{L} \rVert_2^2 \le \lVert \mathbf{\hat{K}}_{L_q} \rVert_2^2 \le  \lVert \mathbf{\hat{K}}_{P-1} \rVert_2^2/(q+1).$$ Combining both cases for $q = 0$ and $q \ge 1$, we have the desired result.
\end{proof}

\subsection{Proof of \Cref{prop:norm_fn}}
\label{apdx:prop_norm_fn}
\begin{proof}
To begin with, consider the following under-determined linear system for $f \in \mathbb{R}^N$, given $N \ge n$
\begin{equation}
    \mathbf{V}_N(z_1,\ldots,z_n) f = \diag(z_1,\ldots,z_n) \mathbf{e},
\end{equation}
 where $\mathbf{e} = \begin{bmatrix}1 & 1 & \ldots & 1 \end{bmatrix}^\top$.
Denote ${f}_N$ to be the minimum 2-norm solution. Suppose for all nodes, $i=1,\ldots,n$, $|z_i| \le 1$. Baz\'{a}n~\cite{bazan2000conditioning} showed that \begin{equation}
\label{eq:bazan_minimal_solution}
    \displaystyle \lim_{N \rightarrow +\infty} \lVert  f_N \rVert_2 = 0.
\end{equation}
Consider multiplying \Cref{eq:linear_system_sparse} on both sides from the left with $\diag(\omega^{Li_0},\ldots,\omega^{Li_{P-1}})$. Notice that the diagonal matrix is non-singular for any $L \in \mathbb{N}$, and the inverse of permutation matrix is its transpose. Then we have
\begin{align}
    &\begin{bmatrix}
        \omega^{Li_0} & \omega^{(L-1)i_0} & \ldots & 1 \\
        \vdots & \vdots & \vdots & \vdots \\
        \omega^{Li_{P-1}} & \omega^{(L-1)i_{P-1}} & \ldots & 1 \\
    \end{bmatrix} \mathbf{K} = 
    \begin{bmatrix}
        \omega^{(L+1)i_0}\\
        \vdots \\
        \omega^{(L+1)i_{P-1}}
    \end{bmatrix},\\
    &\begin{bmatrix}
        1 & \omega^{i_0} & \ldots & \omega^{Li_{0}} \\
        \vdots & \vdots & \vdots & \vdots \\
        1 & \omega^{i_{P-1}} & \ldots & \omega^{Li_{P-1}},\\
    \end{bmatrix} \mathbf{P}^\top \mathbf{K}
    = 
    \begin{bmatrix}
    \omega^{i_0} & & \\
    & \ddots & \\
    & &  \omega^{i_{P-1}}
    \end{bmatrix}^{L+1}
    \mathbf{e}, \\
    & \label{eq:prop_solution}
    \mathbf{V}_{L+1}(\omega^{i_0},\ldots,\omega^{i_{P-1}}) f = (\diag(\omega^{i_0},\ldots,\omega^{i_{P-1}}))^{L+1} \mathbf{e},
\end{align}
where $f \triangleq \mathbf{P}^\top \mathbf{K}$, $\mathbf{P} \in \mathbb{R}^{(L+1)\times(L+1)}$ is the column permutation matrix that reverses the column order in $\mathbf{A}_{\mathcal{I}^P_M,L}$. Note that a solution exists when $L+1 = P$ and it is not unique when $L+1 > P$. Denote $f_{L}$ as the corresponding minimal 2-norm solution of \Cref{eq:prop_solution}. From \Cref{eq:bazan_minimal_solution}, consider \Cref{eq:prop_solution} and take $L \rightarrow +\infty$, $\lVert f_L \rVert_2 \rightarrow 0$. The row permutation matrix does not change the 2-norm of a vector, and hence there is a one-to-one correspondence between the solution in \Cref{eq:prop_solution} and \Cref{eq:linear_system_sparse}, such that the corresponding minimal 2-norm solution for \Cref{eq:linear_system_sparse} is $\mathbf{\hat{K}}_L \triangleq \mathbf{P} f_L$ thus  $\lVert  \mathbf{\hat{K}}_L \rVert_2 \rightarrow 0$.
\end{proof}

\subsection{Proof of \Cref{prop:k2_cond_convergence}}
\label{apdx:k2_cond_convergence}

\begin{proof}
Consider the fact that the Vandermonde matrix $\mathbf{V}_N(z_1,\ldots,z_n)$ with  $n$ distinct nodes $\{z_i\}_{i=1}^{n}$, $z_i \in \mathbb{C}$ of order $N$, $N \ge n$, i.e., $\mathbf{V}_N$ is full rank. The Frobenius-norm condition number is defined as $\kappa_F(\mathbf{V}_N) \triangleq \lVert \mathbf{V}_N \rVert_F \lVert \mathbf{V}_N^{\dagger} \rVert_F$, where $\dagger$ represents Moore-Penrose pseudoinverse. Baz\'{a}n~\cite{bazan2000conditioning} showed that if $\forall i=1,\ldots,n$, with distinct $|z_i| \le 1$, $N \ge n$, then 
\begin{widetext}
\begin{equation}
    \label{eq:kf_vn}
    \kappa_F(\mathbf{V}_N) \le n\left[ 1 + \frac{(n-1) + \lVert f_N \rVert_2^2 + \prod_{i=1}^{n} |z_i|^2 - \sum_{i=1}^{n} |z_i|^2 }{(n-1) \delta^2} \right]^{\frac{n-1}{2}} \phi_{N}(\alpha, \beta),
\end{equation}
\end{widetext}
where $\displaystyle \delta \triangleq \min_{1 \le i < j \le n} |z_i - z_j|$, $\phi_N(\alpha,\beta) \triangleq \sqrt{\frac{1+\alpha^2 + \ldots + \alpha^{2(N-1)}}{1+\beta^2 + \ldots + \beta^{2(N-1)}}}$, $\displaystyle \alpha \triangleq \max_{1\le j \le n} |z_j|$, $\displaystyle \beta \triangleq \min_{1\le j \le n} |z_j|$.

The key to understand the behavior of the upper bound of $\kappa_2(\mathbf{V}_N)$, is to estimate the convergence rate of $\lVert f_N \rVert_2$ which is considered difficult for a general distribution of  nodes~\cite{bazan2000conditioning}. For the particular case of \Cref{eq:linear_system_sparse}, we can show a tight upper bound in \Cref{lem:upper_bound}. Thus,   $\forall 1 \le i \le n, |z_i| = 1$, \Cref{eq:kf_vn} becomes, 
\begin{align}
    \label{eq:kf_upperbound}
    \kappa_F(\mathbf{V}_N) &\le n \left( 1 + \frac{\lVert f_N \rVert_2^2}{(n-1) \delta^2} \right)^{\frac{n-1}{2}}.
\end{align}
Now we note a general inequality between the condition number in the 2-norm and in the Frobenius norm~\cite{bazan2000conditioning} by considering, 
\begin{align}
    &  n -2 <  n - 2 + \kappa_2(\mathbf{V}_N) + \kappa_2^{-1}(\mathbf{V}_N) \le \kappa_F(\mathbf{V}_N),
    \\
    \label{eq:inequality_bazan_condition}
    & \kappa_2(\mathbf{V}_N) \le \frac{1}{2} \left[ \kappa_F(\mathbf{V}_N) - n + 2 + \sqrt{(\kappa_F(\mathbf{V}_N) - n +2)^2 -4}  \right].
\end{align}
The right hand side in  \Cref{eq:inequality_bazan_condition} is monotonically increasing with respect to $\kappa_F(\mathbf{V}_N)$. Therefore using the upper bound from \Cref{eq:kf_upperbound} in \Cref{eq:inequality_bazan_condition}, and  some algebra we have the following upper bound, $\forall N > n$, 
\begin{equation}
    \kappa_2(\mathbf{V}_N) \le 1 + \frac{d}{2} \left[ 1 + \sqrt{1 + \frac{4}{d}} \right],
\end{equation}
where
\begin{equation}
    d \triangleq n \left[\left( 1 + \frac{\lVert f_N \rVert_2^2}{(n-1) \delta^2} \right)^{\frac{n-1}{2}} - 1\right].
\end{equation}
Finally, note that $d$  monotonically increases  with  $\lVert f_N \rVert_2$, and thus with $n=P$, $N = L+1$, $z_l = \omega^{-i_l}$, $l = 0,\ldots,P-1$ and \Cref{lem:upper_bound},  the desired upper bound is achieved. As $L \rightarrow \infty$, $\mathbf{\hat{K}}_{L} \rightarrow 0$ and $d\rightarrow 0$, and thus it is trivial to show that $\kappa_2(\mathbf{A}_{\mathcal{I}_M^P,L}) \rightarrow 1$.
\end{proof}

\nocite{*}
\bibliography{aipsamp}

\end{document}